\documentclass[reqno,a4paper]{amsart}
\usepackage[a4paper,top=3cm,bottom=3cm,inner=3cm,outer=3cm]{geometry}

\usepackage[utf8]{inputenc}
\usepackage[T1]{fontenc}
\usepackage[british]{babel}
\usepackage{xparse}
\usepackage{etoolbox}
\usepackage{amsfonts,amssymb,amsthm,amsmath}
\usepackage{calrsfs}
\usepackage{mathbbol}
\usepackage[all]{xy}
\usepackage{graphicx}
\usepackage{xcolor}
\usepackage{tikz-cd}
\usetikzlibrary{calc}
\usetikzlibrary{arrows.meta}
\usepackage{mathtools}
\usepackage{enumitem}
\usepackage{booktabs}
\usepackage{xspace}
\usepackage[colorlinks,citecolor=blue]{hyperref}

\allowdisplaybreaks
\numberwithin{equation}{section}

\theoremstyle{plain}
\newtheorem{theorem}[subsection]{Theorem}
\newtheorem{proposition}[subsection]{Proposition}
\newtheorem{lemma}[subsection]{Lemma}
\newtheorem{corollary}[subsection]{Corollary}

\theoremstyle{definition}
\newtheorem{definition}[subsection]{Definition}
\newtheorem{example}[subsection]{Example}
\newtheorem{remark}[subsection]{Remark}

\newtheorem{convention}[subsection]{Convention}

\newcommand{\noproof}{\hfill \qed}

\newlist{tfae}{enumerate}{1}
\setlist[tfae]{label=(\roman*), ref=(\roman*)}
\newenvironment{enum}{\begin{enumerate}[label=\((\)\hspace{0.12ex}\roman*\hspace{0.075ex}\()\)]}{\end{enumerate}}
\newenvironment{enumT}{\begin{enumerate}[label=\((\)\hspace{-0.1ex}\roman*\hspace{0.13ex}\()\)]}{\end{enumerate}}

\def\nameit#1{\textrm{#1}~}
\def\thex{\nameit{Theorem}}
\def\prox{\nameit{Proposition}}
\def\corx{\nameit{Corollary}}
\def\lemx{\nameit{Lemma}}
\def\defx{\nameit{Definition}}
\def\remx{\nameit{Remark}}
\def\exax{\nameit{Example}}
\def\dfn#1{{\itshape #1}}

\tikzset{iso/.style={draw=none,every to/.append style={edge node={node [sloped, allow upside down, auto=false]{\(\cong\)}}}}}
\tikzset{simeq/.style={draw=none,every to/.append style={edge node={node [sloped, allow upside down, auto=false]{\(\simeq\)}}}}}
\tikzset{simeqS/.style={draw=none,every to/.append style={edge node={node [sloped, allow upside down, auto=false]{\(\raisebox{0.8em}{\)\simeq\(}\)}}}}}
\tikzset{nmono/.style={arrows={Triangle[reversed, open] ->}}}
\tikzset{nepi/.style={arrows={-Triangle[open]}}}

\NewDocumentEnvironment{cd}{s O{7} O{7} b}{%
	\IfBooleanF{#1}{\begin{equation*}}\begin{tikzcd}[row sep=#2ex,column sep=#3ex,ampersand replacement=\&]
			#4
		\end{tikzcd}\IfBooleanF{#1}{\end{equation*}}\ignorespacesafterend}{}
\newenvironment{eqD}[1]{\begin{equation}\label{#1}}{\end{equation}\ignorespacesafterend}
\newenvironment{eqD*}{\begin{equation*}}{\end{equation*}\ignorespacesafterend}

\newdir{>>}{{}*!/3.5pt/:(1,-.2)@^{>}*!/3.5pt/:(1,+.2)@_{>}*!/7pt/:(1,-.2)@^{>}*!/7pt/:(1,+.2)@_{>}}
\newdir{ >>}{{}*!/8pt/@{|}*!/3.5pt/:(1,-.2)@^{>}*!/3.5pt/:(1,+.2)@_{>}}
\newdir{ |>}{{}*!/-3.5pt/@{|}*!/-8pt/:(1,-.2)@^{>}*!/-8pt/:(1,+.2)@_{>}}
\newdir{ >}{{}*!/-8pt/@{>}}
\newdir{>}{{}*:(1,-.2)@^{>}*:(1,+.2)@_{>}}
\newdir{<}{{}*:(1,+.2)@^{<}*:(1,-.2)@_{<}}

\newlength\horspace
\newcommand{\h}[1][1.0]{\hspace*{#1\horspace}}
\def\:{\colon}

\newcommand{\comp}{\circ}
\newcommand{\iso}{\cong}
\newcommand{\ov}[1]{\overline{#1}}
\def\phi{\varphi}

\DeclareFontFamily{OT1}{pzc}{}
\DeclareFontShape{OT1}{pzc}{m}{it}{<->s*[1.19]pzcmi7t}{}
\DeclareMathAlphabet{\mathpzc}{OT1}{pzc}{m}{it}
\newcommand{\catfont}[1]{\ensuremath{\mathpzc{#1}}\xspace}
\newcommand{\onecat}[1]{\ensuremath{\mathsf{#1}}\xspace}

\def\L{\catfont{L}}

\newcommand{\AbCat}{\catfont{AbCat}}
\newcommand{\Triang}{\catfont{Triang}}
\newcommand{\Sat}{\catfont{SatAbCat}}
\newcommand{\LAdj}{\catfont{LAdj}}
\newcommand{\Sup}{\catfont{Sup}}
\newcommand{\Supmod}{\ensuremath{\mathpzc{Sup}_{\mathrm{mod}}}}
\newcommand{\Suphil}{\ensuremath{\mathpzc{Sup}_{\mathrm{Hil}}}}

\renewcommand{\1}{\onecat{1}}
\newcommand{\Grp}{\onecat{Grp}}
\newcommand{\Ab}{\onecat{Ab}}
\newcommand{\Perf}{\onecat{Perf}}
\newcommand{\HypAb}{\onecat{HypAb}}
\newcommand{\Nil}{\onecat{Nil}}
\newcommand{\CMon}{\onecat{CMon}}
\newcommand{\SES}{\onecat{SES}}
\newcommand{\VectF}{\ensuremath{\mathsf{Vect}_{F}}}
\newcommand{\Mlc}{\onecat{Mlc}}
\newcommand{\Coh}{\onecat{Coh}}

\newcommand{\Scl}{\ensuremath{\mathrm{Sc}}}
\newcommand{\Lat}{\ensuremath{\mathrm{L}}}
\newcommand{\dseg}[1]{\ensuremath{{\downarrow}#1}}
\newcommand{\useg}[1]{\ensuremath{{\uparrow}#1}}
\newcommand{\trunc}{\ensuremath{\tau_{1}}}

\DeclareMathOperator{\Hom}{Hom}
\DeclareMathOperator{\End}{End}
\DeclareMathOperator{\Ext}{Ext}
\DeclareMathOperator{\Spec}{Spec}
\DeclareMathOperator{\Supp}{Supp}
\DeclareMathOperator{\Ass}{Ass}

\newcommand{\AS}{\textup{(AS)}}
\newcommand{\DPN}{\textup{(DPN)}}
\newcommand{\NEC}{\textup{(NEC)}}

\newcommand{\Ker}{\ensuremath{\mathrm{Ker}}}
\newcommand{\Coker}{\ensuremath{\mathrm{Cok}}}
\newcommand{\coker}{\ensuremath{\mathrm{coker}}}
\newcommand{\TwoKer}{\ensuremath{2\text{-}\mathrm{Ker}}}
\newcommand{\TwoCoker}{\ensuremath{2\text{-}\mathrm{Cok}}}
\newcommand{\twoker}{\ensuremath{2\text{-}\mathrm{ker}}}
\newcommand{\twocoker}{\ensuremath{2\text{-}\mathrm{coker}}}

\newcommand{\TwoImg}{\ensuremath{2\text{-}\mathrm{Im}}}
\newcommand{\twoimg}{\ensuremath{2\text{-}\mathrm{im}}}
\newcommand{\TwoCoim}{\ensuremath{2\text{-}\mathrm{Coim}}}
\newcommand{\twocoim}{\ensuremath{2\text{-}\mathrm{coim}}}
\newcommand{\Hker}[2]{\ensuremath{\mathrm{H}^{\ker}_{#1}(#2)}}
\newcommand{\Hcoker}[2]{\ensuremath{\mathrm{H}^{\coker}_{#1}(#2)}}
\newcommand{\HomC}[3]{{#1}\left({#2},\h[1]{#3}\right)}
\newcommand{\id}[1]{\operatorname{id}_{#1}}
\newcommand{\op}{\ensuremath{^{\operatorname{op}}}}

\newcommand{\too}{\longrightarrow}

\newcommand{\toe}{\twoheadrightarrow}
\makeatletter
\newcommand{\aar}[2][]{\xrightarrow[#1]{#2}}
\newcommand{\aR}[2][]{%
	\ext@arrow 0055{\Rightarrowfill@}{#1}{#2}}
\newcommand{\eqq}{\DOTSB\protect\Relbar\protect\joinrel\Relbar}
\def\xeqqfill@{\arrowfill@\Relbar\Relbar\eqq}
\newcommand{\aeqq}[2][]{%
	\ext@arrow 0099\xeqqfill@{#1}{#2}}
\makeatother

\begin{document}

\title{A two-categorical Snake Lemma}

\author{Elena Caviglia}
\author{Luca Mesiti}
\author{Tim Van~der Linden}

\email[E.~Caviglia]{elena.caviglia@outlook.com}
\email[L.~Mesiti]{luca.mesiti@outlook.com}
\email[T.~Van~der Linden]{tim.vanderlinden@uclouvain.be, tim.van.der.linden@vub.be}

\address[E.~Caviglia]{National Institute for Theoretical and Computational Sciences (NITheCS), Stellenbosch, South Africa}

\address[E.~Caviglia, L.~Mesiti]{Department of Mathematical Sciences, Stellenbosch University, South Africa}

\address[L.~Mesiti, T.~Van~der Linden]{Mathematics \& Data Science, Vrije Universiteit Brussel, Pleinlaan 2, B--1050 Brussel, Belgium}

\address[T.~Van~der Linden]{Institut de Recherche en Math\'ematique et Physique, Universit\'e catholique de Louvain, che\-min du cyclotron~2 bte~L7.01.02, B--1348 Louvain-la-Neuve, Belgium}

\thanks{The third author is a Research Director of the Fonds de la Recherche Scientifique--FNRS}

\date{}



\begin{abstract}
	We prove a Snake Lemma for 2-categories. Working in a 2-category with a strong bizero object, we develop short 2-exact sequences and normal image factorisations, and we prove a two-dimensional Normal Short Five Lemma. We then introduce the notion of a homologically self-dual 2-category, via a two-dimensional generalisation of the dinversion of an antinormal pair. We prove that this notion is equivalently characterised by the self-duality of homology, by a Pure Snake Lemma and by a Third Isomorphism Property. Two-dimensional di-exactness implies homological self-duality. Our main result is the Snake Lemma in a 2-di-exact 2-category: a ladder of 2-exact rows with normal verticals induces a 2-exact six-term sequence, with a connecting 1-cell that is 2-natural in the ladder. Di-exactness can be traded for two hypotheses that are not self-dual: that dinversion preserve normality, and that normal 2-epimorphisms compose. Everything specialises, on passing to a locally discrete 2-category, to its classical counterpart. We close by exhibiting three models: a 2-di-exact 2-category of abelian categories containing \(\Coh(X)\) for every Noetherian scheme \(X\); the locally ordered 2-category of complete modular lattices, in which 2-di-exactness amounts to Dedekind's transposition principle; and the 2-category of Hilbert lattices, which is not 2-di-exact but satisfies the non-self-dual hypotheses, by a theorem of Mackey on pairs of closed subspaces.
\end{abstract}

\maketitle

\section{Introduction}\label{S:Introduction}

The Snake Lemma is the engine of elementary homological algebra: from a ladder of short exact sequences it produces a six-term exact sequence relating the kernels and the cokernels of the three vertical maps, together with a connecting morphism that links the two halves. It holds in far greater generality than the abelian setting in which it is usually met---in Grandis's homological categories \cite{Grandis-HA1,Grandis-HA2}, and in di-exact categories \cite{Afsa-24,PVdL3}, where the hypotheses are weak enough to accommodate the non-additive examples of interest.

This paper lifts that theory one dimension. The motivation is that assembling together categories of one kind, often also far away from being abelian, yields 2-categories in which the two-dimensional analogues of kernels and cokernels are available and actually already studied in the literature without realising they were such. In the 2-category \(\AbCat\) of abelian categories, exact functors and natural transformations, the 2-kernel of an exact functor is the Serre subcategory it annihilates and the 2-cokernel is a Serre quotient; in the 2-category \(\Triang\) of triangulated categories the same roles are played by thick subcategories and by Verdier localisations. A Snake Lemma valid at this level therefore says something about localisation sequences of abelian and triangulated categories, and it does so uniformly.

Neither those examples nor the setting in which they are best seen originate here. The first and second authors, together with Z.~Janelidze and \"U.~Reimaa, introduce and investigate \emph{2-homological categories} \cite[Section 1]{CavJanMesRei26}, a two-dimensional generalisation of Grandis's pointed homological categories. Both \(\AbCat\) and \(\Triang\) are worked out there as examples, with the 2-kernels and the 2-cokernels identified as we have just described them \cite[Sections 2, 3 and 7]{CavJanMesRei26}, as is the 2-category of pointed categories \cite[Section 5]{CavJanMesRei26}. That paper also studies the \emph{1-truncation} of a 2-category, and which exactness properties are preserved and reflected by it \cite[Section 1]{CavJanMesRei26}. What comes out of that study is the expectation that the homological diagram lemmas should hold, in some form, in a 2-homological 2-category---but the theory there is still a long way from stating any of them in dimension two, let alone proving one. The present paper takes that step for the Snake Lemma, under hypotheses of a different kind: 2-di-exactness in Section~\ref{S:Snake}, and a non-self-dual alternative in Section~\ref{S:NonSelfDual}. Neither \(\AbCat\) nor \(\Triang\) satisfies the first, and that is proved there as well---for \(\AbCat\) in \cite[Section 2]{CavJanMesRei26} and for \(\Triang\) in \cite[Section 3]{CavJanMesRei26}, whence neither is 2-Puppe exact; \prox\ref{P:AbCatFails} obtains the failure for \(\AbCat\) independently, from a different counterexample and by a different route. What has no counterpart there is Section~\ref{S:Model}, which cuts \(\AbCat\) down to a sub-2-category in which the Snake Lemma is available. Their axioms and ours are compared in \remx\ref{Rem Grandis}: a 2-homological 2-category is exactly a homologically self-dual 2-category with 2-kernels and 2-cokernels in which normal 2-monomorphisms and normal 2-epimorphisms are closed under composition, so it carries the Pure Snake Lemma, and one further hypothesis separates it from the Snake Lemma itself.

Order enters categorical algebra by a different door in the work of Clementino, Montoli and Rodelo on categories enriched in preorders \cite{CleMonRod23,CleRod26}. There the limits are relaxed---comma objects stand in for pullbacks---and the algebra, protomodularity, is kept; here the limits are kept---2-kernels are bipullbacks, \prox\ref{Kernel vs pullback}---and the classes of 1-cells are relaxed instead. The 2-category of complete lattices of Section~\ref{SS:ModelLattices} is an instance of both, and \remx\ref{Rem Lattice Enriched} identifies its 2-kernels with their colax kernels.

The challenge in our investigation is that the classical proof of the Snake Lemma is a diagram chase, and diagram chases do not lift to dimension 2. Equalities between composites become invertible 2-cells; a factorisation through a universal property is determined only up to an invertible 2-cell, and one has to know that those 2-cells cohere; and a step as innocent as ``\(m\comp u=m\comp v\) and \(m\) is monic, so \(u=v\)'' has no two-dimensional counterpart unless \(m\) is asked to be \emph{fully} faithful rather than merely faithful. Our approach is to identify the correct two-dimensional notions and techniques upgrading the one-dimensional ones, so that the classical proofs can be upgraded to the two-dimensional context. We hope that our work can provide a dictionary between one-dimensional and two-dimensional notions for future investigations of 2-dimensional categorical algebra and homological algebra.

We assume familiarity with the basic language of 2-categories, for which \cite{Lac10} is a convenient reference. A 2-category is regarded as a generalised 1-category by way of \dfn{discretisation}: 1-categories are the locally discrete 2-categories, and each notion below is chosen so that its discretisation is the classical one. Every result therefore specialises to its classical counterpart, and the Snake Lemma of \thex\ref{Snake General 2D} specialises to the Snake Lemma.

\subsection{Main results}\label{SS:MainResults}

Throughout, the ambient 2-category has a strong bizero object: a bizero object such that between any two null 1-cells there is exactly one 2-cell. This condition is important in our work. It forces 2-kernels to be fully faithful and 2-cokernels cofully faithful, which fixes the choice of 2-monomorphism and 2-epimorphism; and it makes every invertible 2-cell into a null 1-cell unique (\lemx\ref{L:UniqueNull}), which removes an entire class of coherence obligations---among them the compatibility condition one is tempted to impose on a morphism of short 2-exact sequences, which turns out to be automatic (\prox\ref{P:CoherenceFree}).

As for notation, objects are written as capital letters, 1-cells as lowercase Latin letters and 2-cells as lowercase Greek letters. The one 1-cell outside this rule is the connecting 1-cell \(\partial\) of the Snake Lemma.

The results fall into five groups. The first is the elementary theory of Sections~\ref{S:Preliminaries} and~\ref{S:Normality}: 2-kernels and 2-cokernels and their calculus, normal 2-monomorphisms and normal 2-epimorphisms, short 2-exact sequences, the uniqueness of the normal image factorisation (\prox\ref{Image Factorisation of Normal Map is Unique}), and the Normal Short Five Lemma (\thex\ref{NSFL}).

The second group, in Section~\ref{S:Exact}, is about exactness and its self-duality. A pair of composable normal morphisms is exact when the 2-image of the first is the 2-kernel of the second, equivalently when the 2-cokernel of the first is the 2-coimage of the second, equivalently when their dinversion is null (\prox\ref{Exactness Self-Dual} and \corx\ref{Cor Exact Null Dinversion}). Asking this dinversion to be merely \emph{normal}, rather than null, and asking it of every antinormal decomposition of the zero map, is the definition of a \dfn{homologically self-dual} 2-category; exactness at a spot is thus the degenerate case of homological self-duality at that spot. We prove that homological self-duality is equivalent to the self-duality of homology, and to the Pure Snake Lemma, and to a Third Isomorphism Property (\prox\ref{Criteria HSD} and \prox\ref{Third Iso}).

The third group is given by Sections~\ref{S:Snake} and~\ref{S:Naturality}. A 2-category is 2-di-exact when it has all 2-kernels and 2-cokernels and every antinormal morphism is normal. Such a 2-category is homologically self-dual (\prox\ref{P:DiExactHSD}), so the whole of Section~\ref{S:Exact} is available to it; and in it the Snake Lemma holds (\thex\ref{Snake General 2D}). Section~\ref{S:Naturality} then isolates the notion of a morphism of pure configurations and proves that the comparison of the Pure Snake Lemma is 2-natural in it (\thex\ref{T:NaturalComparison}), from which the 2-naturality of the whole snake sequence follows (\thex\ref{T:NaturalSnake}).

The fourth group is Section~\ref{S:Model}, which asks what the theory is about. The obvious candidate for a two-dimensional model, the 2-category \(\AbCat\) of abelian categories, exact functors and natural transformations, is \emph{not} 2-di-exact: condition (DI2) reduces there to the question whether the essential image of a Serre subcategory in a Serre quotient is again a Serre subcategory, and the answer is no (\prox\ref{P:AbCatFails}). Taking that question as a hypothesis costs almost nothing, since it is inherited by both constructions (\prox\ref{P:SATclosed}), and the abelian categories satisfying it form a 2-di-exact 2-category (\thex\ref{T:SatModel}). It is populated: a condition on subobjects which is a statement about supports in disguise implies the hypothesis (\prox\ref{P:ASimplies}), and it holds for the finitely generated modules over a commutative Noetherian ring and for the coherent sheaves on a Noetherian scheme (\corx\ref{C:Populated}). The section closes with a model at the opposite end of the spectrum: the locally ordered 2-category of complete modular lattices is 2-di-exact (\thex\ref{T:LatticeModel}), condition \textup{(DI2)} being there exactly Dedekind's transposition principle; and the 2-category of all complete lattices repeats the profile of \(\AbCat\)---homologically self-dual, both composition properties, not \(\DPN\) with a five-element witness (\prox\ref{P:SupNotDPN}).

The fifth group is Section~\ref{S:NonSelfDual}, which asks how much of 2-di-exactness the Snake Lemma really consumes. The answer is: two applications of \textup{(DI2)}, and they are each other's dinversions (\prox\ref{P:TwoSites}). Replacing \textup{(DI2)} by the condition \(\DPN\) that dinversion preserve normality---which lies strictly between 2-di-exactness and homological self-duality---together with the closure of normal 2-epimorphisms under composition \(\NEC\), we give a second construction of the connecting 1-cell (\thex\ref{T:SnakeNonSelfDual}), modelled on the classical proof of \cite{Borceux-Bourn}. These hypotheses are not a weakening of \textup{(DI2)} but a trade: the two sets are incomparable, and already in dimension one (\remx\ref{Rem NSD Discrete}). This is where the one genuine use of bipullbacks outside Section~\ref{S:Bipullbacks} occurs. The same section closes the account of \(\AbCat\), showing we really have to restrict it as in Section~\ref{S:Model}: it is homologically self-dual, and normal 2-monomorphisms and normal 2-epimorphisms are closed under composition in it, but it does not satisfy \(\DPN\) (\prox\ref{P:AbCatNotDPN}), so the second set of hypotheses does not rescue it either. What does satisfy the second set of hypotheses is the 2-category of Hilbert lattices---lattices of closed subspaces of complex Hilbert spaces: condition \(\DPN\) holds there by Mackey's characterisation of the dual modular pairs of closed subspaces, while \textup{(DI2)} fails as soon as the dimension is infinite (\thex\ref{T:HilbertModel}), so that the two forms of the Snake Lemma come apart in dimension two; no lattice of finite length can separate them (\prox\ref{P:FiniteLength}).

\subsection{A formalisation}\label{SS:Formalisation}

Much of what follows has been machine-checked. A formalisation in Lean~4, built on the Mathlib library, accompanies this paper and is available at \url{https://github.com/tvdlinde/snake-lean}, archived at \url{https://doi.org/10.5281/zenodo.22668769}. It covers Sections~\ref{S:Preliminaries} to~\ref{S:Exact} in full, apart from the two examples of Section~\ref{SS:SES}, together with the results of Section~\ref{S:Bipullbacks} that the Snake Lemma consumes, \thex\ref{Snake General 2D} and \corx\ref{Snake General}, \thex\ref{T:SnakeNonSelfDual} in both of its forms, and, among the models, \thex\ref{T:LatticeModel} together with \prox\ref{P:SupModular}, \prox\ref{P:SupNotDPN} and \prox\ref{P:SupClass}, \prox\ref{P:HilbertSymmetric} and \prox\ref{P:FiniteLength}, and \prox\ref{P:ASimplies} with \prox\ref{P:ModAS}. Of \(\AbCat\) itself it checks \prox\ref{P:AbCatBizero} and \prox\ref{P:AbCatKernel}, the two results that precede Serre quotients. Of \prox\ref{P:SATclosed} it checks the half concerning Serre subcategories outright, and the half concerning Serre quotients relative to the correspondence between the Serre subcategories of a quotient \(\onecat{A}/S\) and those of \(\onecat{A}\) containing \(S\), which Mathlib does not yet provide; of \prox\ref{P:AbCatNotDPN} it checks the module-theoretic content, the two saturations of Section~\ref{SS:NSDAbCat}, but not the passage from them to \(\AbCat\). Sections~\ref{S:Naturality} and~\ref{S:Model} are otherwise covered in part, and the formalisation records, module by module, what it leaves out.

The formalisation is supplementary, and the reader who ignores it loses nothing: every result below is proved here, and no argument in this paper rests on the Lean development or on any claim about it. It has been useful in two ways. While the paper was being written, each result was first stated at the weakest assumptions under which its proof still compiled, and several of the hypotheses below are as weak as they are for that reason. The definitions of the final development are instead those of the paper, stated under the same hypotheses, so that the code may be read alongside the text and reused in further work.

\section{Preliminaries}\label{S:Preliminaries}

\subsection{Discretisation}\label{SS:Discretisation}

We say that a 2-category is \dfn{locally discrete} when all of its 2-cells are identities. We shall consider 1-categories as locally discrete 2-categories, so that concepts defined for 2-categories obtain a 1-categorical interpretation. This is what we mean when we say that we interpret a 2-categorical concept in a 1-category, and we call the concept thus obtained the \dfn{discretisation} of the original one.

The programme of this paper is to choose 2-dimensional versions of the concepts of a zero object, a kernel, a cokernel, a monomorphism, an epimorphism and a short exact sequence, whose discretisations are the classical ones, and which are strong enough that a result such as the Snake Lemma stays essentially the same when it is generalised from 1-categories to 2-categories. Of course, potentially many different 2-categorical concepts have the same discretisation; part of our work lies in making a coherent choice that works.

\subsection{Bizero objects and null morphisms}\label{SS:Bizero}

\begin{definition}[bizero object]\label{Def Bizero Object}
	Let \(\L\) be a 2-category. A \dfn{bizero object} in \(\L\) is an object \(0\in\L\) such that for every \(A\in\L\) both the hom-categories \(\HomC{\L}{A}{0}\) and \(\HomC{\L}{0}{A}\) are equivalent to the singleton category \(\1\). This means in particular that there exists a morphism \(A\to 0\) which is unique up to a unique isomorphism, and similarly a morphism \(0\to A\).

	We call \(\L\) a \dfn{bipointed} 2-category if it has a bizero object.
\end{definition}

In a 1-category, bizero objects are ordinary zero objects, that is, objects which are both initial and terminal.

\begin{remark}\label{Remark Null}
	Every bipointed 2-category has an associated canonical 2-ideal of null morphisms and null 2-cells, as explained in \cite[Example~2.7]{CavJanMes25}. This is given by taking as \dfn{null morphisms} those that factor through \(0\), and as null 2-cells between two null morphisms only the one given by the universal property of the bizero object as depicted below:
	\begin{eqD}{nulltwocell}
		\begin{cd}*
			{} \& {0} \arrow[dd, equal] \arrow[rd,"i", bend left=20]\& {}\\[-5ex]
			{A}  \arrow[r,iso]\arrow[ru,"t", bend left=20] \arrow[rd,"t'"', bend right=20] \& {} \arrow[r,iso] \& {B.}\\[-5ex]
			{} \& {0} \arrow[ru,"i'"', bend right=20] \& {}
		\end{cd}
	\end{eqD}
	We shall sometimes write \(0\colon A\to B\) for a chosen null morphism \(A\to 0\to B\); one exists by the universal property of the bizero object.
\end{remark}

\begin{definition}[strong bizero object]\label{Def Strong}
	Let \(\L\) be a 2-category. A bizero object \(0\) in \(\L\) is called \dfn{strong} if given any two parallel morphisms that factor through \(0\) there is a unique 2-cell between them, which is then equal to the isomorphism null 2-cell described in \remx\ref{Remark Null}.
	\begin{cd}[0]
		\& 0 \& \\
		A \&\& B \\
		\& 0
		\arrow[bend left =12, from=1-2, to=2-3]
		\arrow["{\exists ! \theta}", Rightarrow, from=1-2, to=3-2]
		\arrow[bend left =12, from=2-1, to=1-2]
		\arrow[bend right =12, from=2-1, to=3-2]
		\arrow[bend right =12, from=3-2, to=2-3]
	\end{cd}
\end{definition}

Having a strong bizero object simplifies the notions of 2-kernel and 2-cokernel introduced in \cite{CavJanMes25} into \defx\ref{Def 2-kernel} below, and makes several useful lemmas hold without further assumptions. In particular, 2-kernels become fully faithful arrows, and 2-cokernels satisfy the dual property. While some examples of bipointed 2-categories do not have a strong bizero object, such as the 2-category of 2-groups, the examples that concern us do. This motivates us to establish the theory presented in this paper in the case of a strong bizero object. There is no difference between a bizero object and a strong bizero object in a 1-category.

The following lemma is the precise form of the slogan that a strong bizero object guarantees that all 2-cells are null. Its point is that it does not ask \(x\) itself to be null.

\begin{lemma}\label{L:UniqueNull}
	Let \(\L\) be a 2-category with a strong bizero object, let \(n\colon A\to B\) be a null morphism and let \(x\colon A\to B\) be an arbitrary morphism. If some 2-cell \(\beta\colon x\aR{}n\) is invertible, then it is the only 2-cell \(x\aR{}n\).
\end{lemma}

\begin{proof}
	Since the bizero object is strong, there is exactly one 2-cell between any two parallel morphisms that factor through \(0\). Applied to \(n\) and \(n\) itself, this says that the only 2-cell \(n\aR{}n\) is \(\id{n}\). Let now \(\gamma\colon x\aR{}n\) be any 2-cell. Then \(\gamma\cdot\beta^{-1}\) is a 2-cell \(n\aR{}n\), so that \(\gamma\cdot\beta^{-1}=\id{n}\) and hence \(\gamma=\beta\).
\end{proof}

\begin{corollary}\label{C:UniqueNull}
	In a 2-category with a strong bizero object, a morphism admits at most one invertible 2-cell to a given null morphism. In particular, any two invertible 2-cells \(x\iso 0\) with the same codomain coincide.\noproof
\end{corollary}

The structure 2-cells that occur throughout this paper---the invertible 2-cells \(g\comp f\iso 0\) expressing that a composite vanishes, and the \(\kappa\) and \(\eta\) of a 2-kernel and of a 2-cokernel---are invertible 2-cells into a null morphism out of a morphism that is not itself null. \lemx\ref{L:UniqueNull} says that each of them is uniquely determined by its codomain, and every claim below that a 2-cell of this shape is ``the'' one induced by a universal property rests on it.

\begin{remark}\label{Rem Bizero Unique}
	Any two bizero objects in a 2-category are equivalent: if \(0\) and \(0'\) are both bizero, the essentially unique morphisms \(0\to 0'\) and \(0'\to 0\) are mutually inverse equivalences, since all hom-categories to and from a bizero object are equivalent to \(\1\). In particular, a morphism factors through \(0\) if and only if it factors through \(0'\), so the class of null morphisms is independent of the choice of bizero object. It follows that if \(0\) is a strong bizero object then so is every bizero object: a 2-category either has all its bizero objects strong, or none of them.
\end{remark}

\begin{definition}[trivial object]\label{Def Trivial}
	An object \(N\) of a bipointed 2-category is \dfn{trivial} when its identity morphism \(\id{N}\) is isomorphic to a null morphism.
\end{definition}

\begin{remark}\label{Rem Trivial}
	An object is trivial if and only if it is equivalent to a bizero object. Indeed, if \(\id{N}\iso i\comp t\) with \(t\colon N\to 0\) and \(i\colon 0\to N\), then the parallel morphisms \(t\comp i\) and \(\id{0}\) both factor through \(0\), so they are isomorphic, and \(t\) and \(i\) are mutually inverse equivalences; the converse is immediate. We take the condition on \(\id{N}\) as the definition because it is the form in which triviality is used: it turns every question of triviality into a question about a single morphism being null, and those are settled by the reflection properties of Section~\ref{SS:Monos}.
\end{remark}

\begin{definition}\label{Def Reflects Null}
	Let \(h\colon B\to C\) be a morphism in a 2-category that has a strong bizero object. We say that \(h\) \dfn{reflects null morphisms} if whenever a composite \(A\aar{f}B\aar{h}C\) is isomorphic to a null morphism, then also \(f\) is isomorphic to a null morphism. If \(h\) satisfies the dual condition, we say that \(h\) \dfn{coreflects null morphisms}.
\end{definition}

\subsection{2-monomorphisms and 2-epimorphisms}\label{SS:Monos}

It is not clear a priori what to consider as the right 2-dimensional generalisation of monomorphisms and epimorphisms. For example, in the 2-category of categories, both faithful functors and fully faithful functors could be chosen as a natural notion of 2-dimensional monomorphism. In \cite{DupVit03} and \cite[Definition~74]{Dup08}, 2-dimensional monomorphisms are defined to be the faithful arrows of \cite{Str82b}, so as to encompass more examples. In a category enriched in preorders every arrow is faithful, and the fully faithful ones are the \emph{ff-monomorphisms} of Clementino and Rodelo \cite[Section~1]{CleRod24}, so that on a locally ordered 2-category the definition below agrees with theirs. But in the setting of 2-categories with a strong bizero object, 2-kernels are automatically fully faithful arrows and 2-cokernels are cofully faithful arrows, as \prox\ref{Prop Kernel Is Mono} below records. This motivates the following definition, and the fullness condition will be used constantly.

\begin{definition}\label{Def Mono}
	Let \(\L\) be a 2-category with a strong bizero object, and let \(f\:A\to B\) be a morphism in \(\L\). We call \(f\) a \dfn{2-monomorphism} if it is a fully faithful arrow, that is, if for every object \(Z\) the functor \(f\comp(-)\colon\HomC{\L}{Z}{A}\to\HomC{\L}{Z}{B}\) is fully faithful. Dually, we call \(f\) a \dfn{2-epimorphism} if it is a cofully faithful arrow.
\end{definition}

\begin{remark}\label{rem2monoismonouptoiso}
	A 2-monomorphism is, in particular, a monomorphism up to isomorphism 2-cells. More precisely, let \(m\colon A\to B\) be a 2-monomorphism and let \(u\), \(v\colon M\to A\) be such that \(m\comp u\) is isomorphic to \(m\comp v\). Then \(u\) is isomorphic to \(v\), since fully faithful arrows lift isomorphism 2-cells to isomorphism 2-cells. 2-epimorphisms satisfy the dual condition.

	Note that even if we knew that \(m\comp u\) is precisely equal to \(m\comp v\), we could still only conclude that \(u\) is isomorphic to \(v\). As a corollary of the above, the discretisation of a 2-monomorphism is a monomorphism, and dually for 2-epimorphisms.
\end{remark}

Being a 2-monomorphism is a condition of representable shape. So is being an equivalence, and it will be convenient to have this recorded.

\begin{lemma}\label{L:RepEquiv}
	For a morphism \(q\colon A\to Q\) in a 2-category \(\L\), the following conditions are equivalent:
	\begin{tfae}
		\item \(q\) is an equivalence;
		\item for every object \(Z\), the functor \((-)\comp q\colon\HomC{\L}{Q}{Z}\to\HomC{\L}{A}{Z}\) is an equivalence of categories;
		\item for every object \(Z\), the functor \(q\comp(-)\colon\HomC{\L}{Z}{A}\to\HomC{\L}{Z}{Q}\) is an equivalence of categories.
	\end{tfae}
\end{lemma}

\begin{proof}
	If \(q\) is an equivalence with quasi-inverse \(p\), then whiskering with \(p\) provides a quasi-inverse for each of the two functors, so that (i) implies both (ii) and (iii).

	Assume (ii). Taking \(Z=A\), essential surjectivity of \((-)\comp q\) applied to \(\id{A}\) yields a morphism \(p\colon Q\to A\) together with an invertible 2-cell \(p\comp q\iso\id{A}\). Whiskering it with \(q\) on the left gives an invertible 2-cell
	\[
		q\comp p\comp q\iso q\iso\id{Q}\comp q\text{.}
	\]
	Taking \(Z=Q\), the functor \((-)\comp q\) is fully faithful, and a fully faithful functor lifts invertible 2-cells to invertible 2-cells; the displayed 2-cell is therefore reflected to an invertible 2-cell \(q\comp p\iso\id{Q}\). Hence \(q\) is an equivalence, and (ii) implies (i). The implication from (iii) follows by duality.
\end{proof}

In a 2-category with a strong bizero object, all 2-monomorphisms reflect null morphisms---not just those that arise as 2-kernels.

\begin{proposition}\label{prop2monoreflectsnull}
	In a 2-category with a strong bizero object, every 2-monomorphism reflects null morphisms. Dually, every 2-epimorphism coreflects null morphisms.
\end{proposition}
\begin{proof}
	We prove the first statement; the second one follows by duality. Let \(m\) be a 2-monomorphism and consider a composite \(A\aar{f}B\aar{m}C\) that is isomorphic to a null morphism, via a 2-cell
	\begin{cd}[5]
		\& B \& \\
		A \& 0 \& C
		\arrow["m", from=1-2, to=2-3]
		\arrow["f", from=2-1, to=1-2]
		\arrow[from=2-1, to=2-2, "a"']
		\arrow["\delta"{inner sep = 1ex}, shift left=5, iso, from=2-1, to=2-3]
		\arrow[from=2-2, to=2-3, "c"']
	\end{cd}
	We prove that then also \(f\) is isomorphic to a null morphism. Since \(0\) is a bizero object, there exists a morphism \(i\colon 0\to B\). Moreover, the composite \(0\aar{i}B\aar{m}C\) is isomorphic to \(0\aar{c}C\), by the universal property of the bizero object. Pasting the two isomorphism 2-cells that we have, we obtain an isomorphism 2-cell of the form
	\begin{cd}[4]
		\& B \&\& \\
		A \& 0 \&\& C \\
		\&\& B
		\arrow["m", from=1-2, to=2-4]
		\arrow["f", from=2-1, to=1-2]
		\arrow["a"', from=2-1, to=2-2]
		\arrow["\delta"{inner sep=1ex}, shift left=3.5,iso, from=2-1, to=2-4]
		\arrow["c"', from=2-2, to=2-4]
		\arrow[shift right=6, iso, from=2-2, to=2-4]
		\arrow["i"', from=2-2, to=3-3]
		\arrow["m"', from=3-3, to=2-4]
	\end{cd}
	Thus \(m\) equalises \(f\) and \(i\comp a\) up to an isomorphism 2-cell. By \remx\ref{rem2monoismonouptoiso}, we conclude that \(f\) is isomorphic to the null morphism \(i\comp a\).
\end{proof}

\subsection{2-kernels and 2-cokernels}\label{SS:Kernels}

\begin{definition}[2-kernel]\label{Def 2-kernel}
	Let \(\L\) be a 2-category with a strong bizero object \(0\). A \dfn{2-kernel} of a morphism \(A\aar{f}B\) in \(\L\) is a morphism \(K\aar{k}A\) together with an isomorphism 2-cell
	\begin{cd}*[4][4]
		K \arrow[rr,shift right= 1.7ex, iso, "\kappa"'{inner sep=-1ex, pos=0.65}] \& A \& B \\[-3ex]
		\& 0
		\arrow[from=1-1, to=1-2,"k"]
		\arrow[from=1-1, to=2-2, bend right=20]
		\arrow[from=1-2, to=1-3,"f"]
		\arrow[from=2-2, to=1-3, bend right=20]
	\end{cd}
	such that:
	\begin{itemize}
		\item[\((1)\)] for every morphism \(Z \aar{z} A\) and every isomorphism 2-cell
		      \begin{cd}[4][4]
			      Z \arrow[rr,shift right= 1.7ex, iso, "\beta"'{inner sep=-1ex, pos=0.65}] \& A \& B \\[-3ex]
			      \& 0
			      \arrow[from=1-1, to=1-2,"z"]
			      \arrow[from=1-1, to=2-2, bend right=20]
			      \arrow[from=1-2, to=1-3,"f"]
			      \arrow[from=2-2, to=1-3, bend right=20]
		      \end{cd}
		      there exist a morphism \(Z \aar{u} K\) and an isomorphism 2-cell
		      \begin{cd}[4][4]
			      Z \arrow[rr,shift right= 1.7ex, iso, "\gamma"'{inner sep=-1ex, pos=0.65}] \&  \& A \\[-3ex]
			      \& K
			      \arrow[from=1-1, to=1-3,"z"]
			      \arrow[from=1-1, to=2-2, bend right=20, "u"']
			      \arrow[from=2-2, to=1-3, bend right=20, "k"']
		      \end{cd}
		\item[\((2)\)] for all morphisms \(u\), \(v\colon Z \to K\) and every 2-cell
		      \begin{cd}*[4.5][4.5]
			      Z  \& K \arrow[d, Rightarrow, "\lambda", shorten <= -0.5ex, shorten <= -0.5ex] \& A \\[-3ex]
			      \& K
			      \arrow[from=1-1, to=1-2,"u"]
			      \arrow[from=1-1, to=2-2, bend right=20,"v"']
			      \arrow[from=1-2, to=1-3,"k"]
			      \arrow[from=2-2, to=1-3, bend right=20,"k"']
		      \end{cd}
		      there exists a unique 2-cell \(\mu\colon u \Rightarrow v\) such that \(k \star \mu = \lambda\).
	\end{itemize}
	Here \(\star\) denotes the whiskering of a morphism with a 2-cell.
\end{definition}

The definition of a 2-cokernel is dual. We record it for completeness.

\begin{definition}[2-cokernel]\label{Def 2-cokernel}
	Let \(\L\) be a 2-category with a strong bizero object \(0\). A \dfn{2-cokernel} of a morphism \(A\aar{f}B\) in \(\L\) is a morphism \(B\aar{q}Q\) together with an isomorphism 2-cell
	\begin{cd}*[4][4]
		A \arrow[rr,shift right= 1.7ex, iso, "\eta"'{inner sep=-1ex, pos=0.65}] \& B \& Q \\[-3ex]
		\& 0
		\arrow[from=1-1, to=1-2,"f"]
		\arrow[from=1-1, to=2-2, bend right=20]
		\arrow[from=1-2, to=1-3,"q"]
		\arrow[from=2-2, to=1-3, bend right=20]
	\end{cd}
	such that:
	\begin{itemize}
		\item[\((1)\)] for every morphism \(B \aar{z} Z\) and every isomorphism 2-cell
		      \begin{cd}[4][4]
			      A \arrow[rr,shift right= 1.7ex, iso, "\beta"'{inner sep=-1ex, pos=0.65}] \& B \& Z \\[-3ex]
			      \& 0
			      \arrow[from=1-1, to=1-2,"f"]
			      \arrow[from=1-1, to=2-2, bend right=20]
			      \arrow[from=1-2, to=1-3,"z"]
			      \arrow[from=2-2, to=1-3, bend right=20]
		      \end{cd}
		      there exist a morphism \(Q \aar{u} Z\) and an isomorphism 2-cell
		      \begin{cd}[4][4]
			      B \arrow[rr,shift right= 1.7ex, iso, "\gamma"'{inner sep=-1ex, pos=0.65}] \& \& Z \\[-3ex]
			      \& Q
			      \arrow[from=1-1, to=1-3,"z"]
			      \arrow[from=1-1, to=2-2, bend right=20, "q"']
			      \arrow[from=2-2, to=1-3, bend right=20, "u"']
		      \end{cd}
		\item[\((2)\)] for all morphisms \(u\), \(v\colon Q \to Z\) and every 2-cell \begin{cd}*[4.5][4.5]
			      B  \& Q \arrow[d, Rightarrow, "\lambda", shorten <= -0.5ex, shorten <= -0.5ex] \& Z \\[-3ex]
			      \& Q
			      \arrow[from=1-1, to=1-2,"q"]
			      \arrow[from=1-1, to=2-2, bend right=20,"q"']
			      \arrow[from=1-2, to=1-3,"u"]
			      \arrow[from=2-2, to=1-3, bend right=20,"v"']
		      \end{cd} there exists a unique 2-cell \(\mu\colon u \Rightarrow v\) such that \(\mu \star q = \lambda\).
	\end{itemize}
\end{definition}

\begin{remark}\label{Rem Kernel Property}
	A 2-kernel is presented above as a morphism \(k\) \emph{together with} an invertible 2-cell \(\kappa\colon f\comp k\iso 0\). By \lemx\ref{L:UniqueNull} there is at most one such 2-cell, so being a 2-kernel is a property of \(k\) rather than extra structure on it; dually for 2-cokernels. For the same reason, condition~(1) needs no clause relating \(\gamma\) to \(\beta\) and \(\kappa\): the compatibility described in \cite{CavJanMes25} is automatic thanks to \lemx\ref{L:UniqueNull}.
\end{remark}

\begin{proposition}\label{Prop Kernel Is Mono}
	In a 2-category with a strong bizero object, 2-kernels are 2-monomorphisms and 2-cokernels are 2-epimorphisms.
\end{proposition}
\begin{proof}
	Let \(k\colon K\to A\) be a 2-kernel. To say that \(k\) is a 2-monomorphism is to say that for every object \(Z\) the functor \(k\comp(-)\colon\HomC{\L}{Z}{K}\to\HomC{\L}{Z}{A}\) is fully faithful; that is, that for all \(u\), \(v\colon Z\to K\) whiskering with \(k\) is a bijection from the 2-cells \(u\Rightarrow v\) to the 2-cells \(k\comp u\Rightarrow k\comp v\). This is exactly condition~(2) of \defx\ref{Def 2-kernel}. Dually, condition~(2) of \defx\ref{Def 2-cokernel} shows that every 2-cokernel is cofully faithful.
\end{proof}

\begin{corollary}\label{corollkeruniqueuptoequiv}
	2-kernels and 2-cokernels are unique up to equivalence. Moreover, if \(f\), \(g\colon A\to B\) and there is an isomorphism 2-cell \(f\iso g\), then a 2-kernel of \(f\) is a 2-kernel of \(g\), and dually.
\end{corollary}
\begin{proof}
	Let \(k\colon K\to A\) and \(k'\colon K'\to A\) both be 2-kernels of \(f\). Applying condition~(1) of \defx\ref{Def 2-kernel} for \(k'\) to the morphism \(k\), which carries an invertible 2-cell \(f\comp k\iso 0\), produces \(u\colon K\to K'\) with \(k'\comp u\iso k\); symmetrically we obtain \(v\colon K'\to K\) with \(k\comp v\iso k'\). Then \(k\comp v\comp u\iso k'\comp u\iso k\iso k\comp\id{K}\), and \(k\) is a 2-monomorphism by \prox\ref{Prop Kernel Is Mono}, so \(v\comp u\iso\id{K}\) by \remx\ref{rem2monoismonouptoiso}; symmetrically \(u\comp v\iso\id{K'}\). Hence \(u\) is an equivalence. The second assertion holds because the defining data of a 2-kernel of \(f\) and of a 2-kernel of \(g\) are carried into one another by pasting with the given invertible 2-cell \(f\iso g\). The statements for 2-cokernels are dual.
\end{proof}

We write \(\Ker(f)\), \(\Coker(f)\) for the kernel and cokernel objects of a morphism \(f\) in a 1-category and \(\ker(f)\), \(\coker(f)\) for the corresponding morphisms, and \(\TwoKer(f)\), \(\TwoCoker(f)\), \(\twoker(f)\), \(\twocoker(f)\) for their 2-dimensional counterparts. By \corx\ref{corollkeruniqueuptoequiv} this notation is unambiguous up to equivalence.

\begin{remark}\label{Rem Ideal Closed}
	As explained in \cite[Definition~2.20 and Proposition~4.11]{CavJanMes25}, the canonical 2-ideal associated to a 2-category with a strong bizero object is closed in a 2-dimensional sense: 2-kernels reflect null morphisms and 2-cokernels coreflect them, in the sense of \defx\ref{Def Reflects Null}. By \prox\ref{Prop Kernel Is Mono} and \prox\ref{prop2monoreflectsnull} this holds more generally for every 2-monomorphism and every 2-epimorphism.
\end{remark}

Conversely to the fact that a 2-kernel is a 2-monomorphism, a 2-monomorphism has trivial 2-kernel.

\begin{lemma}\label{2-Mono Trivial Kernel}
	In a 2-category with a strong bizero object, a 2-monomorphism has trivial 2-kernel. Dually, a 2-epimorphism has trivial 2-cokernel.
\end{lemma}
\begin{proof}
	Let \(m\colon A\to B\) be a 2-monomorphism with 2-kernel \(n\colon N\to A\). The composite \(m\comp n\) is isomorphic to a null morphism, and \(m\) reflects null morphisms by \prox\ref{prop2monoreflectsnull}, so \(n\) is isomorphic to a null morphism. Then \(n\comp\id{N}=n\) is isomorphic to a null morphism, and \(n\), being a 2-kernel, is itself a 2-monomorphism and so reflects null morphisms as well; hence \(\id{N}\) is isomorphic to a null morphism, which is to say that \(N\) is trivial. The dual statement follows by duality.
\end{proof}

\section{2-kernels, 2-cokernels and normality}\label{S:Normality}

\subsection{2-z-exactness and normality}\label{SS:ZExact}

From now on we work in a 2-category with a strong bizero object in which all 2-kernels and all 2-cokernels exist. This generalises the notion of a z-exact category (see for example \cite{PVdL3}) to dimension 2.

\begin{definition}\label{Def ZExact}
	A \dfn{2-z-exact 2-category} is a 2-category with a strong bizero object that has all 2-kernels and all 2-cokernels.
\end{definition}

\begin{definition}\label{Def Normal Mono}
	In a 2-z-exact 2-category, a morphism which occurs as a 2-kernel of some morphism is called a \dfn{normal 2-monomorphism}. Dually, a morphism which occurs as a 2-cokernel of some morphism is called a \dfn{normal 2-epimorphism}.
\end{definition}

\begin{definition}\label{Def Normal}
	A morphism \(f\) is \dfn{normal} when it is isomorphic to a composite \(m\comp e\) in which \(m\) is a normal 2-monomorphism and \(e\) is a normal 2-epimorphism. It is \dfn{antinormal} when it is isomorphic to a composite \(e\comp m\) in which \(m\) is a normal 2-monomorphism and \(e\) is a normal 2-epimorphism.
\end{definition}

We ask throughout for a factorisation up to an invertible 2-cell rather than on the nose. Nothing below is sensitive to the difference, since a 2-kernel and a 2-cokernel depend on their morphism only up to isomorphism, by \corx\ref{corollkeruniqueuptoequiv}.

\subsection{Composition and cancellation}\label{SS:Normal}

\begin{proposition}\label{Composites of Normal Monos}
	Let \(f\colon A\to B\), \(g\colon B\to C\) and \(h\colon A\to C\) be morphisms in a 2-z-exact 2-category, and assume that \(h\) is isomorphic to \(g\comp f\). Then:
	\begin{enumT}
		\item if \(h\) is a 2-monomorphism and \(g\) is faithful, then \(f\) is a 2-monomorphism;
		\item if \(h\) is a normal 2-monomorphism and \(g\) is a 2-monomorphism, then \(f\) is a normal 2-monomorphism.
	\end{enumT}
	Dually, if \(h\) is a 2-epimorphism and \(f\) is cofaithful then \(g\) is a 2-epimorphism; and if \(h\) is a normal 2-epimorphism and \(f\) is a 2-epimorphism then \(g\) is a normal 2-epimorphism.
\end{proposition}
\begin{proof}
	We prove \((i)\). Consider a pair of morphisms \(u\), \(v\colon Q\to A\) and a 2-cell \(\lambda\colon f\comp u\aR{}f\comp v\). We must show that there is a unique 2-cell \(\xi\colon u\aR{}v\) with \(f\star\xi=\lambda\). Starting from \(\lambda\) we form the following pasting:
	\begin{cd}[2]
		\&[-2ex] A \&[-2ex]\& \\
		Q \&\& B \& C \\
		\& A
		\arrow["f", from=1-2, to=2-3]
		\arrow["h",bend left, from=1-2, to=2-4]
		\arrow["\lambda"', Rightarrow, from=1-2, to=3-2, shorten <=2ex, shorten >=2ex]
		\arrow["u", from=2-1, to=1-2]
		\arrow["v"', from=2-1, to=3-2]
		\arrow["g", from=2-3, to=2-4]
		\arrow[shift right=5,xshift=-2ex, iso, from=2-3, to=2-4]
		\arrow[shift left=5,xshift=-2ex, iso, from=2-3, to=2-4]
		\arrow["f"', from=3-2, to=2-3]
		\arrow["h"',bend right, from=3-2, to=2-4]
	\end{cd}
	Since \(h\) is a fully faithful arrow, there is a unique 2-cell \(\phi\colon u\aR{}v\) such that \(h\star\phi\) coincides with this pasting. Then \(f\star\phi=\lambda\): as \(g\) is faithful it suffices to check this after postcomposing with \(g\), and there it holds, by pasting with the given invertible 2-cell \(h\iso g\comp f\). Moreover \(\phi\) is the unique such 2-cell, since any two 2-cells \(\xi\) with \(f\star\xi=\lambda\) have equal whiskerings with \(h\), and \(h\) is faithful.

	We now prove \((ii)\). Say \(h\) is the 2-kernel of \(\ell\colon C\to D\); we prove that \(f\) is the 2-kernel of \(\ell\comp g\). The 2-dimensional universal property is \((i)\), a 2-monomorphism being in particular faithful. For the 1-dimensional one, note first that \(g\comp f\) is a 2-kernel of \(\ell\), by \corx\ref{corollkeruniqueuptoequiv}; in particular \(\ell\comp g\comp f\) is isomorphic to a null morphism. Let \(r\colon M\to B\) be such that \((\ell\comp g)\comp r\) is isomorphic to a null morphism. Since \(g\comp f\) is the 2-kernel of \(\ell\), there are \(s\colon M\to A\) and an invertible 2-cell \(\beta\colon g\comp f\comp s\iso g\comp r\); as \(g\) is fully faithful, \(\beta\) lifts to an invertible 2-cell \(f\comp s\iso r\). So \(f\) satisfies the universal property of the 2-kernel of \(\ell\comp g\).
\end{proof}

\begin{corollary}\label{C:NormalTransport}
	In a 2-z-exact 2-category, the composite of a normal 2-monomorphism with an equivalence, on either side, is again a normal 2-monomorphism; dually for normal 2-epimorphisms. Consequently a morphism isomorphic to \(v\comp w\comp u\), with \(u\) and \(v\) equivalences, is normal if and only if \(w\) is, and antinormal if and only if \(w\) is.
\end{corollary}
\begin{proof}
	Let \(k\colon K\to A\) be a 2-kernel of \(f\colon A\to B\), let \(u\colon K'\to K\) and \(w\colon A\to A'\) be equivalences and let \(w'\) be a quasi-inverse of \(w\). We claim that \(w\comp k\comp u\) is a 2-kernel of \(f\comp w'\). It is a 2-monomorphism, being a composite of 2-monomorphisms, so condition~\((2)\) of \defx\ref{Def 2-kernel} holds; and \(f\comp w'\comp w\comp k\comp u\iso f\comp k\comp u\) is null. If \(z\colon Z\to A'\) satisfies \(f\comp w'\comp z\iso 0\), then the universal property of \(k\) yields \(s\colon Z\to K\) with \(w'\comp z\iso k\comp s\), whence \(z\iso w\comp k\comp s\iso(w\comp k\comp u)\comp(u'\comp s)\) for a quasi-inverse \(u'\) of \(u\). This proves the first assertion, since every equivalence is invertible up to isomorphism on either side; the second is dual.

	For the consequence, let \(w\iso m\comp e\) be normal. Then \(v\comp w\comp u\iso(v\comp m)\comp(e\comp u)\) is a normal 2-monomorphism after a normal 2-epimorphism, hence normal; and the converse follows by composing with quasi-inverses. The same computation applies with the factors interchanged, which is antinormality.
\end{proof}

\begin{proposition}\label{CoKernel of Composite}
	In a 2-category with a strong bizero object, if a morphism \(f\) is isomorphic to a composite \(m\comp g\) with \(m\) a 2-monomorphism, then \(\twoker(f)\simeq\twoker(g)\). Dually, if \(f\) is isomorphic to a composite \(g'\comp e\) with \(e\) a 2-epimorphism, then \(\twocoker(f)\simeq\twocoker(g')\).

	In particular, if \(f\) is isomorphic to a composite \(m\comp e\) in which \(m\) is a 2-monomorphism and \(e\) is a 2-epimorphism, then \(\twocoker(f)\simeq\twocoker(m)\) and \(\twoker(f)\simeq\twoker(e)\).
\end{proposition}
\begin{proof}
	Say \(f\colon A\to B\), \(g\colon A\to Q\) and \(m\colon Q\to B\). Consider the 2-kernel \(K\aar{k}A\) of \(g\). Certainly \(f\comp k\) is isomorphic to a null morphism, since \(g\comp k\) is. Let \(r\colon M\to A\) be a morphism such that \(f\comp r\) is isomorphic to a null morphism. Then so is \(m\comp g\comp r\), and \(m\) reflects null morphisms by \prox\ref{prop2monoreflectsnull}, so \(g\comp r\) is isomorphic to a null morphism as well. The universal property of the 2-kernel of \(g\) now yields \(s\colon M\to K\) and an invertible 2-cell \(r\iso k\comp s\). Since \(k=\twoker(g)\) is a 2-monomorphism, this exhibits \(k\) as a 2-kernel of \(f\).

	Conversely, consider the 2-kernel \(H\aar{h}A\) of \(f\). Since \(f\comp h\) is isomorphic to a null morphism, so is \(m\comp g\comp h\), and hence so is \(g\comp h\), again by \prox\ref{prop2monoreflectsnull}. That \(h\) is a 2-kernel of \(g\) now follows in the same way, the 2-dimensional universal property being automatic since \(h\) is a 2-kernel of \(f\). The second part of the statement is dual.
\end{proof}

Note that \prox\ref{CoKernel of Composite} asks of \(m\) and \(e\) only that they be a 2-monomorphism and a 2-epimorphism; neither has to be normal. This is what the general case of the Snake Lemma will need in Section~\ref{SS:GeneralCase}.

\subsection{Short 2-exact sequences}\label{SS:SES}

\begin{proposition}\label{kernel is kernel of its cokernel}
	In a 2-z-exact 2-category, any normal 2-epimorphism is a 2-cokernel of its 2-kernel; any normal 2-monomorphism is a 2-kernel of its 2-cokernel.
\end{proposition}
\begin{proof}
	We prove the first statement; the second follows by duality. Let \(q\colon B\to Q\) be a 2-cokernel of a morphism \(f\colon A\to B\), with invertible 2-cell \(\eta\colon q\comp f\iso 0\), and let \((k\colon K\to B,\ \kappa\colon q\comp k\iso 0)\) be the 2-kernel of \(q\). We verify that \((q,\kappa)\) is a 2-cokernel of \(k\).

	Condition~(1) of \defx\ref{Def 2-kernel}, applied to the datum \((f,\eta)\), provides a morphism \(w\colon A\to K\) and an invertible 2-cell \(\gamma\colon k\comp w\iso f\). Let now \(z\colon B\to Z\) be a morphism with an invertible 2-cell \(\beta\colon z\comp k\iso 0\). Then \((\beta\star w)\cdot(z\star\gamma^{-1})\colon z\comp f\iso 0\) is invertible, so condition~(1) of \defx\ref{Def 2-cokernel} applied to \(q\) yields a morphism \(u\colon Q\to Z\) and an invertible 2-cell \(u\comp q\iso z\). This is condition~(1) for the 2-cokernel of \(k\). Condition~(2) for the 2-cokernel of \(k\) is literally condition~(2) of \(q=\twocoker(f)\). Hence \(q\) is a 2-cokernel of \(k\).
\end{proof}

\begin{remark}\label{remchoiceexhibit2cells}
	In the proof above, the 2-cell \(\kappa\) that exhibits \(k\) as the 2-kernel of \(q\) is the very one that exhibits \(q\) as the 2-cokernel of \(k\). This is no surprise, as by \lemx\ref{L:UniqueNull} we know it is always possible to choose this, up to changing the null morphism and the isomorphism 2-cell that exhibit the 2-cokernel.

	For the same reason, in the following definition, we can avoid imposing coherence conditions on the isomorphism 2-cells exhibiting the 2-kernel and the 2-cokernel.
\end{remark}

\begin{definition}\label{Def:SES}
	A \dfn{short 2-exact sequence} is a pair of composable morphisms \(k\colon K\to A\) and \(q\colon A\to Q\) such that \(k\) is a 2-kernel of \(q\) and \(q\) is a 2-cokernel of \(k\). We write \(\kappa\colon q\comp k\iso 0\) for the invertible 2-cell exhibiting both (we can always choose this thanks to \remx\ref{remchoiceexhibit2cells}), and picture the situation as
	\begin{equation}\label{SES}
		\xymatrix{0 \ar[r] & K \ar[r]^-{k} & A \ar[r]^-{q} & Q \ar[r] & 0\text{.}}
	\end{equation}
\end{definition}

\begin{example}\label{Ex SES}
	By \prox\ref{kernel is kernel of its cokernel}, if \(m\colon A\to B\) is a normal 2-monomorphism with 2-cokernel \(\twocoker(m)\colon B\to C\), then \(A\aar{m}B\aar{\twocoker(m)}C\) is a short 2-exact sequence; dually for a normal 2-epimorphism and its 2-kernel.
\end{example}

The two universal properties in \defx\ref{Def:SES} are independent demands, and neither follows from the other. The following group-theoretical example exhibits a 2-epimorphism which has exactly the 2-kernel one expects of it and is nonetheless not a 2-cokernel of that 2-kernel.

\begin{example}[hypoabelianisation]\label{Ex Hypoabelianisation}
	Let \(\LAdj\) be the 2-category whose objects are the pointed finitely cocomplete categories, whose 1-cells are the left adjoints between them and whose 2-cells are the natural transformations. The trivial category is a strong bizero object in \(\LAdj\): a null 1-cell is a functor constant at the zero object, and between two such there is exactly one natural transformation, the zero object having a single endomorphism. Every reflector \(r\colon\onecat{A}\to\onecat{B}\) onto a full replete subcategory is a 2-epimorphism of \(\LAdj\). Indeed, write \(i\) for its fully faithful right adjoint; the unit \(\eta_{A}\colon A\to irA\) is inverted by \(r\), so a natural transformation \(\theta\colon f\comp r\Rightarrow g\comp r\) is determined by its whiskering \(\theta\star i\), and precomposition with \(r\) is fully faithful.

	Take the abelianisation \(\mathrm{ab}\colon\Grp\to\Ab\colon G\mapsto G/[G,G]\), the reflector onto the abelian groups. Its 2-kernel is the category \(\Perf\) of \dfn{perfect} groups, those with \(G=[G,G]\), included by \(k\colon\Perf\to\Grp\). That inclusion is a 1-cell of \(\LAdj\), being left adjoint to the \dfn{perfect radical} \(G\mapsto pG\), the largest perfect subgroup of \(G\); a homomorphic image of a perfect group is perfect, so every morphism from a perfect group into \(G\) lands in \(pG\). For the universal property, a 1-cell \(z\colon\onecat{Z}\to\Grp\) carries an invertible 2-cell \(\mathrm{ab}\comp z\iso 0\) precisely when every \(z(Z)\) is perfect, in which case \(z\) corestricts to \(w\colon\onecat{Z}\to\Perf\) with \(k\comp w=z\); and \(w\) is again a left adjoint, its right adjoint being \(k\) followed by the right adjoint of \(z\). As \(k\) is fully faithful, this corestriction is unique up to a unique invertible 2-cell.

	Yet \(\mathrm{ab}\) is not a 2-cokernel of \(k\). Let \(f\colon\Grp\to\Nil_{2}\colon G\mapsto G/\gamma_{3}(G)\), be the reflector onto the groups of nilpotency class at most two. A perfect group coincides with every term of its lower central series, so \(f\comp k\iso 0\). Were \(\mathrm{ab}\) a 2-cokernel of \(k\), the 1-cell \(f\) would factor as \(u\comp\mathrm{ab}\) for some \(u\colon\Ab\to\Nil_{2}\), and any two groups with isomorphic abelianisations would have isomorphic images under \(f\). They need not: the free group of rank two and the free abelian group of rank two both abelianise to \(\mathbb{Z}^{2}\), while the second is sent by \(f\) to \(\mathbb{Z}^{2}\) and the first to the free nilpotent group of class two and rank two, which is not abelian.

	A 2-cokernel of \(k\) exists nonetheless, and it is the \dfn{hypoabelianisation}
	\[
		h\colon\Grp\to\HypAb\colon G\mapsto G/pG\text{.}
	\]
	Perfect groups are closed under extensions: if \(N\) contains \(pG\) with \(N/pG\) perfect, then \(N=[N,N]\cdot pG\), while \(pG=[pG,pG]\subseteq[N,N]\), so that \(N=[N,N]\). Hence \(G/pG\) has trivial perfect radical, and the groups with that property---those whose transfinite derived series reaches the trivial group---are the \dfn{hypoabelian} ones, onto which \(h\) is the reflector. It annihilates \(k\), a perfect group being its own perfect radical. It is universal as well: the perfect radical is characteristic, so \(G\to G/pG\) is a cokernel of the inclusion \(pG\to G\), and a 1-cell \(g\) with \(g\comp k\iso 0\) preserves cokernels and therefore sends every unit of \(h\) to an isomorphism, so that it factors through \(h\), uniquely up to a unique invertible 2-cell. Thus
	\[
		\xymatrix{0 \ar[r] & \Perf \ar[r]^-{k} & \Grp \ar[r]^-{h} & \HypAb \ar[r] & 0}
	\]
	is a short 2-exact sequence, whereas the pair \((k,\mathrm{ab})\) is not, although \(k\) is a 2-kernel of \(\mathrm{ab}\) and \(\mathrm{ab}\) is a 2-epimorphism. Since \(pG\subseteq[G,G]\), the abelianisation factors through the hypoabelianisation, by a comparison \(\HypAb\to\Ab\) which is far from an equivalence: a free group is residually nilpotent, hence hypoabelian.

	What separates the two 2-epimorphisms is an iteration. The reflector \(\mathrm{ab}\) divides out the commutator subgroup once, and the 2-cokernel divides out the transfinite derived series, which is where that operation converges. The argument just given uses nothing about groups beyond the closure of \(\Perf\) under extensions, so it applies to any reflector whose annihilated objects are closed under extensions: such a reflector is a 2-cokernel of its own 2-kernel exactly when its radical is idempotent.
\end{example}

We do not claim that \(\LAdj\) is 2-z-exact; the example uses only the 2-kernel and the 2-cokernel exhibited above.

\begin{proposition}\label{Prop Equivalence CoKernel}
	For a short 2-exact sequence
	\begin{equation}
		\xymatrix{0 \ar[r] & K \ar[r]^-{k} & A \ar[r]^-{q} & Q \ar[r] & 0}
	\end{equation}
	in a 2-z-exact 2-category, the object \(K\) is trivial if and only if \(q\) is an equivalence.

	Dually, \(Q\) is trivial if and only if \(k\) is an equivalence.
\end{proposition}
\begin{proof}
	We prove the first equivalence; the second is dual.

	Suppose \(K\) is trivial, so that \(\id{K}\iso 0\) and hence \(k=k\comp\id{K}\) is isomorphic to a null morphism. Then for every \(z\colon A\to Z\) the composite \(z\comp k\) is isomorphic to a null morphism, so condition~(1) of \defx\ref{Def 2-cokernel} yields \(u\colon Q\to Z\) with \(u\comp q\iso z\); this says that \((-)\comp q\colon\HomC{\L}{Q}{Z}\to\HomC{\L}{A}{Z}\) is essentially surjective. Condition~(2) says that it is fully faithful. As this holds for every \(Z\), \lemx\ref{L:RepEquiv} makes \(q\) an equivalence.

	Conversely, suppose \(q\) is an equivalence, with quasi-inverse \(p\). Whiskering \(\kappa\colon q\comp k\iso 0\) with \(p\) gives \(p\comp q\comp k\iso 0\), whence \(k\) is isomorphic to a null morphism. Now \(k\) is a 2-kernel, so a 2-monomorphism by \prox\ref{Prop Kernel Is Mono}, and therefore reflects null morphisms; applied to \(k\comp\id{K}=k\) this gives \(\id{K}\iso 0\), that is, \(K\) is trivial.
\end{proof}

\begin{corollary}\label{Trivial Kernel Normal Mono}
	Let \(f\colon A\to B\) be a normal morphism in a 2-z-exact 2-category, say \(f\iso m\comp e\) with \(e\) a normal 2-epimorphism and \(m\) a normal 2-monomorphism. If the 2-kernel of \(f\) is trivial, then \(e\) is an equivalence, and hence \(f\) is a normal 2-monomorphism. Dually, if the 2-cokernel of \(f\) is trivial, then \(m\) is an equivalence, and hence \(f\) is a normal 2-epimorphism.
\end{corollary}
\begin{proof}
	We prove the first statement; the second is dual. By \prox\ref{CoKernel of Composite} we have \(\twoker(f)\simeq\twoker(e)\), so the 2-kernel of \(e\) is trivial. As \(e\) is a normal 2-epimorphism it is the 2-cokernel of its 2-kernel, by \prox\ref{kernel is kernel of its cokernel}, so that \(\twoker(e)\) and \(e\) form a short 2-exact sequence. \prox\ref{Prop Equivalence CoKernel} applied to that sequence shows \(e\) to be an equivalence, whence \(f\iso m\) is a normal 2-monomorphism.\end{proof}

\begin{proposition}\label{Normal Epi Mono Equivalence}
	In a 2-z-exact 2-category, a normal 2-epimorphism that is a 2-monomorphism is an equivalence. Dually, a normal 2-monomorphism that is a 2-epimorphism is an equivalence.
\end{proposition}
\begin{proof}
	Let \(q\colon A\to Q\) be a 2-cokernel of \(g\colon G\to A\), with \(\eta\colon q\comp g\iso 0\), and suppose \(q\) is a 2-monomorphism. Then \(q\) reflects null morphisms by \prox\ref{prop2monoreflectsnull}, so \(g\) is isomorphic to a null morphism. Hence \(\id{A}\comp g\) is isomorphic to a null morphism, and condition~(1) of \defx\ref{Def 2-cokernel} yields \(u\colon Q\to A\) with \(u\comp q\iso\id{A}\). Whiskering with \(q\) gives \(q\comp u\comp q\iso q\iso\id{Q}\comp q\), and \(q\), being a 2-cokernel, is a 2-epimorphism by \prox\ref{Prop Kernel Is Mono}; so \(q\comp u\iso\id{Q}\) by the dual of \remx\ref{rem2monoismonouptoiso}. Thus \(q\) is an equivalence. The second statement is dual.
\end{proof}

\begin{corollary}\label{Equivalence Is Mono Plus Normal Epi}
	For a morphism \(f\colon A\to B\) in a 2-z-exact 2-category, the following are equivalent:
	\begin{tfae}
		\item \(f\) is an equivalence;
		\item \(f\) is both a normal 2-monomorphism and a normal 2-epimorphism;
		\item \(f\) is both a normal 2-monomorphism and a 2-epimorphism;
		\item \(f\) is both a 2-monomorphism and a normal 2-epimorphism.
	\end{tfae}
\end{corollary}
\begin{proof}
	We prove \((i)\Rightarrow(ii)\). The identity \(\id{B}\) is a 2-kernel of any chosen null morphism \(t\colon B\to 0\): the composite \(t\comp\id{B}\) is null, every morphism factors through the identity, and \(\id{B}\) is fully faithful. By \corx\ref{corollkeruniqueuptoequiv} the equivalence \(f\) is therefore a 2-kernel of \(t\) as well, and dually a 2-cokernel of any chosen null morphism \(0\to A\). Clearly \((ii)\) implies both \((iii)\) and \((iv)\), and \prox\ref{Normal Epi Mono Equivalence} shows that either of these implies \((i)\).
\end{proof}

\begin{corollary}\label{Normal Mono Criterion}
	For a morphism \(f\) in a 2-z-exact 2-category, the following are equivalent:
	\begin{tfae}
		\item \(f\) is a normal 2-monomorphism;
		\item \(f\) is a 2-monomorphism and it is normal.
	\end{tfae}
	Dually, \(f\) is a normal 2-epimorphism if and only if it is a 2-epimorphism and it is normal.
\end{corollary}
\begin{proof}
	We prove the first part; the second is dual. A normal 2-monomorphism is a 2-monomorphism, and it is normal, being isomorphic to itself composed with the identity. Conversely, if \(f\) is a 2-monomorphism then its 2-kernel is trivial by \lemx\ref{2-Mono Trivial Kernel}, so \corx\ref{Trivial Kernel Normal Mono} makes it a normal 2-monomorphism.
\end{proof}

\begin{proposition}\label{Normal Equivalence Trivial}
	A normal morphism is an equivalence if and only if its 2-kernel and its 2-cokernel are trivial.
\end{proposition}
\begin{proof}
	An equivalence is both a 2-monomorphism and a 2-epimorphism, so its 2-kernel and 2-cokernel are trivial by \lemx\ref{2-Mono Trivial Kernel}. Conversely, let \(f\iso m\comp e\) be normal with trivial 2-kernel and trivial 2-cokernel. By \corx\ref{Trivial Kernel Normal Mono}, triviality of the 2-kernel makes \(e\) an equivalence and triviality of the 2-cokernel makes \(m\) an equivalence; hence so is \(f\).
\end{proof}

\subsection{The normal image factorisation}\label{SS:ImageFactorisation}

\begin{proposition}\label{Image Factorisation of Normal Map is Unique}
	If, in a 2-z-exact 2-category, a morphism \(f\colon A\to B\) factors, up to an invertible 2-cell, as a normal 2-epimorphism \(e\colon A\to I\) followed by a normal 2-monomorphism \(m\colon I\to B\), then this factorisation is unique up to equivalence. We call it the \dfn{normal image factorisation} of \(f\) and write \(I=\TwoImg(f)\), \(m=\twoimg(f)\), and dually \(I=\TwoCoim(f)\), \(e=\twocoim(f)\).

	If \(f\) further factors, up to an invertible 2-cell, as a 2-epimorphism \(e'\colon A\to I'\) followed by a 2-monomorphism \(m'\colon I'\to B\), then necessarily \(e'\) is a normal 2-epimorphism and \(m'\) is a normal 2-monomorphism.
\end{proposition}
\begin{proof}
	Suppose \(f\iso m\comp e\) and \(f\iso m'\comp e'\) as stated. Consider the 2-kernel \(K\aar{k}A\) of \(f\), which is a 2-kernel of \(e\) by \prox\ref{CoKernel of Composite}; so \(e\) is a 2-cokernel of \(k\), by \prox\ref{kernel is kernel of its cokernel}. Since \(m'\comp e'\comp k\) is isomorphic to a null morphism and \(m'\) is a 2-monomorphism, \prox\ref{prop2monoreflectsnull} gives that \(e'\comp k\) is isomorphic to a null morphism too. The universal property of the 2-cokernel \(e\) of \(k\) now yields a morphism \(t\colon I\to I'\) and an invertible 2-cell \(\gamma\colon e'\iso t\comp e\). Moreover, since \(e\) is cofully faithful, the pasting
	\begin{cd}
		\& I \& {I'} \& \\
		A \& I \&\& B \text{,}
		\arrow["t", from=1-2, to=1-3]
		\arrow["{\gamma^{-1}}"{inner sep=1ex}, xshift = -2ex,shift right=6, iso, from=1-2, to=1-3]
		\arrow["{m'}", from=1-3, to=2-4]
		\arrow["e", from=2-1, to=1-2]
		\arrow["{e'}"{description},bend right=15, from=2-1, to=1-3]
		\arrow["e"', from=2-1, to=2-2]
		\arrow["m"', from=2-2, to=2-4]
		\arrow[shift left=5, iso, from=2-2, to=2-4]
	\end{cd}
	in which the invertible 2-cell on the right is the composite \(m'\comp e'\iso f\iso m\comp e\), induces an invertible 2-cell \(\sigma\colon m'\comp t\iso m\). The situation is this:
	\begin{cd}[4]
		\&\& {I'} \& \\
		K \& A \&\& B \\
		\&\& I
		\arrow["{m'}",tail, from=1-3, to=2-4]
		\arrow["k", from=2-1, to=2-2,nmono]
		\arrow["{e'}", two heads,from=2-2, to=1-3]
		\arrow["\gamma"'{inner sep=1ex}, iso,from=2-2, to=2-4,xshift =-4ex]
		\arrow["\sigma"'{inner sep=1ex}, iso, from=2-2, to=2-4,xshift =4ex]
		\arrow["e"', from=2-2, to=3-3,nepi]
		\arrow["t"{description}, from=3-3, to=1-3]
		\arrow["m"', from=3-3, to=2-4,nmono]
	\end{cd}
	By part~\textup{(ii)} of \prox\ref{Composites of Normal Monos}, since \(m\) is a normal 2-monomorphism and \(m'\) is a 2-monomorphism, \(t\) is a normal 2-monomorphism; and by the dual of part~\textup{(i)}, since \(e\) and \(e'\) are 2-epimorphisms, \(t\) is a 2-epimorphism. So \(t\) is an equivalence by \corx\ref{Equivalence Is Mono Plus Normal Epi}, and the two factorisations agree up to equivalence.

	In the course of this argument we assumed of \(m'\) and \(e'\) only that they are a 2-monomorphism and a 2-epimorphism. Since \(t\) is an equivalence, \corx\ref{corollkeruniqueuptoequiv} makes \(m'\) a normal 2-monomorphism and \(e'\) a normal 2-epimorphism.
\end{proof}

\begin{proposition}\label{P:NormalCancel}
	Let \(u\colon A\to B\) be a morphism of a 2-z-exact 2-category and let \(m\colon B\to C\) be a 2-monomorphism. If \(m\comp u\) is normal, then \(u\) is normal. Dually, if \(e\colon A\to B\) is a 2-epimorphism and \(u\comp e\) is normal, then \(u\) is normal.
\end{proposition}
\begin{proof}
	Let \(m\comp u\iso n\comp p\) be a normal image factorisation. By \prox\ref{CoKernel of Composite}, applied to \(m\comp u\) with the 2-monomorphism \(m\) and to \(n\comp p\) with the 2-monomorphism \(n\), the morphisms \(u\), \(m\comp u\) and \(p\) all have the same 2-kernel \(k\); and \(p\), being a normal 2-epimorphism, is a 2-cokernel of \(k\), by \prox\ref{kernel is kernel of its cokernel}. Since \(u\comp k\iso 0\), the universal property of \(p\) yields \(u'\) with \(u'\comp p\iso u\). Then \(m\comp u'\comp p\iso m\comp u\iso n\comp p\), and \(p\) is a 2-epimorphism, so \(m\comp u'\iso n\); as \(n\) is a normal 2-monomorphism and \(m\) is a 2-monomorphism, part~\textup{(ii)} of \prox\ref{Composites of Normal Monos} makes \(u'\) a normal 2-monomorphism. Hence \(u\iso u'\comp p\) is a normal 2-epimorphism followed by a normal 2-monomorphism, so normal.
\end{proof}

\subsection{Morphisms of short 2-exact sequences}\label{SS:MorphismSES}

\begin{definition}[morphism of short 2-exact sequences]\label{Def:morphism of SES}
	Let \((k',q')\) and \((k,q)\) be short 2-exact sequences, with structure 2-cells \(\kappa'\colon q'\comp k'\iso 0\) and \(\kappa\colon q\comp k\iso 0\). A \dfn{morphism of short 2-exact sequences} between them, drawn as
	\begin{cd}[5][5]
		0 \& {K'} \& {A'} \& {Q'} \& 0 \\
		0 \& K \& A \& Q \& 0
		\arrow[from=1-1, to=1-2]
		\arrow["{k'}", from=1-2, to=1-3]
		\arrow[""{name=0, anchor=center, inner sep=0}, "{f}"', from=1-2, to=2-2]
		\arrow["{q'}", from=1-3, to=1-4]
		\arrow[""{name=1, anchor=center, inner sep=0}, "{g}", from=1-3, to=2-3]
		\arrow[from=1-4, to=1-5]
		\arrow[""{name=2, anchor=center, inner sep=0}, "{h}", from=1-4, to=2-4]
		\arrow[from=2-1, to=2-2]
		\arrow["k"', from=2-2, to=2-3]
		\arrow["q"', from=2-3, to=2-4]
		\arrow[from=2-4, to=2-5]
		\arrow["\phi_K"{inner sep=1.1ex},shift right=1, iso, from=0, to=1]
		\arrow["\phi_Q"{inner sep=1.1ex},shift right=1, iso, from=1, to=2]
	\end{cd}
	consists of morphisms \(f\colon K'\to K\), \(g\colon A'\to A\) and \(h\colon Q'\to Q\) together with invertible 2-cells
	\[
		\phi_K\colon g\comp k'\iso k\comp f
		\qquad\text{and}\qquad
		\phi_Q\colon h\comp q'\iso q\comp g
	\]
	filling the two squares. Note that in the discretisation the 2-cells \(\phi_K\) and \(\phi_Q\) are identities, so the two squares commute on the nose.
\end{definition}

One might expect the definition to impose in addition a coherence condition, namely that the two invertible 2-cells \(q\comp g\comp k'\iso 0\)---one built from \(\phi_K\) and the structure 2-cell \(\kappa\) of the lower sequence, the other from \(\phi_Q\) and the structure 2-cell \(\kappa'\) of the upper one---agree. This is automatic, as we show below.

\begin{proposition}\label{P:CoherenceFree}
	Let \(\L\) be a 2-category with a strong bizero object. In the situation of \defx\ref{Def:morphism of SES},
	\begin{equation}\label{E:Coherence}
		(\kappa\star f)\cdot (q\star\phi_K)=(h\star\kappa')\cdot (\phi_Q^{-1}\star k')
	\end{equation}
	holds for every choice of \(f\), \(g\), \(h\), \(\phi_K\) and \(\phi_Q\), up to composing with the unique invertible 2-cell connecting the two null codomain 1-cells.
\end{proposition}
\begin{proof}
	Whiskering an invertible 2-cell with a morphism yields an invertible 2-cell, and a vertical composite of invertible 2-cells is invertible. Both sides of \eqref{E:Coherence} are therefore invertible 2-cells \(q\comp g\comp k'\aR{}0\) with the same domain and the same null codomain. By \lemx\ref{L:UniqueNull} they coincide.
\end{proof}

Note that strongness of the bizero object is what ensures that the coherence above is automatic.

\subsection{The Normal Short Five Lemma}\label{SS:NSFL}

\begin{theorem}[The Normal Short Five Lemma]\label{NSFL}
	Consider a morphism of short 2-exact sequences
	\begin{cd}[5][5]
		0 \& {K'} \& {A'} \& {Q'} \& 0 \\
		0 \& K \& A \& Q \& 0
		\arrow[from=1-1, to=1-2]
		\arrow["{k'}", from=1-2, to=1-3]
		\arrow[""{name=0, anchor=center, inner sep=0}, "{f}"', from=1-2, to=2-2]
		\arrow["{q'}", from=1-3, to=1-4]
		\arrow[""{name=1, anchor=center, inner sep=0}, "{g}", from=1-3, to=2-3]
		\arrow[from=1-4, to=1-5]
		\arrow[""{name=2, anchor=center, inner sep=0}, "{h}", from=1-4, to=2-4]
		\arrow[from=2-1, to=2-2]
		\arrow["k"', from=2-2, to=2-3]
		\arrow["q"', from=2-3, to=2-4]
		\arrow[from=2-4, to=2-5]
		\arrow["\phi_K"{inner sep=1.1ex},shift right=1, iso, from=0, to=1]
		\arrow["\phi_Q"{inner sep=1.1ex},shift right=1, iso, from=1, to=2]
	\end{cd}
	in a 2-z-exact 2-category, and assume that \(g\) is normal.
	\begin{enumerate}
		\item If \(f\) and \(h\) are 2-monomorphisms, then \(g\) is a normal 2-monomorphism.
		\item If \(f\) and \(h\) are 2-epimorphisms, then \(g\) is a normal 2-epimorphism.
		\item If \(f\) and \(h\) are equivalences, then \(g\) is an equivalence.
	\end{enumerate}
\end{theorem}
\begin{proof}
	We prove~(1). Write \(n\colon N\to A'\) for \(\twoker(g)\), so that \(g\comp n\iso 0\). Postcomposing with \(q\) and using \(\phi_Q\colon h\comp q'\iso q\comp g\) gives \(h\comp q'\comp n\iso 0\); as \(h\) is a 2-monomorphism it reflects null morphisms, so \(q'\comp n\iso 0\). Since \(k'=\twoker(q')\), the morphism \(n\) factors as \(n\iso k'\comp m\) for some \(m\colon N\to K'\). Precomposing \(\phi_K\colon g\comp k'\iso k\comp f\) with \(m\) now gives
	\[
		k\comp f\comp m\iso g\comp k'\comp m\iso g\comp n\iso 0\text{,}
	\]
	so that \(f\comp m\iso 0\) because \(k\) is a 2-monomorphism, and then \(m\iso 0\) because \(f\) is one. Hence \(n\iso k'\comp m\iso 0\), and \(n\), being a 2-kernel and so itself a 2-monomorphism, reflects this to \(\id{N}\iso 0\). Thus \(\TwoKer(g)\) is trivial, and since \(g\) is normal, \corx\ref{Trivial Kernel Normal Mono} makes it a normal 2-monomorphism.

	For~(2), write \(p\colon A\to P\) for \(\twocoker(g)\), so that \(p\comp g\iso 0\). Precomposing with \(k'\) and using \(\phi_K\) gives \(p\comp k\comp f\iso 0\); as \(f\) is a 2-epimorphism it coreflects null morphisms, so \(p\comp k\iso 0\). Since \(q=\twocoker(k)\), the morphism \(p\) factors as \(p\iso u\comp q\) for some \(u\colon Q\to P\). Postcomposing \(\phi_Q\) with \(u\) gives
	\[
		u\comp h\comp q'\iso u\comp q\comp g\iso p\comp g\iso 0\text{,}
	\]
	so that \(u\comp h\iso 0\) because \(q'\) is a 2-epimorphism, and then \(u\iso 0\) because \(h\) is one. Hence \(p\iso u\comp q\iso 0\) and \(\id{P}\iso 0\), so \(\TwoCoker(g)\) is trivial and \(g\) is a normal 2-epimorphism.

	Statement~(3) follows from~(1) and~(2): an equivalence is both a 2-monomorphism and a 2-epimorphism, so \(g\) is at once a normal 2-monomorphism and a normal 2-epimorphism, hence an equivalence by \corx\ref{Equivalence Is Mono Plus Normal Epi}.
\end{proof}

\begin{remark}\label{Rem NSFL Hypotheses}
	The proof consumes less than the statement offers. Part~(1) uses only that \(k'\) is a 2-kernel of \(q'\) and that \(k\), \(f\) and \(h\) are 2-monomorphisms: no property of \(q\) is used at all, \(q'\) need not be a 2-cokernel, and \(k\) need only be a 2-monomorphism rather than a 2-kernel. Part~(2) uses the mirror halves. In particular neither row need be a full short 2-exact sequence, 2-z-exactness is not needed beyond the existence of the 2-kernel or the 2-cokernel of \(g\) itself, and no bilimit is involved anywhere.
\end{remark}

\section{Exact sequences and the Pure Snake Lemma}\label{S:Exact}

\subsection{The canonical comparison}\label{SS:Comparison}

\begin{lemma}[Normality and the canonical comparison]\label{Normal Iff Comparison Iso}
	Let \(f\colon A\to B\) be a morphism in a 2-z-exact 2-category. Writing \(\twocoim(f)=\twocoker(\twoker(f))\) and \(\twoimg(f)=\twoker(\twocoker(f))\), there is a \dfn{comparison} \(j_f\colon\TwoCoim(f)\to\TwoImg(f)\), unique up to a unique invertible 2-cell, with \(\twoimg(f)\comp j_f\comp\twocoim(f)\iso f\). The morphism \(f\) is normal if and only if \(j_f\) is an equivalence.
\end{lemma}
\begin{proof}
	Existence and uniqueness of \(j_f\) are the universal properties of \(\twocoim(f)\) and \(\twoimg(f)\): the composite \(\twocoker(f)\comp f\) is isomorphic to a null morphism, so \(f\) factors through \(\twoimg(f)=\twoker(\twocoker(f))\), and dually the resulting morphism factors through \(\twocoim(f)\). Uniqueness holds because \(\twocoim(f)\) is a 2-epimorphism and \(\twoimg(f)\) a 2-monomorphism, so any two 1-cells filling the triangle agree, by \remx\ref{rem2monoismonouptoiso} and its dual.

	If \(j_f\) is an equivalence, then \(f\iso\twoimg(f)\comp\bigl(j_f\comp\twocoim(f)\bigr)\) exhibits \(f\) as a normal 2-monomorphism after a normal 2-epimorphism, the second factor because a 2-cokernel composed with an equivalence is again a 2-cokernel; so \(f\) is normal. Conversely, let \(f\iso m\comp e\) with \(m\) a normal 2-monomorphism and \(e\) a normal 2-epimorphism. By \prox\ref{CoKernel of Composite} we have \(\twoker(f)\simeq\twoker(e)\), and \(e\) is the 2-cokernel of its 2-kernel by \prox\ref{kernel is kernel of its cokernel}, so \(\twocoim(f)\simeq e\); dually \(\twoimg(f)\simeq m\). Under these equivalences the defining relation reads \(m\comp j_f\comp e\iso m\comp e\), and as \(m\) is fully faithful and \(e\) cofully faithful, \(j_f\) is isomorphic to an identity, hence an equivalence.
\end{proof}

\subsection{Exactness}\label{SS:Exactness}

\begin{proposition}\label{Exactness Self-Dual}
	For a pair \((f,g)\) of composable normal morphisms
	\begin{equation*}
		\xymatrix@!0@=4em{
		A \ar[rr]^-{f} \ar@{-{ >>}}[rd]_-{\twocoim(f)} &&
		B \ar[rr]^-{g} \ar@{-{ >>}}[rd]|-{\twocoim(g)} &&
		C \\
		& \TwoImg(f) \ar@{{ |>}->}[ru]|-{\twoimg(f)}
		&& \TwoCoim(g) \ar@{{ |>}->}[ru]_-{\twoimg(g)}}
	\end{equation*}
	with their respective normal image factorisations, the following conditions are equivalent:
	\begin{tfae}
		\item \(\twoimg(f)\simeq\twoker(g)\);
		\item \(\twocoker(f)\simeq\twocoim(g)\);
		\item the pair \((\twoimg(f),\twocoim(g))\) is a short 2-exact sequence.
	\end{tfae}
\end{proposition}
\begin{proof}
	By \prox\ref{CoKernel of Composite}, since \(\twoimg(g)\) is a normal 2-monomorphism, \(\twoker(g)\simeq\twoker(\twocoim(g))\); dually, since \(\twocoim(f)\) is a normal 2-epimorphism, \(\twocoker(f)\simeq\twocoker(\twoimg(f))\). Now \(\twoimg(f)\simeq\twoker(g)\) as in~(i) if and only if (iii) holds, because \(\twocoim(g)\), being a normal 2-epimorphism, is the 2-cokernel of its 2-kernel by \prox\ref{kernel is kernel of its cokernel}; likewise \(\twocoker(f)\simeq\twocoim(g)\) as in (ii) if and only if (iii) holds, because \(\twoimg(f)\), being a normal 2-monomorphism, is the 2-kernel of its 2-cokernel.
\end{proof}

The equivalent conditions of \prox\ref{Exactness Self-Dual} imply \(g\comp f\iso 0\).

\begin{definition}\label{Def:ExactSequence}
	A pair \((f,g)\) of composable normal morphisms is \dfn{exact in \(B\)} when the equivalent conditions of \prox\ref{Exactness Self-Dual} hold. A sequence of morphisms
	\begin{equation*}
		\cdots\to C_{n+1}\to C_n\to C_{n-1}\to\cdots
	\end{equation*}
	is \dfn{exact} if it is exact in each position where the condition makes sense. In particular, every morphism occurring in an exact sequence is required to be normal.
\end{definition}

Duality interchanges conditions (i) and (ii) and fixes (iii), so the notion just defined is self-dual. A fourth condition displays the symmetry rather than establishing it, and we shall meet it in \corx\ref{Cor Exact Null Dinversion} once the dinversion has been introduced.

\subsection{Dinversion}\label{SS:Dinversion}

\begin{definition}[dinversion]\label{Def:Dinversion}
	An \dfn{antinormal pair} \((m,e)\) consists of a normal 2-monomorphism \(m\colon K\to X\) and a normal 2-epimorphism \(e\colon X\to R\); its \dfn{antinormal composite} is \(e\comp m\). The \dfn{dinverse} of \((m,e)\) is the antinormal pair \((\twoker(e),\twocoker(m))\), and the \dfn{dinversion} of \((m,e)\) is the antinormal composite of its dinverse, namely
	\[
		w\coloneq\twocoker(m)\comp\twoker(e)\colon\TwoKer(e)\to\TwoCoker(m)\text{,}
	\]
	the vertical part of the cross
	\[
		\xymatrix@R=4ex@C=4.5em{
		& \TwoKer(e) \ar@{{ |>}->}[d]^-{\twoker(e)} \\
		K \ar@{{ |>}->}[r]^-{m} & X \ar@{-{ >>}}[r]^-{e} \ar@{-{ >>}}[d]_-{\twocoker(m)} & R \\
		& \TwoCoker(m)
		}
	\]
	Since \(m\) is the 2-kernel of its 2-cokernel and \(e\) the 2-cokernel of its 2-kernel, by \prox\ref{kernel is kernel of its cokernel}, dinversion is involutive: the dinverse of \((\twoker(e),\twocoker(m))\) is again \((m,e)\), up to equivalence.

	An \dfn{antinormal decomposition of the zero map} is an antinormal pair \((m,e)\) with \(e\comp m\iso 0\).
\end{definition}

\begin{proposition}\label{P:NullDinversion}
	Let \((m,e)\) be an antinormal decomposition of the zero map, let \(k\) be a 2-kernel of \(e\) and \(q\) a 2-cokernel of \(m\), and let \(w=q\comp k\) be its dinversion. Then \((m,e)\) is a short 2-exact sequence if and only if \(w\iso 0\).
\end{proposition}
\begin{proof}
	Suppose \(m=\twoker(e)\). From \(e\comp k\iso 0\) the morphism \(k\) factors through \(m\), say \(k\iso m\comp u\), and then \(w\iso q\comp m\comp u\iso 0\) because \(q\comp m\iso 0\). Conversely suppose \(w\iso 0\), and let \(t\) satisfy \(e\comp t\iso 0\). Then \(t\) factors through \(k\), say \(t\iso k\comp u\), so that \(q\comp t\iso w\comp u\iso 0\); and \(m\), being a normal 2-monomorphism, is a 2-kernel of its own 2-cokernel \(q\) by \prox\ref{kernel is kernel of its cokernel}, so \(t\) factors through \(m\). Together with \(e\comp m\iso 0\) and the fact that \(m\) is a 2-monomorphism, this exhibits \(m\) as a 2-kernel of \(e\). That \(e\) is then a 2-cokernel of \(m\) is the same argument read in the dual.
\end{proof}

\begin{corollary}\label{Cor Exact Null Dinversion}
	Let \((f,g)\) be a pair of composable normal morphisms with \(g\comp f\iso 0\). Then \((f,g)\) is exact in \(B\) if and only if the dinversion
	\[
		w=\twocoker(f)\comp\twoker(g)\colon\TwoKer(g)\to\TwoCoker(f)
	\]
	of the antinormal decomposition \((\twoimg(f),\twocoim(g))\) is null.
\end{corollary}
\begin{proof}
	The antinormal pair \((\twoimg(f),\twocoim(g))\) is an antinormal decomposition of the zero map because \(g\comp f \iso 0\) and by \prox\ref{prop2monoreflectsnull} 2-monomorphisms reflect null morphisms and 2-epimorphisms coreflect null morphisms. We can thus apply \prox\ref{P:NullDinversion}.
\end{proof}

Duality carries this condition to itself, and both directions of the proof are the same three-step factorisation argument read in the two possible orders. Note that the symmetry of exactness is thereby established before homological self-duality is defined, and that neither \prox\ref{Exactness Self-Dual} nor \prox\ref{P:NullDinversion} uses 2-z-exactness.

\begin{lemma}\label{L:DinversionCoker}
	Let \((m,e)\) be an antinormal pair in a 2-z-exact 2-category, with dinversion \(w=\twocoker(m)\comp\twoker(e)\). Then \(\TwoCoker(e\comp m)\simeq\TwoCoker(w)\), and dually \(\TwoKer(e\comp m)\simeq\TwoKer(w)\).
\end{lemma}
\begin{proof}
	Write \(m\colon K\to X\), \(e\colon X\to R\), \(k=\twoker(e)\) and \(q=\twocoker(m)\:X\to C\), so that \(w=q\comp k\). Being a normal 2-epimorphism, \(e\) is a 2-cokernel of \(k\), by \prox\ref{kernel is kernel of its cokernel}.

	Let \(u\colon R\to T\) satisfy \(u\comp e\comp m\iso 0\). Then \(q=\twocoker(m)\) yields \(u'\colon C\to T\) with \(u'\comp q\iso u\comp e\), and \(u'\comp w\iso u\comp e\comp k\iso 0\). Conversely, let \(v\colon C\to T\) satisfy \(v\comp w\iso 0\). Then \(v\comp q\comp k\iso 0\) and \(e=\twocoker(k)\) yields \(v'\colon R\to T\) with \(v'\comp e\iso v\comp q\), and \(v'\comp e\comp m\iso v\comp q\comp m\iso 0\). Each construction undoes the other, since \(e\) and \(q\) are 2-epimorphisms.

	Now let \(q_1\colon R\to Q_1\) be a 2-cokernel of \(e\comp m\) and \(q_2\colon C\to Q_2\) one of \(w\). Applying the first construction to \(q_1\) and the second to \(q_2\) gives \(q_1'\colon C\to Q_1\) with \(q_1'\comp q\iso q_1\comp e\) and \(q_1'\comp w\iso 0\), and \(q_2'\colon R\to Q_2\) with \(q_2'\comp e\iso q_2\comp q\) and \(q_2'\comp e\comp m\iso 0\); whence \(x\colon Q_2\to Q_1\) with \(x\comp q_2\iso q_1'\) and \(y\colon Q_1\to Q_2\) with \(y\comp q_1\iso q_2'\). Then \(x\comp q_2'\comp e\iso x\comp q_2\comp q\iso q_1'\comp q\iso q_1\comp e\), so \(x\comp q_2'\iso q_1\) as \(e\) is a 2-epimorphism; hence \(x\comp y\comp q_1\iso q_1\) and \(x\comp y\iso\id{Q_1}\), \(q_1\) being a 2-epimorphism. Symmetrically \(y\comp x\iso\id{Q_2}\). So \(x\) is an equivalence.
\end{proof}

\begin{remark}\label{Rem Dinversion Coker}
	The proof uses of \((m,e)\) only that \(q\) is a 2-cokernel of \(m\) and that \(e\) is a 2-cokernel of \(\twoker(e)\); nothing is asked of the composite \(e\comp m\), which need be neither null nor normal. In the abelian case the lemma reads \(R/e(K)\iso X/(K+N)\iso C/q(N)\) for \(N=\Ker(e)\), which is the identification the classical Snake Lemma makes when it computes the two ends of the connecting morphism.
\end{remark}

\subsection{Normal chain complexes and homology}\label{SS:Homology}

\begin{proposition}\label{P:NullNormal}
	In a 2-z-exact 2-category, every null morphism is normal.
\end{proposition}
\begin{proof}
	Let \(0\colon A\to B\) be null. Since \(\id{A}\) is a 2-kernel of it, \(\TwoCoim(0)=\TwoCoker(\id{A})\), which is trivial because a 2-epimorphism has trivial 2-cokernel, by the dual of \lemx\ref{2-Mono Trivial Kernel}; dually \(\TwoImg(0)=\TwoKer(\id{B})\) is trivial. Now any morphism \(t\colon P\to I\) between trivial objects is an equivalence: taking \(s\colon I\to P\) to be a null morphism, both \(t\comp s\) and \(\id{I}\) are null, hence isomorphic by \lemx\ref{L:UniqueNull}, and dually for \(s\comp t\). The composite \(\twoker(\id{B})\comp t\comp\twocoker(\id{A})\) is therefore a normal 2-monomorphism after a normal 2-epimorphism---a 2-cokernel composed with an equivalence being again a 2-cokernel---and it is null, hence isomorphic to \(0\).
\end{proof}

\begin{definition}[normal chain complex and its homology]\label{Def:Homology}
	A \dfn{normal chain complex} \((C,d)\) is a sequence of normal morphisms \(d_n\colon C_n\to C_{n-1}\), for \(n\in\mathbb{Z}\), with \(d_n\comp d_{n+1}\iso 0\). This relation makes \(\twoimg(d_{n+1})\) factor through \(\twoker(d_n)\) and, dually, \(\twocoim(d_n)\) factor through \(\twocoker(d_{n+1})\). The \dfn{cokernel homology} and \dfn{kernel homology} of \((C,d)\) at position \(n\) are
	\[
		\Hcoker{n}{C}\coloneq\TwoCoker\bigl(\TwoImg(d_{n+1})\to\TwoKer(d_n)\bigr)
	\]
	and
	\[
		\Hker{n}{C}\coloneq\TwoKer\bigl(\TwoCoker(d_{n+1})\to\TwoCoim(d_n)\bigr)\text{.}
	\]
	Equivalently, \(\Hcoker{n}{C}=\TwoCoim(w_n)\) and \(\Hker{n}{C}=\TwoImg(w_n)\), where
	\[
		w_n=\twocoker(d_{n+1})\comp\twoker(d_n)
	\]
	is the dinversion of the antinormal pair \((\twoimg(d_{n+1}),\twocoim(d_n))\); we write
	\[
		j_n\coloneq j_{w_n}\colon\Hcoker{n}{C}\to\Hker{n}{C}
	\]
	for the comparison of \lemx\ref{Normal Iff Comparison Iso}.
\end{definition}

\begin{remark}\label{Rem Padding}
	A normal chain complex is indexed by \(\mathbb{Z}\), so a finite sequence of normal morphisms has to be padded with bizero objects and null morphisms before it becomes one. \prox\ref{P:NullNormal} is what makes this legitimate, and it is the only place in this section where 2-z-exactness does anything beyond making a statement expressible.
\end{remark}

\subsection{Homological self-duality}\label{SS:HSD}

\begin{definition}[homologically self-dual 2-category]\label{Def:HSD}
	A 2-z-exact 2-category is \dfn{homologically self-dual} when the dinversion of every antinormal decomposition of the zero map is a normal morphism.
\end{definition}

By \prox\ref{P:NullDinversion} and \prox\ref{P:NullNormal}, homological self-duality and exactness are comparable conditions, and one is far weaker than the other: exactness asks a dinversion to be \emph{null}, homological self-duality asks every dinversion to be \emph{normal}, and a null morphism is normal. Exactness at a spot is thus the degenerate case of homological self-duality at that spot, and condition (ii) below is what survives when the dinversion does not vanish.

\begin{proposition}[criteria for homological self-duality]\label{Criteria HSD}
	For a 2-z-exact 2-category the following conditions are equivalent.
	\begin{tfae}
		\item The 2-category is homologically self-dual; that is, the dinversion of every antinormal decomposition of the zero map is normal.
		\item \emph{Homology is self-dual:} for every normal chain complex \((C,d)\) and every \(n\), the comparison \(j_n\colon\Hcoker{n}{C}\to\Hker{n}{C}\) is an equivalence.
		\item \emph{The Pure Snake condition holds:} in every morphism of short 2-exact sequences with identity middle component, as in \lemx\ref{Pure Snake Lemma} below, the comparison \(\TwoCoker(f)\to\TwoKer(h)\) is an equivalence.
	\end{tfae}
\end{proposition}
\begin{proof}
	By \lemx\ref{Normal Iff Comparison Iso} a morphism is normal exactly when its comparison is an equivalence. We show that each of the three conditions asserts precisely this for the dinversion of an arbitrary antinormal decomposition of the zero map. Condition~(i) is that assertion verbatim.

	For~(ii), a normal chain complex and a position \(n\) provide the antinormal decomposition \((\twoimg(d_{n+1}),\twocoim(d_n))\) of the zero map, whose antinormal composite is null since \(d_n\comp d_{n+1}\iso 0\); its dinversion is \(w_n\), with comparison \(j_n\). Conversely an antinormal decomposition \((m,e)\) of the zero map, with \(m\colon K\to X\) and \(e\colon X\to R\), sits at position \(0\) of the normal chain complex with \(K\), \(X\) and \(R\) in degrees \(2\), \(1\) and \(0\), a bizero object in every other degree, \(m\) and \(e\) as the two remaining differentials and null morphisms elsewhere; the padding is normal by \prox\ref{P:NullNormal}, and the dinversion at position \(0\) is \(w\). Hence~(i) and~(ii) coincide.

	For~(iii), take a morphism of short 2-exact sequences with identity middle component, with top row \(A\aar{a}B\aar{b}C\), bottom row \(X\aar{c}B\aar{d}Z\) and verticals \(f\), \(\id{B}\), \(h\). Then \((a,d)\) is an antinormal decomposition of the zero map, since \(d\comp a\iso d\comp c\comp f\iso 0\), with dinversion \(w=\twocoker(a)\comp\twoker(d)\iso b\comp c\). We claim \(\twoker(w)\simeq f\). Indeed \(w\comp f\iso b\comp a\iso 0\), so \(f\) factors through \(\twoker(w)\); and if \(t\) carries an invertible 2-cell \(w\comp t\iso 0\), then \(b\comp(c\comp t)\iso 0\), so the universal property of \(a=\twoker(b)\) gives \(s\) with \(a\comp s\iso c\comp t\), whence \(c\comp f\comp s\iso c\comp t\) and, \(c\) being fully faithful, \(f\comp s\iso t\) essentially uniquely. Dually \(\twocoker(w)\simeq h\). Thus \(\TwoCoim(w)\simeq\TwoCoker(f)\) and \(\TwoImg(w)\simeq\TwoKer(h)\), with \(j_w\) the comparison \(\TwoCoker(f)\to\TwoKer(h)\).

	Conversely, let \((m,e)\) be an antinormal decomposition of the zero map. Take the top row \(K\aar{m}X\aar{\twocoker(m)}\TwoCoker(m)\) and the bottom row \(\TwoKer(e)\aar{\twoker(e)}X\aar{e}R\), both short 2-exact by \prox\ref{kernel is kernel of its cokernel}, and the identity of \(X\) as middle component. The invertible 2-cell \(e\comp m\iso 0\) factors \(m\) through \(\twoker(e)\) by some \(t\colon K\to\TwoKer(e)\) and factors \(e\) through \(\twocoker(m)\) by some \(r\colon\TwoCoker(m)\to R\); these are the two outer verticals, and no coherence condition need be checked, by \prox\ref{P:CoherenceFree}. The dinversion of the resulting configuration is \(\twocoker(m)\comp\twoker(e)\), which is \(w\) itself. Hence~(i) and~(iii) coincide.
\end{proof}

\begin{remark}\label{Rem HSD Cost}
	The two directions of each equivalence are not paid for alike. The implications out of homological self-duality use no 2-z-exactness at all: every 2-kernel and 2-cokernel they need is either named in the target condition or produced by the normal image factorisation that self-duality itself supplies. The converses do use it---\mbox{(iii)~\(\Rightarrow\)~(i)} has to form the 2-cokernel of \(t\) and the 2-kernel of \(r\) before condition~(iii) can be applied, and \mbox{(ii)~\(\Rightarrow\)~(i)} needs it twice, once to pad the complex and once to form the comparison.
\end{remark}

\subsection{The Third Isomorphism Property}\label{SS:ThirdIso}

The Third Isomorphism Property is a further equivalent form of homological self-duality. We record it separately, as it concerns not a single dinversion but the short 2-exactness of the sequence comparing successive quotients, respectively subobjects, of a tower of normal 2-monomorphisms, respectively 2-epimorphisms.

\begin{definition}[totally normal sequence]\label{Def:TotallyNormal}
	A sequence of 2-monomorphisms \(A\aar{a}X\aar{c}B\) is \dfn{totally normal} when \(a\), \(c\) and the composite \(c\comp a\) are all normal 2-monomorphisms. Dually, a sequence of 2-epimorphisms \(X\aar{p}Y\aar{q}Z\) is \dfn{totally normal} when \(p\), \(q\) and \(q\comp p\) are all normal 2-epimorphisms.
\end{definition}

\begin{proposition}[Third Isomorphism Property]\label{Third Iso}
	For a 2-z-exact 2-category the following conditions are equivalent.
	\begin{tfae}
		\item The 2-category is homologically self-dual.
		\item For every totally normal sequence of 2-monomorphisms \(A\aar{a}X\aar{c}B\), the induced sequence
		\[
			\TwoCoker(a)\longrightarrow\TwoCoker(c\comp a)\longrightarrow\TwoCoker(c)
		\]
		is short 2-exact.
		\item For every totally normal sequence of 2-epimorphisms \(X\aar{p}Y\aar{q}Z\), the induced sequence
		\[
			\TwoKer(p)\longrightarrow\TwoKer(q\comp p)\longrightarrow\TwoKer(q)
		\]
		is short 2-exact.
	\end{tfae}
\end{proposition}
\begin{proof}
	Conditions~(ii) and~(iii) are dual, so we prove that~(i) and~(ii) are equivalent, using the Pure Snake characterisation \prox\ref{Criteria HSD}\,\textup{(iii)}.

	Given a totally normal sequence of 2-monomorphisms \(A\aar{a}X\aar{c}B\), put \(b=c\comp a\) and form the morphism of short 2-exact sequences with identity middle component whose top row is \(A\aar{b}B\aar{\twocoker(b)}\TwoCoker(b)\), bottom row \(X\aar{c}B\aar{\twocoker(c)}\TwoCoker(c)\), and right vertical \(r\) induced by \(\twocoker(c)\comp b\iso 0\). Both rows are short 2-exact by \prox\ref{kernel is kernel of its cokernel}, as \(b\) and \(c\) are normal 2-monomorphisms, and no coherence condition arises, by \prox\ref{P:CoherenceFree}. As in the proof of \prox\ref{Criteria HSD}, the dinversion is \(w=\twocoker(b)\comp c\), with \(\twoker(w)\simeq a\) and \(\twocoker(w)\simeq r\); hence the induced 1-cell \(\TwoCoker(a)\simeq\TwoCoim(w)\to\TwoCoker(b)\) is \(\twoimg(w)\comp j_w\), with 2-cokernel \(r\). The sequence \(\TwoCoker(a)\to\TwoCoker(b)\to\TwoCoker(c)\) is therefore short 2-exact if and only if \(j_w\) is an equivalence, that is, if and only if \(w\) is normal; and by \prox\ref{Criteria HSD} this holds for all such \(w\) precisely when the 2-category is homologically self-dual.

	Conversely, let \((m,e)\) be an antinormal decomposition of the zero map, with dinversion \(w=\twocoker(m)\comp\twoker(e)\). The invertible 2-cell \(e\comp m\iso 0\) corestricts \(m\) to a morphism \(t\colon K\to\TwoKer(e)\) with \(m\iso\twoker(e)\comp t\), and \(t\) is a normal 2-monomorphism by part~\textup{(ii)} of \prox\ref{Composites of Normal Monos}, the 2-monomorphism \(\twoker(e)\) cancelling off the normal 2-monomorphism \(m\). So \(K\aar{t}\TwoKer(e)\aar{\twoker(e)}X\) is a totally normal sequence of 2-monomorphisms. Applying~(ii) to it returns a short 2-exact sequence whose 2-monomorphism part is the induced 1-cell \(\TwoCoker(t)\to\TwoCoker(m)\); composed with \(\twocoker(t)\), that 1-cell is \(w\), by the construction of the induced morphism. Hence \(w\) is a normal 2-monomorphism after a normal 2-epimorphism, so normal.
\end{proof}

Note that the converse needs nothing about \(\twoker(w)\): applying~(ii) already exhibits \(w\) as a 2-cokernel followed by a 2-kernel.

\subsection{Exactness via homology}\label{SS:ExactnessHomology}

\begin{proposition}[exactness via homology]\label{Exactness via Homology}
	In a homologically self-dual 2-category, let \((f,g)\) be a pair of composable normal morphisms with \(g\comp f\iso 0\). Taking the 2-kernel of \(g\) and the 2-cokernel of \(f\), we obtain induced morphisms \(e\) and \(m\) and the 2-image \(H\) of the composite \(\twocoker(f)\comp\twoker(g)\).
	\begin{equation*}
		\xymatrix@!0@=4em{
		A \ar[rr]^-{f} \ar@{-->}[rd]_-{e} &&
		B \ar[rr]^-{g} \ar@{-{ >>}}[rd]|-{\twocoker(f)} &&
		C \\
		& \TwoKer(g) \ar@{--{ >>}}[rd]_-{r} \ar@{{ |>}->}[ru]|-{\twoker(g)}
		&& \TwoCoker(f) \ar@{-->}[ru]_-{m}\\
		&& H \ar@{{ |>}-->}[ru]_-{s}}
	\end{equation*}
	The following conditions are equivalent:
	\begin{tfae}
		\item \(e\) is a 2-epimorphism;
		\item \(m\) is a 2-monomorphism;
		\item \((f,g)\) is exact in \(B\);
		\item \(H\) is trivial.
	\end{tfae}
\end{proposition}
\begin{proof}
	Write \(p=\twocoker(f)\comp\twoker(g)\); then \(H\) is its 2-image, with \(r=\twocoim(p)\) and \(s=\twoimg(p)\). For a morphism \(x\) into \(\TwoKer(g)\), an invertible 2-cell \(p\comp x\iso 0\) amounts to \(\twoker(g)\comp x\) factoring through \(\twoimg(f)=\twoker(\twocoker(f))\); since \(\twoimg(f)\simeq\twoker(g)\comp\twoimg(e)\) and \(\twoker(g)\) is fully faithful, this amounts to \(x\) factoring through \(\twoimg(e)\). Hence \(\twoker(p)\simeq\twoimg(e)\) and \(H=\TwoCoim(p)\simeq\TwoCoker(e)\); dually \(H=\TwoImg(p)\simeq\TwoKer(m)\).

	The pair \((f,g)\) is exact in \(B\) precisely when \(\twoimg(f)\simeq\twoker(g)\), that is, when \(\twoimg(e)\simeq\TwoKer(g)\), that is, when \(\TwoCoim(p)=H\) is trivial; this is (iii)\(\Leftrightarrow\)(iv). If \((f,g)\) is exact, then \(\twoker(g)\comp e\iso f\iso\twoimg(f)\comp\twocoim(f)\) with \(\twoimg(f)\simeq\twoker(g)\), so \(e\simeq\twocoim(f)\) is a normal 2-epimorphism and dually \(m\simeq\twoimg(g)\) is a normal 2-monomorphism; thus (iii) implies (i) and (ii). Conversely, a 2-epimorphism has trivial 2-cokernel and a 2-monomorphism has trivial 2-kernel by \lemx\ref{2-Mono Trivial Kernel}, so \(e\) a 2-epimorphism gives \(H\simeq\TwoCoker(e)\) trivial and \(m\) a 2-monomorphism gives \(H\simeq\TwoKer(m)\) trivial; thus (i) and (ii) each imply (iv).
\end{proof}

\begin{remark}\label{Rem Exactness Homology Cost}
	Homological self-duality is used for one thing only: to know that \(p\) is normal, so that its 2-image and its 2-coimage agree and the object \(H\) exists at all. Once \(H\) is given, together with the factorisation of \(p\) as \(\twoimg(p)\comp\twocoim(p)\), the equivalence of the four conditions holds in any 2-category with a strong bizero object. This matters in Section~\ref{S:Snake}, where the criterion is applied repeatedly.
\end{remark}

\subsection{The Pure Snake Lemma}\label{SS:PureSnake}

\begin{lemma}[The Pure Snake Lemma]\label{Pure Snake Lemma}
	In a homologically self-dual 2-category, a morphism of short 2-exact sequences with identity middle component
	\begin{equation*}
		\xymatrix@!0@=4em{
		0 \ar[r] & A \ar[d]_{f} \ar[r]^-{a} &
		B \ar@{=}[d] \ar[r]^-{b} &
		C \ar[d]^{h} \ar[r] &
		0 \\
		0 \ar[r] &
		X \ar[r]_-{c} &
		B \ar[r]_-{d} &
		Z \ar[r] & 0}
	\end{equation*}
	induces the 2-exact sequence
	\begin{equation*}
		\xymatrix{0 \ar[r] & A \ar[r]^-f & X \ar[r]^-{b\comp c} & C \ar[r]^-h & Z \ar[r] & 0\text{.}}
	\end{equation*}
	In particular \(f\) is a normal 2-monomorphism and \(h\) is a normal 2-epimorphism, and the \dfn{comparison}
	\[
		j\colon\TwoCoker(f)\longrightarrow\TwoKer(h)\text{,}
		\qquad\text{characterised by}\qquad
		\twoker(h)\comp j\comp\twocoker(f)\iso b\comp c\text{,}
	\]
	is an equivalence, unique with that property, up to a unique invertible 2-cell.
\end{lemma}
\begin{proof}
	The proof of \prox\ref{Criteria HSD} identifies \(f\) with \(\twoker(w)\) and \(h\) with \(\twocoker(w)\) for the dinversion \(w\iso b\comp c\), so that \(\TwoCoim(w)\simeq\TwoCoker(f)\) and \(\TwoImg(w)\simeq\TwoKer(h)\) and the comparison \(j_w\) of \lemx\ref{Normal Iff Comparison Iso} is the displayed \(j\). Its uniqueness is that of \lemx\ref{Normal Iff Comparison Iso}: \(\twocoker(f)\) is a 2-epimorphism and \(\twoker(h)\) a 2-monomorphism, so any two 1-cells filling the triangle agree. Homological self-duality makes \(w\) normal, hence \(j\) an equivalence by \lemx\ref{Normal Iff Comparison Iso}, which is 2-exactness at \(X\) and at~\(C\); exactness at \(A\) and at \(Z\) says that \(f\) is a 2-monomorphism and \(h\) a 2-epimorphism, which they are, being a 2-kernel and a 2-cokernel of \(w\).
\end{proof}

\begin{remark}\label{Rem Pure Snake Cost}
	The identifications \(f\simeq\twoker(w)\) and \(h\simeq\twocoker(w)\) use neither homological self-duality nor the strongness of the bizero object: they are factorisation arguments through the universal properties of the two rows. Self-duality enters at exactly one point, the normality of \(w\), which is what 2-exactness at \(X\) and at~\(C\) requires.
\end{remark}

That the comparison is named, and unique, is what makes it possible to speak of its naturality; this is taken up in Section~\ref{S:Naturality}.

\section{Bipullbacks and the squares of a morphism of short 2-exact sequences}\label{S:Bipullbacks}

So far, we have not used bipullbacks. 2-kernels and 2-cokernels, short 2-exact sequences, the normal image factorisation, the Normal Short Five Lemma, dinversion, homological self-duality and the Pure Snake Lemma all rest on the universal properties of \defx\ref{Def 2-kernel} and \defx\ref{Def 2-cokernel} and on the reflection of null morphisms, and on nothing else. We introduce bipullbacks here, where they are first needed, and record what the two squares of a morphism of short 2-exact sequences have to say about them.

\subsection{Bipullbacks}\label{SS:Bipullback}

\begin{definition}[bipullback, or bi-iso-comma object]\label{Def Bipullback}
	Let \(A\aar{f}C\) and \(B\aar{g}C\) be morphisms in a 2-category \(\L\). A \dfn{bipullback} of \(f\) and \(g\) is an object \(P\) together with morphisms \(P\aar{p_1}A\) and \(P\aar{p_2}B\) and an isomorphism 2-cell
	\begin{cd}[7][7]
		P \arrow[r,"p_1"] \arrow[d,"p_2"'] \& A \arrow[d,"f"] \\
		B \arrow[r,"g"'] \& C
		\arrow[from=2-1,to=1-2,iso,"\phi"'{pos=0.5,inner sep=1ex}]
	\end{cd}
	such that:
	\begin{itemize}
		\item[\((1)\)] for all morphisms \(Z\aar{a}A\) and \(Z\aar{b}B\) and every isomorphism 2-cell
		      \[
			      \begin{cd}*[7][7]
				      Z \arrow[r,"a"] \arrow[d,"b"'] \& A \arrow[d,"f"] \\
				      B \arrow[r,"g"'] \& C
				      \arrow[from=2-1,to=1-2,iso,"\psi"'{pos=0.5,inner sep=1ex}]
			      \end{cd}
		      \]
		      there exist a morphism \(Z\aar{u}P\) and isomorphism 2-cells
		      \[
			      \begin{cd}*[4][4]
				      Z \arrow[rr,shift right= 1.7ex, iso, "\gamma_1"'{inner sep=-1ex, pos=0.65}] \& \& A \\[-3ex]
				      \& P
				      \arrow[from=1-1, to=1-3,"a"]
				      \arrow[from=1-1, to=2-2, bend right=20, "u"']
				      \arrow[from=2-2, to=1-3, bend right=20, "p_1"']
			      \end{cd}
			      \qquad\text{and}\qquad
			      \begin{cd}*[4][4]
				      Z \arrow[rr,shift right= 1.7ex, iso, "\gamma_2"'{inner sep=-1ex, pos=0.65}] \& \& B \\[-3ex]
				      \& P
				      \arrow[from=1-1, to=1-3,"b"]
				      \arrow[from=1-1, to=2-2, bend right=20, "u"']
				      \arrow[from=2-2, to=1-3, bend right=20, "p_2"']
			      \end{cd}
		      \]
		      such that \(\psi\) is equal to the pasting
		      \begin{cd}
			      Z \&[-2ex]\& \\[-2ex]
			      \& P \& A \\
			      \& B \& C
			      \arrow["u", from=1-1, to=2-2]
			      \arrow["a", bend left = 30, from=1-1, to=2-3]
			      \arrow["b"', bend right=30, from=1-1, to=3-2]
			      \arrow["{p_1}", from=2-2, to=2-3]
			      \arrow[shift left=8, iso, xshift= -4.5ex, "{\gamma_1}"{pos= 0.5, inner sep=1ex}, from=2-2, to=2-3]
			      \arrow[shift right=3.5, xshift=-9ex,iso,"{\gamma_2}"{pos= 0.5, inner sep=1ex}, from=2-2, to=2-3]
			      \arrow[""{name=0, anchor=center, inner sep=0}, "{p_2}"', from=2-2, to=3-2]
			      \arrow[""{name=1, anchor=center, inner sep=0}, "f", from=2-3, to=3-3]
			      \arrow["g"', from=3-2, to=3-3]
			      \arrow[iso, "\phi"'{pos=0.5,inner sep=1ex}, from=0, to=1]
		      \end{cd}
		\item[\((2)\)] for all morphisms \(u\), \(v\colon Z\to P\) and all 2-cells
		      \[
			      \begin{cd}*[4.5][4.5]
				      Z  \& P \arrow[d, Rightarrow, "\lambda_1", shorten <= -0.5ex, shorten <= -0.5ex] \& A \\[-3ex]
				      \& P
				      \arrow[from=1-1, to=1-2,"u"]
				      \arrow[from=1-1, to=2-2, bend right=20,"v"']
				      \arrow[from=1-2, to=1-3,"p_1"]
				      \arrow[from=2-2, to=1-3, bend right=20,"p_1"']
			      \end{cd}
			      \qquad\text{and}\qquad
			      \begin{cd}*[4.5][4.5]
				      Z  \& P \arrow[d, Rightarrow, "\lambda_2", shorten <= -0.5ex, shorten <= -0.5ex] \& B \\[-3ex]
				      \& P
				      \arrow[from=1-1, to=1-2,"u"]
				      \arrow[from=1-1, to=2-2, bend right=20,"v"']
				      \arrow[from=1-2, to=1-3,"p_2"]
				      \arrow[from=2-2, to=1-3, bend right=20,"p_2"']
			      \end{cd}
		      \]
		      satisfying \((g\star\lambda_2)\cdot(\phi\star u)=(\phi\star v)\cdot(f\star\lambda_1)\), there exists a unique 2-cell \(\mu\colon u\Rightarrow v\) such that \(p_1\star\mu=\lambda_1\) and \(p_2\star\mu=\lambda_2\).
	\end{itemize}
	The dual notion is that of a \dfn{bipushout}.
\end{definition}

\begin{proposition}\label{Kernel vs pullback}
	Let \(\L\) be a 2-category with a strong bizero object and let \(f\colon A\to B\) be a morphism. A morphism \(k\colon K\to A\) is a 2-kernel of \(f\) if and only if it is, together with the essentially unique morphism \(K\to 0\), a bipullback of \(f\) along a chosen null morphism \(0\to B\). Dually for 2-cokernels and bipushouts.
\end{proposition}
\begin{proof}
	A square over the cospan formed by \(A\aar{f}B\) and a chosen null morphism \(0\to B\), with vertex \(Z\), consists of a morphism \(a\colon Z\to A\), a morphism \(Z\to 0\)---of which there is one up to a unique isomorphism---and an invertible 2-cell \(\psi\) whose codomain is a null morphism; so it amounts to a morphism \(a\) together with an invertible 2-cell \(f\comp a\iso 0\), which is exactly the datum to which condition~(1) of \defx\ref{Def 2-kernel} applies. The second factorisation 2-cell \(\gamma_2\) and the pasting condition carry no information, by \lemx\ref{L:UniqueNull}: any two 2-cells into a morphism isomorphic to a null one agree. Likewise a pair of 2-cells \((\lambda_1,\lambda_2)\) as in condition~(2) of \defx\ref{Def Bipullback} reduces to \(\lambda_1\), and the required equation holds automatically for the same reason. Conditions~(1) and~(2) of \defx\ref{Def Bipullback} therefore say precisely what conditions~(1) and~(2) of \defx\ref{Def 2-kernel} say.
\end{proof}

\begin{proposition}\label{P:NormalMonoBipullback}
	Normal 2-monomorphisms are stable under bipullback. Precisely, let
	\begin{cd}[7][7]
		P \arrow[r,"p_1"] \arrow[d,"p_2"'] \& A \arrow[d,"f"] \\
		B \arrow[r,"g"'] \& C
		\arrow[from=2-1,to=1-2,iso,"\phi"'{pos=0.5,inner sep=1ex}]
	\end{cd}
	be a bipullback in a 2-category with a strong bizero object, let \(g\) be a 2-kernel of \(w\colon C\to W\), and suppose that \(p_1\) is a 2-monomorphism. Then \(p_1\) is a 2-kernel of \(w\comp f\); in particular it is a normal 2-monomorphism.
\end{proposition}
\begin{proof}
	First, \(w\comp f\comp p_1\iso w\comp g\comp p_2\iso 0\), using \(\phi\) and the structure 2-cell of the 2-kernel \(g\). For the universal property, let \(z\colon Z\to A\) satisfy \(w\comp f\comp z\iso 0\). Since \(g\) is a 2-kernel of \(w\), there are \(y\colon Z\to B\) and an invertible 2-cell \(g\comp y\iso f\comp z\); this is a square over the cospan formed by \(f\) and \(g\), so condition~\((1)\) of \defx\ref{Def Bipullback} provides \(u\colon Z\to P\) with \(p_1\comp u\iso z\). The 2-dimensional condition of \defx\ref{Def 2-kernel} is the hypothesis that \(p_1\) be a 2-monomorphism. Hence \(p_1\) is a 2-kernel of \(w\comp f\).
\end{proof}

\begin{remark}\label{Rem Bipullback Mono}
	The hypothesis that \(p_1\) be a 2-monomorphism is what one would obtain from the stability of 2-monomorphisms under bipullback; we have not needed that stability in general, and in the one application below (\prox\ref{P:NSDKappa}) the hypothesis is verified directly. Note also that no bilimit is constructed in the proof: as in \remx\ref{Rem Left Pullback Cost}, the bipullback is the one given.
\end{remark}


\subsection{The two squares}\label{SS:Squares}

Throughout this subsection we consider a morphism of short 2-exact sequences
\begin{cd}[5][5]
	0 \& {K'} \& {A'} \& {Q'} \& 0 \\
	0 \& K \& A \& Q \& 0
	\arrow[from=1-1, to=1-2]
	\arrow["{k'}", from=1-2, to=1-3]
	\arrow[""{name=0, anchor=center, inner sep=0}, "{f}"', from=1-2, to=2-2]
	\arrow["{q'}", from=1-3, to=1-4]
	\arrow[""{name=1, anchor=center, inner sep=0}, "{g}", from=1-3, to=2-3]
	\arrow[from=1-4, to=1-5]
	\arrow[""{name=2, anchor=center, inner sep=0}, "{h}", from=1-4, to=2-4]
	\arrow[from=2-1, to=2-2]
	\arrow["k"', from=2-2, to=2-3]
	\arrow["q"', from=2-3, to=2-4]
	\arrow[from=2-4, to=2-5]
	\arrow["\phi_K"{inner sep=1.1ex},shift right=1, iso, from=0, to=1]
	\arrow["\phi_Q"{inner sep=1.1ex},shift right=1, iso, from=1, to=2]
\end{cd}
as in \defx\ref{Def:morphism of SES}.

\begin{proposition}\label{Mono Implies Left Pullback}
	If the morphism \(h\) is a 2-monomorphism, then the square on the left is a bipullback.
\end{proposition}
\begin{proof}
	We verify the universal property of \defx\ref{Def Bipullback} directly for the square with vertex \(K'\), projections \(f\colon K'\to K\) and \(k'\colon K'\to A'\) and structure 2-cell \(\phi_K\colon g\comp k'\iso k\comp f\).

	For condition~(1), let \(a\colon Z\to K\) and \(b\colon Z\to A'\) come with an invertible 2-cell \(\psi\colon g\comp b\iso k\comp a\). Postcomposing with \(q\) and using the structure 2-cell \(\kappa\colon q\comp k\iso 0\) of the lower row gives \(q\comp g\comp b\iso 0\), and \(\phi_Q\colon h\comp q'\iso q\comp g\) turns this into \(h\comp q'\comp b\iso 0\). As \(h\) is a 2-monomorphism it reflects null morphisms, by \prox\ref{prop2monoreflectsnull}, so \(q'\comp b\iso 0\); since \(k'=\twoker(q')\), the morphism \(b\) factors as \(b\iso k'\comp u\) for some \(u\colon Z\to K'\), and this is the invertible 2-cell \(\gamma_2\). For \(\gamma_1\), note that
	\[
		k\comp f\comp u\iso g\comp k'\comp u\iso g\comp b\iso k\comp a\text{,}
	\]
	so that \(f\comp u\iso a\) because \(k\) is a 2-monomorphism, by \remx\ref{rem2monoismonouptoiso}. The pasting condition is then an equation between two invertible 2-cells with codomain \(k\comp a\); it holds because \(k\) is a 2-monomorphism, whiskering with \(k\) being injective on 2-cells.

	For condition~(2), let \(u\), \(v\colon Z\to K'\) and let \(\lambda_1\colon f\comp u\aR{}f\comp v\) and \(\lambda_2\colon k'\comp u\aR{}k'\comp v\) satisfy the compatibility of \defx\ref{Def Bipullback}. Since \(k'\) is a 2-kernel, hence a 2-monomorphism, there is a unique 2-cell \(\mu\colon u\aR{}v\) with \(k'\star\mu=\lambda_2\); and \(f\star\mu=\lambda_1\) follows, because whiskering both sides with the 2-monomorphism \(k\) gives equal 2-cells, by the compatibility just assumed.
\end{proof}

\begin{remark}\label{Rem Left Pullback Cost}
	The proof forms no bipullback other than the one it asserts, so it needs no bilimit hypothesis beyond the statement. It also uses less than the statement offers: only that \(k'\) is a 2-kernel of \(q'\), that \(k\) is a 2-kernel of \(q\), and that \(h\) is a 2-monomorphism. That \(q'\) is a 2-cokernel of \(k'\), and that \(q\) is a 2-cokernel of \(k\), never enter.
\end{remark}

\begin{proposition}\label{Right Square Pullback}
	If the square on the right is a bipullback, then the morphism \(f\) is an equivalence.
\end{proposition}
\begin{proof}
	The right-hand square exhibits \(A'\) as the bipullback of \(q\) along \(h\), with projections \(g\colon A'\to A\) and \(q'\colon A'\to Q'\) and structure 2-cell \(\phi_Q\). Composing the structure 2-cell \(\kappa\colon q\comp k\iso 0\) of the lower row with the null isomorphism \(h\comp 0\iso 0\) yields an invertible 2-cell \(q\comp k\iso h\comp 0\), to which the universal property of the bipullback associates a morphism \(s\colon K\to A'\) together with invertible 2-cells \(g\comp s\iso k\) and \(q'\comp s\iso 0\). From the latter and \(k'=\twoker(q')\) we obtain \(t\colon K\to K'\) with an invertible 2-cell \(k'\comp t\iso s\). We claim that \(t\) is a quasi-inverse of \(f\).

	First, \(k\comp f\comp t\iso g\comp k'\comp t\iso g\comp s\iso k\), using \(\phi_K\); since \(k\) is fully faithful, this reflects to an invertible 2-cell \(f\comp t\iso\id{K}\). Next, the morphisms \(s\comp f\) and \(k'\) from \(K'\) to \(A'\) satisfy \(g\comp s\comp f\iso k\comp f\iso g\comp k'\), through \(\phi_K\), and \(q'\comp s\comp f\iso 0\iso q'\comp k'\), through \(\kappa'\). These two invertible 2-cells are compatible with the structure 2-cell \(\phi_Q\) of the bipullback: the compatibility required is an equation between two invertible 2-cells with the same domain and with codomain \(q\comp g\comp k'\), which is isomorphic to a null morphism, so \lemx\ref{L:UniqueNull} supplies it. The uniqueness clause of the bipullback then yields an invertible 2-cell \(s\comp f\iso k'\). Hence \(k'\comp t\comp f\iso s\comp f\iso k'\), and the full faithfulness of \(k'\) reflects this to \(t\comp f\iso\id{K'}\). Therefore \(f\) is an equivalence.
\end{proof}

\begin{corollary}\label{Left Square Pushout}
	If the square on the left is a bipushout, then the morphism \(h\) is an equivalence.
\end{corollary}
\begin{proof}
	This is the statement dual to \prox\ref{Right Square Pullback}. Passing to the opposite 2-category interchanges 2-kernels with 2-cokernels, bipullbacks with bipushouts, and the short 2-exact sequence \(K\to A\to Q\) with \(Q\to A\to K\); under this duality the left-hand bipushout square and the morphism \(h\) correspond, respectively, to the right-hand bipullback square and the morphism \(f\) of \prox\ref{Right Square Pullback}.
\end{proof}

\section{The Snake Lemma in 2-di-exact 2-categories}\label{S:Snake}

\subsection{2-di-exact 2-categories}\label{SS:DiExact}

\begin{definition}[2-di-exact 2-category]\label{Def:DiExact}
	A 2-category \(\L\) with a strong bizero object is called \dfn{2-di-exact} if it satisfies the following two conditions:
	\begin{enum}
		\item[\textup{(DI1)}] \(\L\) has all 2-kernels and all 2-cokernels;
		\item[\textup{(DI2)}] every morphism that factorises, up to an invertible 2-cell, as a 2-kernel followed by a 2-cokernel also factorises, up to an invertible 2-cell, as a 2-cokernel followed by a 2-kernel.
		\begin{cd}[2.5]
			\& Q \& \\
			A \&\& B \\
			\& I
			\arrow["e", two heads, from=1-2, to=2-3]
			\arrow["m", tail, from=2-1, to=1-2]
			\arrow[iso, from=2-1, to=2-3]
			\arrow["{e'}"', dashed, two heads, from=2-1, to=3-2]
			\arrow["{m'}"', dashed, tail, from=3-2, to=2-3]
		\end{cd}
	\end{enum}
\end{definition}

In the language of Section~\ref{S:Normality}, condition (DI1) says that \(\L\) is 2-z-exact and condition (DI2) says that every antinormal morphism is normal. The Pure Snake Lemma is stated for a homologically self-dual 2-category, and the two hypotheses are connected by the following proposition.

\begin{proposition}\label{P:DiExactHSD}
	In a 2-category satisfying \textup{(DI2)}, the dinversion of every antinormal pair is normal. In particular, a 2-di-exact 2-category is homologically self-dual.
\end{proposition}
\begin{proof}
	Let \((m,e)\) be an antinormal pair, with 2-kernel \(\twoker(e)\) and 2-cokernel \(\twocoker(m)\). A 2-kernel is a normal 2-monomorphism and a 2-cokernel is a normal 2-epimorphism, so the dinversion \(w=\twocoker(m)\comp\twoker(e)\) is a normal 2-monomorphism followed by a normal 2-epimorphism, that is, an antinormal morphism. By \textup{(DI2)} it is normal. Homological self-duality is the special case in which \(e\comp m\iso 0\).
\end{proof}

\begin{remark}\label{Rem Bridge Cost}
	The argument never uses the hypothesis \(e\comp m\iso 0\), which is why the conclusion is stated for every antinormal pair and not only for an antinormal decomposition of the zero map. It does not use condition \textup{(DI1)} either, so no 2-kernel or 2-cokernel need exist beyond the two named in the statement; and it does not use that the bizero object is strong.
\end{remark}

Everything proved in Section~\ref{S:Exact} is therefore available under the standing hypothesis of this section: the Pure Snake Lemma, \prox\ref{Criteria HSD}, \prox\ref{Third Iso} and \prox\ref{Exactness via Homology}.

\begin{example}\label{Ex DiExact}
	Any abelian category \(\onecat{A}\), regarded as a locally discrete 2-category in the sense of Section~\ref{S:Preliminaries}, is 2-di-exact.

	Everything discretises. Two parallel null 1-cells are both the zero morphism, hence equal, hence carry exactly one 2-cell between them, so the zero object of \(\onecat{A}\) is a strong bizero object; in dimension one there is no difference between a bizero object and a strong one. Since the only invertible 2-cells are identities, a 1-cell is essentially null exactly when the morphism it names is zero, 2-kernels are kernels, 2-cokernels are cokernels, 2-monomorphisms are monomorphisms and 2-epimorphisms are epimorphisms. Condition \textup{(DI1)} is then the statement that every morphism of \(\onecat{A}\) has a kernel and a cokernel, and condition \textup{(DI2)} is the image factorisation: in an abelian category every monomorphism is the kernel of its cokernel and every epimorphism is the cokernel of its kernel, so every 1-cell is normal, and \emph{a fortiori} every antinormal one is.
\end{example}

\begin{remark}\label{Rem Model}
	\exax\ref{Ex DiExact} is one-dimensional, so it exhibits nothing that the classical theory does not already contain; its purpose is to show that the hypotheses of this section are consistent, and with them that homological self-duality and the Pure Snake Lemma are not vacuous. A genuinely two-dimensional model is the subject of Section~\ref{S:Model}. The candidate that suggests itself, the 2-category \(\AbCat\) of abelian categories, exact functors and natural transformations, satisfies \textup{(DI1)} but not \textup{(DI2)} (\prox\ref{P:AbCatFails}); what works instead is to cut out the abelian categories in which the obstruction vanishes, and the resulting 2-category (\thex\ref{T:SatModel}) contains the finitely generated modules over every commutative Noetherian ring and \(\Coh(X)\) for every Noetherian scheme \(X\) (\corx\ref{C:Populated}). A second model, locally ordered rather than locally discrete, is the 2-category of complete modular lattices of \thex\ref{T:LatticeModel}, where condition \textup{(DI2)} is Dedekind's transposition principle.
\end{remark}

\begin{remark}[2-homological 2-categories]\label{Rem Grandis}
	The 2-homological 2-categories of \cite[Section 1]{CavJanMesRei26} are axiomatised differently, and dinversion says where they sit. The conditions there are that the 2-category have a strong bizero object and all 2-kernels and all 2-cokernels; that its normal 2-monomorphisms and its normal 2-epimorphisms each be closed under composition; and that every antinormal composite \(e\comp m\) for which \(\twoker(e)\) factors through \(m\) be normal. The first condition is 2-z-exactness and the second is \(\NEC\) together with its dual. The third is condition \textup{(DI2)} restricted to those antinormal pairs whose dinversion vanishes: a normal 2-monomorphism \(m\) is a 2-kernel of \(\twocoker(m)\) by \prox\ref{kernel is kernel of its cokernel}, so \(\twoker(e)\) factors through \(m\) exactly when the dinversion \(w=\twocoker(m)\comp\twoker(e)\) is null.

	What identifies the restriction is that dinversion interchanges that hypothesis with that conclusion. It carries an antinormal pair \((m,e)\) to \((\twoker(e),\twocoker(m))\), whose antinormal composite is \(w\) and whose own dinversion is \(e\comp m\) again, and it is involutive, so it permutes the antinormal pairs. Read on the dinverse pair, the third condition therefore says that the dinversion of every antinormal decomposition of the zero map is normal, and that is \defx\ref{Def:HSD}. A 2-homological 2-category is precisely a homologically self-dual 2-z-exact 2-category satisfying \(\NEC\) and its dual.

	This locates the gap between the two settings, and it is a single hypothesis. Such a 2-category satisfies everything Section~\ref{S:Exact} asks for, so the Pure Snake Lemma holds in it; and it satisfies \(\NEC\), which is the second hypothesis of \thex\ref{T:SnakeNonSelfDual}. Only the first is missing, homological self-duality standing where \(\DPN\) is wanted. By \corx\ref{C:HSDstrict} that difference is real: the 2-category \(\AbCat\) is 2-homological, by \prox\ref{P:AbCatComposition} and \prox\ref{P:AbCatHSD}, and it fails \(\DPN\) by \prox\ref{P:AbCatNotDPN}, so neither \thex\ref{Snake General 2D} nor \thex\ref{T:SnakeNonSelfDual} is available there. The 2-category \(\Sup\) of \prox\ref{P:SupNotDPN} is a second example of the same kind, with the pentagon as witness. In the other direction, condition \textup{(DI2)} does not imply the two closure conditions: \remx\ref{Rem NSD Discrete} separates it from \(\NEC\) already in dimension one. So neither package contains the other.
\end{remark}

\subsection{The statement, and the construction}\label{SS:Construction}

\begin{theorem}[The Snake Lemma]\label{Snake General 2D}
	In a 2-di-exact 2-category, consider a diagram
	\begin{equation}\label{Snake Ladder}
		\begin{cd}*
			\& {A} \& {B} \& {C} \& 0 \\
			0 \& X \& Y \& Z \&
			\arrow["{a}", from=1-2, to=1-3]
			\arrow[""{name=0, anchor=center, inner sep=0}, "{f}"', from=1-2, to=2-2]
			\arrow["{b}", from=1-3, to=1-4]
			\arrow[""{name=1, anchor=center, inner sep=0}, "{g}", from=1-3, to=2-3]
			\arrow[from=1-4, to=1-5]
			\arrow[""{name=2, anchor=center, inner sep=0}, "{h}", from=1-4, to=2-4]
			\arrow[from=2-1, to=2-2]
			\arrow["c"', from=2-2, to=2-3]
			\arrow["d"', from=2-3, to=2-4]
			\arrow["\phi"{inner sep=1.1ex},shift right=1, iso, from=0, to=1]
			\arrow["\psi"{inner sep=1.1ex},shift right=1, iso, from=1, to=2]
		\end{cd}
	\end{equation}
	whose two squares commute up to invertible 2-cells \(\phi\colon g\comp a\iso c\comp f\) and \(\psi\colon h\comp b\iso d\comp g\), and in which \(f\), \(g\) and \(h\) are normal. If the horizontal rows are 2-exact, then there exists a 1-cell \(\partial\) such that the sequence
	\begin{equation}\label{Snake Sequence}
		\xymatrix@R=5ex@C=2em{
		\TwoKer({f}) \ar[r]^-{\ov{a}} &
		\TwoKer({g}) \ar[r]^-{\ov{b}} &
		\TwoKer({h}) \ar `d[l] `[dll]|-{\partial} [dll] &\\
		\TwoCoker({f}) \ar[r]_-{\underline{c}} &
		\TwoCoker({g}) \ar[r]_-{\underline{d}} &
		\TwoCoker({h})\text{,}
		}
	\end{equation}
	in which \(\ov{a}\) and \(\ov{b}\) are induced by taking 2-kernels and \(\underline{c}\) and \(\underline{d}\) by taking 2-cokernels, is 2-exact.

	If, moreover, \(a=\twoker(b)\), then \(\ov{a}=\twoker(\ov{b})\); dually, if \(d=\twocoker(c)\), then \(\underline{d}=\twocoker(\underline{c})\).
\end{theorem}

We prove first the special case in which \(a=\twoker(b)\) and \(d=\twocoker(c)\), so that both rows are short 2-exact sequences, and reduce the general case to it in Section~\ref{SS:GeneralCase}. Throughout, 2-exactness of \eqref{Snake Sequence} asserts in particular that its six 1-cells are normal, as \defx\ref{Def:ExactSequence} requires; for \(\ov{b}\) and \(\underline{c}\) this comes out of the construction, and we point it out where it does.

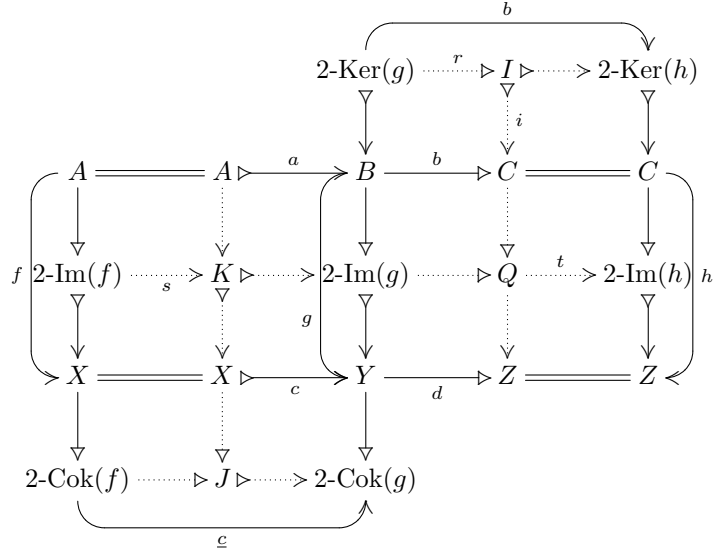
\begin{figure}
	\begin{equation*}
		\xymatrix{&& \TwoKer({g}) \ar@{.{ >>}}[r]^-{r} \ar@{{ |>}->}[d] & I \ar@{{ |>}.>}[r] \ar@{{ |>}.>}[d]^-{i} & \TwoKer({h}) \ar@{{ |>}->}[d] \ar@{<-} `u[l] `[ll]_-{\ov{b}}\\
		A \ar@{=}[r] \ar@{-{ >>}}[d] & A \ar@{.>}[d] \ar@{{ |>}->}[r]^-{a} & B \ar@{-{ >>}}[r]^-{b} \ar@{-{ >>}}[d] & C \ar@{.{ >>}}[d] \ar@{=}[r] & C \ar@{-{ >>}}[d]\\
		\TwoImg({f}) \ar@{.>}[r]_-{s} \ar@{{ |>}->}[d] & K \ar@{{ |>}.>}[r] \ar@{{ |>}.>}[d] & \TwoImg({g}) \ar@{.{ >>}}[r] \ar@{{ |>}->}[d] & Q \ar@{.>}[r]^-{t} \ar@{.>}[d] & \TwoImg({h}) \ar@{{ |>}->}[d] \\
		X \ar@{<-} `l[u] `[uu]^-{f} \ar@{-{ >>}}[d] \ar@{=}[r] & X \ar@{.{ >>}}[d] \ar@{{ |>}->}[r]_-{c} & Y \ar@{<-} `l[u] `[uu]^(.2){g} \ar@{-{ >>}}[r]_-{d} \ar@{-{ >>}}[d] & Z \ar@{=}[r] & Z \ar@{<-} `r[u] `[uu]_-{h}\\
		\TwoCoker({f}) \ar@{.{ >>}}[r] & J \ar@{{ |>}.>}[r] & \TwoCoker({g}) \ar@{<-} `d[l] `[ll]^-{\underline{c}}}
	\end{equation*}
	\caption{Constructing the snake sequence}\label{Fig Constructing Snake}
\end{figure}

The construction proceeds as in Figure~\ref{Fig Constructing Snake}. The composite \(b\comp\twoker(g)\) is a normal 2-monomorphism followed by a normal 2-epimorphism, hence antinormal, hence normal by \textup{(DI2)}; we factor it as a normal 2-epimorphism \(r\colon\TwoKer(g)\to I\) followed by a normal 2-monomorphism \(i\colon I\to C\).

Since \(d\comp g\comp\twoker(g)\iso 0\) and \(\psi\colon h\comp b\iso d\comp g\), we have \(h\comp b\comp\twoker(g)\iso 0\), that is, \(h\comp i\comp r\iso 0\); as \(r\) is a 2-epimorphism, \(h\comp i\iso 0\). Hence \(i\) corestricts along \(\twoker(h)\) to a 1-cell \(\ov{\imath}\colon I\to\TwoKer(h)\) with \(\twoker(h)\comp\ov{\imath}\iso i\), and \(\ov{\imath}\) is a normal 2-monomorphism by part~\textup{(ii)} of \prox\ref{Composites of Normal Monos}, the 2-monomorphism \(\twoker(h)\) cancelling off the normal 2-monomorphism \(i\).

\begin{lemma}\label{L:ImageOfBBar}
	The pair \((r,\ov{\imath})\) is a normal image factorisation of \(\ov{b}\). In particular \(\ov{b}\) is normal, with \(\twoimg(\ov{b})\simeq\ov{\imath}\), and \(\TwoCoker(\ov{b})\simeq\TwoCoker(\ov{\imath})\).
\end{lemma}
\begin{proof}
	Composing with the 2-monomorphism \(\twoker(h)\) turns both \(\ov{\imath}\comp r\) and \(\ov{b}\) into \(b\comp\twoker(g)\): for the first by the defining 2-cell of \(\ov{\imath}\) and the factorisation of \(b\comp\twoker(g)\), for the second by the defining 2-cell of \(\ov{b}\). By \remx\ref{rem2monoismonouptoiso} they agree up to an invertible 2-cell. A normal 2-epimorphism followed by a normal 2-monomorphism is a normal image factorisation, by \prox\ref{Image Factorisation of Normal Map is Unique}, and the last assertion is \prox\ref{CoKernel of Composite}.
\end{proof}

Write \(Q\coloneq\TwoCoker(i)\) and \(q\coloneq\twocoker(i)\). From \(h\comp i\iso 0\) we get \(\twoimg(h)\comp\twocoim(h)\comp i\iso 0\), and \(\twoimg(h)\) is a 2-monomorphism, so \(\twocoim(h)\comp i\iso 0\); the universal property of \(q\) therefore induces \(t\colon Q\to\TwoImg(h)\) with \(t\comp q\iso\twocoim(h)\), where we identify \(\TwoCoim(h)\) with \(\TwoImg(h)\) through the comparison of \lemx\ref{Normal Iff Comparison Iso}, an equivalence because \(h\) is normal.

The first application of the Pure Snake Lemma is to the rows \(I\aar{i}C\aar{q}Q\) and \(\TwoKer(h)\aar{\twoker(h)}C\aar{\twocoim(h)}\TwoImg(h)\) through the shared object \(C\), with left vertical \(\ov{\imath}\) and right vertical \(t\). Both rows are short 2-exact, by \prox\ref{kernel is kernel of its cokernel}, and \prox\ref{P:CoherenceFree} disposes of any coherence. The lemma yields an equivalence \(\TwoCoker(\ov{\imath})\simeq\TwoKer(t)\), and with \lemx\ref{L:ImageOfBBar} an equivalence
\[
	z_1\colon\TwoCoker(\ov{b})\longrightarrow\TwoKer(t)\text{.}
\]
Dually, writing \(s\colon\TwoImg(f)\to K\) for the 1-cell obtained by the mirror construction in the lower half of Figure~\ref{Fig Constructing Snake}, we obtain an equivalence \(z_3\colon\TwoCoker(s)\to\TwoKer(\underline{c})\), and \(\underline{c}\) is normal.

\subsection{The connecting 1-cell}\label{SS:Connecting}

The third application of the Pure Snake Lemma is to the two rows
\[
	\xymatrix{0 \ar[r] & \TwoImg({f}) \ar[d]_-{s} \ar@{{ |>}->}[r]^-{\ell} & \TwoImg({g}) \ar@{=}[d] \ar@{-{ >>}}[r]^-{u} & Q \ar[d]^{t} \ar[r]& 0\\
	0 \ar[r] & K \ar@{{ |>}->}[r] & \TwoImg({g}) \ar@{-{ >>}}[r] & \TwoImg({h}) \ar[r]&  0}
\]
through the shared object \(\TwoImg(g)\). Here \(\ell\) and \(u\) are the induced comparisons: writing \(p=\twocoim(g)\), \(e=\twocoim(f)\) and \(k=\twoker(g)\), the 1-cell \(\ell\) is characterised by \(\ell\comp e\iso p\comp a\) and \(u\) by \(u\comp p\iso q\comp b\). Both exist: \(g\comp a\comp\twoker(f)\iso c\comp f\comp\twoker(f)\iso 0\) makes \(a\comp\twoker(f)\) factor through \(k\), whence \(p\comp a\comp\twoker(f)\iso 0\) and \(e=\twocoker(\twoker(f))\) induces \(\ell\); and \(q\comp b\comp k\iso q\comp i\comp r\iso 0\) with \(p=\twocoker(k)\) induces \(u\). We identify \(\TwoCoim(f)\) with \(\TwoImg(f)\) and \(\TwoCoim(g)\) with \(\TwoImg(g)\), \(f\) and \(g\) being normal.

That the top row is short 2-exact is a routine verification. It has two halves. The first is that \(\ell\) is a normal 2-monomorphism, and the second that \(u\) is a 2-cokernel of \(\ell\).

\begin{lemma}\label{L:ImgMapNormal}
	The 1-cell \(\ell\) is a normal 2-monomorphism, and \(e\) is a normal 2-epimorphism.
\end{lemma}
\begin{proof}
	The composite \(p\comp a\) is a normal 2-monomorphism followed by a normal 2-epimorphism, hence antinormal, hence normal by \textup{(DI2)}. Now \(p\comp a\iso\ell\comp e\) exhibits it as a 2-epimorphism followed by a 2-monomorphism: \(e\) is a 2-epimorphism, and \(\ell\) is a 2-monomorphism because \(\twoimg(g)\comp\ell\iso c\comp\twoimg(f)\) is a composite of two 2-monomorphisms, so that part~\textup{(i)} of \prox\ref{Composites of Normal Monos} applies. The second assertion of \prox\ref{Image Factorisation of Normal Map is Unique} then makes \(e\) a normal 2-epimorphism and \(\ell\) a normal 2-monomorphism.
\end{proof}

\begin{proposition}\label{P:RoutineVerification}
	Let \(b\) be a 2-cokernel of \(a\), let \(b\comp k\iso i\comp r\) with \(r\) a 2-epimorphism, let \(q\) be a 2-cokernel of \(i\) and let \(p\) be a 2-cokernel of \(k\). If \(e\) is a 2-epimorphism, then \(u\) is a 2-cokernel of \(\ell\).
\end{proposition}
\begin{proof}
	First, \(u\comp\ell\comp e\iso u\comp p\comp a\iso q\comp b\comp a\iso 0\), and \(e\) is a 2-epimorphism, so \(u\comp\ell\iso 0\).

	Now let \(w\) satisfy \(w\comp\ell\iso 0\). Then \(w\comp p\comp a\iso w\comp\ell\comp e\iso 0\), so that \(b=\twocoker(a)\) yields \(v\) with \(v\comp b\iso w\comp p\). Next, \(v\comp i\comp r\iso v\comp b\comp k\iso w\comp p\comp k\iso 0\), because \(p\comp k\) is the composite of a 2-kernel with its 2-cokernel; as \(r\) is a 2-epimorphism, \(v\comp i\iso 0\). Hence \(q=\twocoker(i)\) yields \(w'\) with \(w'\comp q\iso v\), and then \(w'\comp u\comp p\iso w'\comp q\comp b\iso v\comp b\iso w\comp p\), so that \(w'\comp u\iso w\) since \(p\) is a 2-epimorphism. Finally \(u\) is a 2-epimorphism, because \(u\comp p\iso q\comp b\) is one.
\end{proof}

By \lemx\ref{L:ImgMapNormal} the 1-cell \(\ell\) is a normal 2-monomorphism, hence the 2-kernel of its 2-cokernel, which \prox\ref{P:RoutineVerification} identifies with \(u\); so the top row is short 2-exact. The bottom row is short 2-exact by duality. The Pure Snake Lemma therefore supplies an equivalence
\[
	z_2\colon\TwoCoker(s)\longrightarrow\TwoKer(t)\text{,}
\]
and we write \(z\coloneq z_3\comp z_2^{-1}\comp z_1\colon\TwoCoker(\ov{b})\to\TwoKer(\underline{c})\) for the composite of the three, a quasi-inverse being chosen for \(z_2\). The \dfn{connecting 1-cell} is
\begin{equation}\label{Def Partial}
	\partial\coloneq\twoker(\underline{c})\comp z\comp\twocoker(\ov{b})\colon\TwoKer(h)\longrightarrow\TwoCoker(f)\text{.}
\end{equation}

\begin{proposition}\label{P:Shape}
	Let \(q'\) be a 2-cokernel of \(\ov{b}\), let \(z\) be an equivalence, let \(m\) be a 2-kernel of \(\underline{c}\), and put \(\partial=m\comp z\comp q'\). If \(\underline{c}\) is normal, then the pair \((\partial,\underline{c})\) is 2-exact at \(\TwoCoker(f)\); and if \(\ov{b}\) is normal, then the pair \((\ov{b},\partial)\) is 2-exact at \(\TwoKer(h)\).
\end{proposition}
\begin{proof}
	A 2-cokernel followed by an equivalence is again a 2-cokernel, so \(z\comp q'\) is a normal 2-epimorphism and \(\partial=m\comp(z\comp q')\) already exhibits \(\partial\) as a normal 2-epimorphism followed by \(\twoker(\underline{c})\); the morphism \(\underline{c}\) being normal, this is 2-exactness at \(\TwoCoker(f)\), by \prox\ref{Exactness Self-Dual}. For the second assertion, take a normal image factorisation \(\ov{b}\iso n\comp e'\). Since \(e'\) is a 2-epimorphism, \(q'\) is a 2-cokernel of \(n\) as well; \(n\) being a normal 2-monomorphism, it is the 2-kernel of that 2-cokernel, by \prox\ref{kernel is kernel of its cokernel}; and appending the 2-monomorphism \(m\comp z\) does not change a 2-kernel, by \prox\ref{CoKernel of Composite}. So \(n\) is a 2-kernel of \(\partial\), which is 2-exactness at \(\TwoKer(h)\), the morphism \(\partial\) being normal by the first part of the proof.
\end{proof}

Applied to \eqref{Def Partial}, this gives 2-exactness of the snake sequence at \(\TwoKer(h)\) and at \(\TwoCoker(f)\); the normality of \(\ov{b}\) that the second half needs is \lemx\ref{L:ImageOfBBar}, and the normality of \(\underline{c}\) that the first half needs is its dual. Note that the first half uses nothing about \(\ov{b}\) at all---not even that \(q'\) is its 2-cokernel rather than some other morphism's.

\begin{remark}\label{Rem Partial Choices}
	The 1-cell \(\partial\) depends on the choices made in the construction, but only up to an invertible 2-cell: each of the three comparisons is unique up to an invertible 2-cell over its own data, by \lemx\ref{Pure Snake Lemma}, and isomorphic equivalences have isomorphic quasi-inverses. Two runs of the construction therefore give connecting 1-cells that agree up to an invertible 2-cell.
\end{remark}

\subsection{Exactness at the outer positions}\label{SS:OuterExactness}

\begin{proposition}\label{P:KerBarA}
	Suppose \(a=\twoker(b)\) and let \(c\) be a 2-monomorphism. Then \(\ov{a}=\twoker(\ov{b})\). Dually, if \(d=\twocoker(c)\) and \(b\) is a 2-epimorphism, then \(\underline{d}=\twocoker(\underline{c})\).
\end{proposition}
\begin{proof}
	From \(b\comp a\iso 0\) and the full faithfulness of \(\twoker(h)\) we get \(\ov{b}\comp\ov{a}\iso 0\). Suppose now that \(x\colon T\to\TwoKer(g)\) satisfies \(\ov{b}\comp x\iso 0\). Since \(\twoker(h)\comp\ov{b}\iso b\comp\twoker(g)\), the composite \(b\comp\twoker(g)\comp x\) is isomorphic to a null morphism, so \(a=\twoker(b)\) provides \(y\colon T\to A\) with \(a\comp y\iso\twoker(g)\comp x\). Then
	\[
		c\comp f\comp y\iso g\comp a\comp y\iso g\comp\twoker(g)\comp x\iso 0\text{,}
	\]
	and here the 2-dimensionality is essential: as \(c\) is fully faithful it reflects null morphisms, by \prox\ref{prop2monoreflectsnull}, so \(f\comp y\iso 0\); hence \(y\iso\twoker(f)\comp y'\) for an essentially unique \(y'\colon T\to\TwoKer(f)\). Now
	\[
		\twoker(g)\comp\ov{a}\comp y'\iso a\comp\twoker(f)\comp y'\iso a\comp y\iso\twoker(g)\comp x\text{,}
	\]
	and \(\twoker(g)\) being fully faithful forces \(\ov{a}\comp y'\iso x\). Finally \(\ov{a}\) is itself a 2-monomorphism: the composite \(\twoker(g)\comp\ov{a}\iso a\comp\twoker(f)\) is one, being a composite of two 2-monomorphisms, and \(\twoker(g)\) is one, so part~\textup{(i)} of \prox\ref{Composites of Normal Monos} applies. Hence the factorisation of \(x\) through \(\ov{a}\) is essentially unique, and \(\ov{a}=\twoker(\ov{b})\).
\end{proof}

\begin{remark}\label{Rem BarA Cost}
	This costs far less than the theorem states. It uses neither 2-di-exactness, nor 2-z-exactness, nor the Pure Snake Lemma, nor the normality of \(f\), \(g\) and \(h\). Of the lower row it uses only that \(c\) is a 2-monomorphism: \(c\) need not be a 2-kernel, and \(d\) does not occur in the proof at all. It does use \(a=\twoker(b)\), which is the hypothesis of the special case, and it uses that the bizero object is strong---at the reflection step, and nowhere else.
\end{remark}

In the special case, \(\ov{a}\) is thus a 2-kernel of \(\ov{b}\), so the pair \((\ov{a},\ov{b})\) is 2-exact at \(\TwoKer(g)\); dually \(\underline{d}\) is a 2-cokernel of \(\underline{c}\), and since \(\underline{c}\) is normal, by the dual of \lemx\ref{L:ImageOfBBar} again, the pair \((\underline{c},\underline{d})\) is 2-exact at \(\TwoCoker(g)\), by \prox\ref{Exactness Self-Dual}. Together with \prox\ref{P:Shape} this establishes \thex\ref{Snake General 2D} in the special case.

\subsection{The general case}\label{SS:GeneralCase}

\begin{figure}
	\begin{equation*}
		\xymatrix@!0@R=4em@C=7em{
		\TwoKer({f}) \ar@{--{ >>}}[r]^-{\ov{a}''} \ar@{{ |>}->}[d] & \TwoKer({f}') \ar@{{ |>}-->}[r]^-{\ov{a}'} \ar@{{ |>}->}[d] & \TwoKer({g}) \ar@{<--} `u[l] `[ll]_-{\ov{a}} \ar@{-->}[r]^-{\ov{b}} \ar@{{ |>}->}[d] & \TwoKer({h}') \ar@{-->}[dddll]^(.8){\partial} \ar@{{ |>}->}[d] \ar@{{ |>}->}[rd] \\
		A \ar[d]_{f} \ar@{-{ >>}}[r]^-{\twocoim({a})} & \TwoCoim({a}) \ar[d]_{{f}'} \ar@{{ |>}->}[r]^-{{a}'} &
		B \ar[d]_(.3){g} \ar@{-{ >>}}[r]^(.7){b} &
		C \ar[d]^{{h}'} \ar@{=}[r] &
		C \ar[d]^{h} \\
		X \ar@{-{ >>}}[rd] \ar@{=}[r] &
		X \ar@{-{ >>}}[d] \ar@{{ |>}->}[r]_(.3){c} &
		Y \ar@{-{ >>}}[d] \ar@{-{ >>}}[r]_-{{d}'} & \TwoImg({d}) \ar@{-{ >>}}[d] \ar@{{ |>}->}[r]_-{\twoimg({d})} &
		Z \ar@{-{ >>}}[d]\\
		& \TwoCoker({f}') \ar@{-->}[r]_-{\underline{c}} & \TwoCoker({g}) \ar@{--{ >>}}[r]_-{\underline{d}'} & \TwoCoker({h}') \ar@{{ |>}-->}[r]_-{\underline{d}''} & \TwoCoker({h}) \ar@{<--} `d[l] `[ll]^-{\underline{d}}}
	\end{equation*}
	\caption{Reducing the general case to the special one}\label{Fig Snake General}
\end{figure}
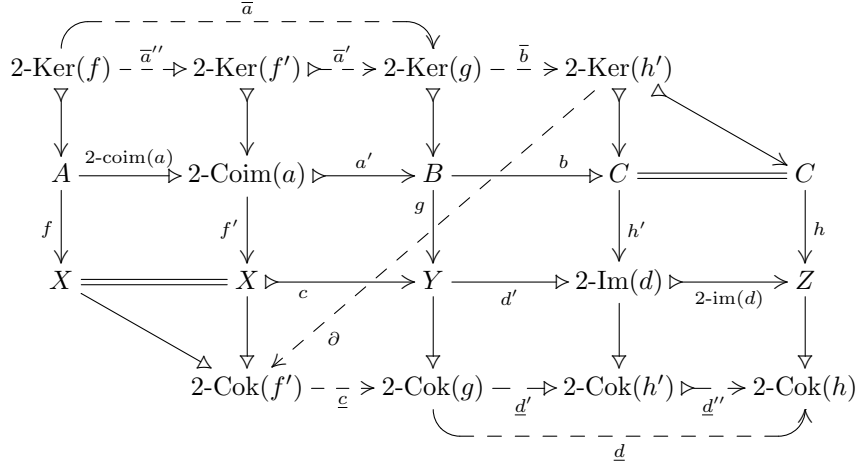

We take the normal image factorisations \(a\iso a'\comp\twocoim(a)\) and \(d\iso\twoimg(d)\comp d'\) of the normal morphisms \(a\) and \(d\), and apply the special case to the resulting diagram, as in Figure~\ref{Fig Snake General}. Write \(e\coloneq\twocoim(a)\).

The reduction needs two identifications that are worth stating, since the whole of Figure~\ref{Fig Snake General} rests on them. Since \(h\iso\twoimg(d)\comp h'\) with \(\twoimg(d)\) a 2-monomorphism, \prox\ref{CoKernel of Composite} gives \(\TwoKer(h')\simeq\TwoKer(h)\); dually \(\TwoCoker(f')\simeq\TwoCoker(f)\). Both are instances of that proposition in the form in which neither factor is required to be normal.

\begin{lemma}\label{L:Restriction}
	The morphism \(f\) restricts along \(e\), that is, there is a 1-cell \(f'\colon\TwoCoim(a)\to X\) with \(f'\comp e\iso f\); and \(f'\) is normal. Dually \(h\) corestricts along \(\twoimg(d)\) to a normal \(h'\).
\end{lemma}
\begin{proof}
	Write \(k\) for \(\twoker(e)\). Then \(c\comp f\comp k\iso g\comp a\comp k\iso 0\), since \(a\iso a'\comp e\) and \(e\comp k\iso 0\); and \(c\), being fully faithful, reflects this to \(f\comp k\iso 0\), by \prox\ref{prop2monoreflectsnull}. As \(e\) is a 2-cokernel of \(k\), this produces \(f'\) with \(f'\comp e\iso f\).

	For normality, let \(w\) be the 1-cell with \(w\comp e\iso\twocoim(f)\), induced in the same way. Then \(w\) is a normal 2-epimorphism, since \(w\comp e\) is one and \(e\) is a 2-epimorphism, by the dual of part~\textup{(ii)} of \prox\ref{Composites of Normal Monos}. Cancelling the 2-epimorphism \(e\) in \(f'\comp e\iso f\iso\twoimg(f)\comp w\comp e\) gives \(f'\iso\twoimg(f)\comp w\), a normal image factorisation. The dual statement follows by duality.
\end{proof}

It remains to see that the comparisons \(\ov{a}''\colon\TwoKer(f)\to\TwoKer(f')\) and \(\underline{d}''\colon\TwoCoker(h')\to\TwoCoker(h)\) are a normal 2-epimorphism and a normal 2-monomorphism, for that is what gives 2-exactness of the snake sequence at \(\TwoKer(g)\) and at \(\TwoCoker(g)\). This is a fourth application of the Pure Snake Lemma.

\begin{proposition}\label{P:KerComparison}
	Let \(e=\twocoim(a)\), let \(v\colon\TwoKer(e)\to\TwoKer(f)\) be the comparison induced by \(f\comp\twoker(e)\iso 0\), and let \(f'\) be the restriction of \(f\) along \(e\) of \lemx\ref{L:Restriction}. If \(v\) has a 2-cokernel, then the comparison \(\ov{a}''\colon\TwoKer(f)\to\TwoKer(f')\) is a normal 2-epimorphism. Dually for \(\underline{d}''\).
\end{proposition}
\begin{proof}
	The two rows
	\[
		\TwoKer(e)\rightarrowtail A\twoheadrightarrow\TwoCoim(a)
		\qquad\text{and}\qquad
		\TwoKer(f)\rightarrowtail A\twoheadrightarrow\TwoCoim(f)
	\]
	are short 2-exact and share the middle object \(A\); their verticals are \(v\) on the left and, on the right, the 1-cell \(w\) with \(w\comp e\iso\twocoim(f)\) of \lemx\ref{L:Restriction}. This is a configuration of the shape to which \lemx\ref{Pure Snake Lemma} applies, so it supplies an equivalence \(z\colon\TwoCoker(v)\to\TwoKer(w)\) with \(\twoker(w)\comp z\comp\twocoker(v)\iso e\comp\twoker(f)\). Now \(f'\iso\twoimg(f)\comp w\) by \lemx\ref{L:Restriction}, so \(\TwoKer(w)\simeq\TwoKer(f')\) by \prox\ref{CoKernel of Composite}; and the displayed triangle is exactly the one characterising \(\ov{a}''\). Hence \(\ov{a}''\iso z\comp\twocoker(v)\), a 2-cokernel followed by an equivalence, and therefore a normal 2-epimorphism.
\end{proof}

This completes the proof of \thex\ref{Snake General 2D}: the special case applies to the middle of Figure~\ref{Fig Snake General}, the two identifications transport its conclusions to \(\TwoKer(h)\) and \(\TwoCoker(f)\), and \prox\ref{P:KerComparison} supplies 2-exactness at the two remaining positions.

\begin{remark}\label{Rem General Cost}
	The reduction is 2-dimensional in exactly one place, and it is not the place the special case is: \prox\ref{prop2monoreflectsnull} is used three times, to produce \(f'\), the comparison on 2-kernels, and---dually---\(h'\). \prox\ref{P:KerComparison} uses no 2-di-exactness beyond what \prox\ref{P:DiExactHSD} needs to make the Pure Snake Lemma available, no 2-z-exactness, and no identification of a 2-image; its only existence assumption is a 2-cokernel of \(v\).
\end{remark}

\subsection{The classical Snake Lemma}\label{SS:Classical}

A locally discrete 2-category is a 1-category, and under the discretisation of Section~\ref{SS:Discretisation} every notion used above becomes its classical counterpart: a bizero object is a zero object, a strong bizero object is again a zero object, 2-kernels and 2-cokernels are kernels and cokernels, 2-monomorphisms and 2-epimorphisms are monomorphisms and epimorphisms, equivalences are isomorphisms, and 2-di-exactness is di-exactness. \thex\ref{Snake General 2D} therefore specialises to the classical statement.

\begin{corollary}[The Snake Lemma, 1-categorical version]\label{Snake General}
	In a di-exact category, consider a commutative diagram
	\begin{equation*}
		\xymatrix@!0@=4em{
		& A \ar[d]_{f} \ar[r]^-{a} &
		B \ar[d]_{g} \ar[r]^-{b} &
		C \ar[d]^{h} \ar[r] &
		0 \\
		0 \ar[r] &
		X \ar[r]_-{c} &
		Y \ar[r]_-{d} &
		Z}
	\end{equation*}
	in which \(f\), \(g\) and \(h\) are normal. If the horizontal rows are exact, then there is a morphism \(\partial\colon\Ker(h)\to\Coker(f)\) making the sequence
	\begin{equation*}
		\xymatrix@R=5ex@C=2em{
		\Ker({f}) \ar[r]^-{\ov{a}} &
		\Ker({g}) \ar[r]^-{\ov{b}} &
		\Ker({h}) \ar `d[l] `[dll]|-{\partial} [dll] &\\
		\Coker({f}) \ar[r]_-{\underline{c}} &
		\Coker({g}) \ar[r]_-{\underline{d}} &
		\Coker({h})
		}
	\end{equation*}
	exact. If moreover \(a=\ker(b)\), then \(\ov{a}=\ker(\ov{b})\); dually, if \(d=\coker(c)\), then \(\underline{d}=\coker(\underline{c})\).\noproof
\end{corollary}

\begin{remark}\label{Rem Truncation}
	The corollary rests on the discretisation of Section~\ref{SS:Discretisation}, which reads a 1-category as a locally discrete 2-category. That reading is a section of a quotient running the other way: the \dfn{1-truncation} \(\trunc\L\) of a 2-category \(\L\) is the 1-category with the same objects in which a morphism \(A\to B\) is an isomorphism class \([f]\) of 1-cells \(f\colon A\to B\), all 2-cells being discarded. When \(\L\) has a bizero object, any two null morphisms \(A\to B\) are isomorphic, and their common class makes \(\trunc\L\) pointed. It turns out that the truncation preserves and reflects every assertion of \thex\ref{Snake General 2D}, so that the theorem---though not the construction that proves it, and not the naturality of Section~\ref{S:Naturality}---may alternatively be deduced from its 1-categorical counterpart. This remark records the argument, and where it stops. The 1-truncation is studied in its own right in \cite{CavJanMesRei26}---a bipointed 2-category with all 2-kernels and 2-cokernels is 2-homological precisely when its 1-truncation is homological \cite[Section 1]{CavJanMesRei26}---and it is there, not here, that the general account of the exactness properties it preserves and reflects belongs; what follows is the special case that the Snake Lemma needs.

	The truncation sends 2-kernels to kernels. This is not a general fact about bilimits: a bipullback in the sense of \defx\ref{Def Bipullback} yields, in general, only a \emph{weak} pullback in \(\trunc\L\), because condition~(1) there produces a factorisation for each choice of filling 2-cell \(\psi\), and factorisations for distinct choices need not be isomorphic. What saves the 2-kernel is its condition~(2). Given \([z]\) with \([f]\comp[z]=0\), condition~(1) of \defx\ref{Def 2-kernel} provides a factorisation \([z]=[\twoker(f)]\comp[u]\), and the class \([u]\) is unique, because \(\twoker(f)\comp(-)\) is fully faithful by \prox\ref{Prop Kernel Is Mono} and fully faithful functors reflect isomorphisms. Conversely, the truncation \emph{reflects} kernels wherever a 2-kernel exists: if \([k']\) is a kernel of \([f]\), then \([k']\) and \([\twoker(f)]\) differ by an isomorphism of \(\trunc\L\); the isomorphisms of \(\trunc\L\) are precisely the classes of the equivalences of \(\L\); so \(k'\) is, up to an invertible 2-cell, a 2-kernel of~\(f\) composed with an equivalence, which \corx\ref{C:NormalTransport} makes a 2-kernel of~\(f\). Under condition \textup{(DI1)} this argument and its dual identify each notion of this paper that is a property of 1-cells with its discretisation interpreted in \(\trunc\L\): normal 2-monomorphisms and normal 2-epimorphisms, normal and antinormal morphisms, dinversion, short 2-exact sequences and, through condition~(i) of \prox\ref{Exactness Self-Dual}, exactness of a pair of normal morphisms. Thus \(\trunc\L\) is z-exact, and each of the conditions \textup{(DI2)}, \DPN, \NEC\ and homological self-duality holds in \(\L\) if and only if its discretisation holds in \(\trunc\L\); in particular, a 2-category satisfying \textup{(DI1)} is 2-di-exact precisely when its 1-truncation is di-exact. Here lies the abstract reason why each of these axioms, evaluated in a model, will turn into a condition with no 2-dimensional content left in it: Serre saturation in \(\AbCat\) (\prox\ref{P:DIabcat}), Dedekind's transposition principle on complete lattices (\prox\ref{P:SupModular}), Mackey's symmetry on Hilbert lattices (\prox\ref{P:HilbertSymmetric}).

	The Snake Lemma therefore transports. For 2-di-exact \(\L\), the hypotheses of \thex\ref{Snake General 2D} truncate to those of \corx\ref{Snake General} in the di-exact category \(\trunc\L\): the invertible 2-cells \(\phi\) and \(\psi\) become commuting squares, 2-exact rows exact rows, normal verticals normal verticals. There the conclusion holds by the 1-categorical Snake Lemma, \cite[Theorem~2.2.2 and Corollary~2.2.4]{PVdL3}, and it reflects: the objects of the six-term exact sequence are \(\TwoKer(f)\), \dots, \(\TwoCoker(h)\), its induced morphisms are the classes of \(\ov{a}\), \(\ov{b}\), \(\underline{c}\) and \(\underline{d}\), all of its morphisms are normal, any representative of the connecting class serves as \(\partial\), and 2-exactness at the four positions of \thex\ref{Snake General 2D} follows by reflecting exactness at each of them. The Pure Snake Lemma corresponds in the same way to \cite[Lemma~2.2.1]{PVdL3}, proved there, as \lemx\ref{Pure Snake Lemma} is here, from homological self-duality alone. What the transport does not yield is the 2-cells. Its \(\partial\) is a bare isomorphism class, where the construction of Sections~\ref{SS:Construction} and~\ref{SS:Connecting} produces a specific 1-cell, well defined up to an invertible 2-cell by \remx\ref{Rem Partial Choices}; the 2-naturality established in Section~\ref{S:Naturality} is a statement about the 2-cells attached to that specific \(\partial\), of which \(\trunc\L\) retains nothing, the functoriality of the truncated sequence being its shadow rather than a substitute. Nor does the truncation reach the hypotheses: reflection is not creation, a kernel in \(\trunc\L\) does not produce a 2-kernel in \(\L\), so condition \textup{(DI1)}---in a model, \prox\ref{P:AbCatKernel} or \prox\ref{P:SupKernels}---must be established at the level of the hom-categories, which the truncation does not see.
\end{remark}

\section{Naturality}\label{S:Naturality}

The connecting 1-cell \(\partial\) of \thex\ref{Snake General 2D} is built from three applications of the Pure Snake Lemma, so its behaviour under a morphism of ladders is governed by the behaviour of the comparison of \lemx\ref{Pure Snake Lemma}. Such comparison is unique up to a unique invertible 2-cell; we can thus ask whether it is 2-natural. In this section, we introduce the needed definition of morphism of ladders and prove that 2-naturality indeed holds.

\subsection{Pure configurations and their morphisms}\label{SS:PureConfigurations}

\begin{definition}[pure configuration]\label{Def:PureConfiguration}
	A \dfn{pure configuration} is a morphism of short 2-exact sequences whose middle component is an identity: a pair of short 2-exact sequences
	\[
		A\aar{a}Y\aar{b}C
		\qquad\text{and}\qquad
		X\aar{c}Y\aar{d}Z
	\]
	through a common object \(Y\), together with 1-cells \(f\colon A\to X\) and \(h\colon C\to Z\) and invertible 2-cells
	\[
		\theta_f\colon a\iso c\comp f
		\qquad\text{and}\qquad
		\theta_h\colon h\comp b\iso d\text{.}
	\]
	Its \dfn{dinversion} is \(b\comp c\colon X\to C\).
\end{definition}

This is the shape to which \lemx\ref{Pure Snake Lemma} applies. We write \(\Pi\) for a pure configuration and \(j_\Pi\colon\TwoCoker(f)\to\TwoKer(h)\) for its comparison, characterised by \(\twoker(h)\comp j_\Pi\comp\twocoker(f)\iso b\comp c\).

\begin{definition}[morphism of pure configurations]\label{Def:MorphismPure}
	Let \(\Pi\) and \(\Pi'\) be pure configurations, with data written as above and primed. A \dfn{morphism of pure configurations} \(p\colon\Pi\to\Pi'\) consists of a morphism of short 2-exact sequences \((p_A,p_Y,p_C)\) from the top row of \(\Pi\) to the top row of \(\Pi'\) and a morphism of short 2-exact sequences \((p_X,p_Y,p_Z)\) from the bottom row of \(\Pi\) to the bottom row of \(\Pi'\), \emph{with the same middle component} \(p_Y\). We write
	\[
		\pi_a\colon p_Y\comp a\iso a'\comp p_A\text{,}\quad
		\pi_b\colon p_C\comp b\iso b'\comp p_Y\text{,}\quad
		\pi_c\colon p_Y\comp c\iso c'\comp p_X\text{,}\quad
		\pi_d\colon p_Z\comp d\iso d'\comp p_Y
	\]
	for the four filling 2-cells.
	\begin{cd}[3.5][3]
		0 \&\& A \&\& Y \&\& C \&\& 0 \& \\
		\& 0 \&\& {A'} \&\& {Y'} \&\& {C'} \&\& 0 \\
		0 \&\& X \&\& Y \&\& Z \&\& 0 \\
		\& 0 \&\& {X'} \&\& {Y'} \&\& {Z'} \&\& 0
		\arrow[from=1-1, to=1-3]
		\arrow["a", from=1-3, to=1-5]
		\arrow["{p_A}"{description}, dashed, from=1-3, to=2-4]
		\arrow["f"'{pos=0.4}, from=1-3, to=3-3]
		\arrow["b", from=1-5, to=1-7]
		\arrow["{p_Y}"{description}, dashed, from=1-5, to=2-6]
		\arrow[equal, from=1-5, to=3-5]
		\arrow[from=1-7, to=1-9]
		\arrow["{p_C}"{description}, dashed, from=1-7, to=2-8]
		\arrow["h"{pos=0.4}, from=1-7, to=3-7]
		\arrow[from=2-2, to=2-4]
		\arrow["{a'}"'{pos=0.35}, from=2-4, to=2-6]
		\arrow["{f'}"'{pos=0.35}, from=2-4, to=4-4]
		\arrow["{b'}"'{pos=0.35}, from=2-6, to=2-8]
		\arrow[equal, from=2-6, to=4-6]
		\arrow[from=2-8, to=2-10]
		\arrow["{h'}"{pos=0.35}, from=2-8, to=4-8]
		\arrow[from=3-1, to=3-3]
		\arrow["c"'{pos=0.65}, from=3-3, to=3-5]
		\arrow["{p_X}"{description}, dashed, from=3-3, to=4-4]
		\arrow["d"'{pos=0.65}, from=3-5, to=3-7]
		\arrow["{p_Y}"{description}, dashed, from=3-5, to=4-6]
		\arrow[from=3-7, to=3-9]
		\arrow["{p_Z}"{description}, dashed, from=3-7, to=4-8]
		\arrow[from=4-2, to=4-4]
		\arrow["{c'}"', from=4-4, to=4-6]
		\arrow["{d'}"', from=4-6, to=4-8]
		\arrow[from=4-8, to=4-10]
	\end{cd}
\end{definition}

The definition asks for no compatibility between the two morphisms of short 2-exact sequences beyond sharing \(p_Y\), and it imposes no comparison 2-cells on the verticals \(f\) and \(h\). None is needed: they come for free.

\begin{lemma}\label{L:InducedVerticals}
	A morphism of pure configurations \(p\colon\Pi\to\Pi'\) induces unique invertible 2-cells
	\[
		\pi_f\colon p_X\comp f\iso f'\comp p_A
		\qquad\text{and}\qquad
		\pi_h\colon p_Z\comp h\iso h'\comp p_C
	\]
	compatible with \(\theta_f\), \(\theta_h\) and with the four filling 2-cells, in the sense that
	\[
		c'\star\pi_f=(\theta'_f\star p_A) \cdot \pi_a \cdot (p_Y\star\theta_f)^{-1} \cdot (\pi_c\star f)^{-1}
		\qquad\text{and dually.}
	\]
\end{lemma}
\begin{proof}
	Pasting the available 2-cells gives
	\[
		c'\comp p_X\comp f\iso p_Y\comp c\comp f\iso p_Y\comp a\iso a'\comp p_A\iso c'\comp f'\comp p_A\text{,}
	\]
	using \(\pi_c\), \(\theta_f\), \(\pi_a\) and \(\theta'_f\) in that order. Since \(c'\) is a 2-monomorphism, being a 2-kernel, this invertible 2-cell is reflected, uniquely, to an invertible 2-cell \(\pi_f\colon p_X\comp f\iso f'\comp p_A\), by \remx\ref{rem2monoismonouptoiso}; the displayed equation is the statement that \(\pi_f\) is that reflection, and it determines \(\pi_f\).

	Dually, pasting \(\pi_b\), \(\theta_h\), \(\pi_d\) and \(\theta'_h\) gives
	\[
		h'\comp p_C\comp b\iso h'\comp b'\comp p_Y\iso d'\comp p_Y\iso p_Z\comp d\iso p_Z\comp h\comp b\text{,}
	\]
	and \(b\) is a 2-epimorphism, being a 2-cokernel, so this is coreflected uniquely to \(\pi_h\).
\end{proof}

Consequently \(p\) induces 1-cells on the 2-cokernel of \(f\) and the 2-kernel of \(h\): from \(\pi_f\) we obtain
\[
	\underline{p}_f\colon\TwoCoker(f)\to\TwoCoker(f')
	\qquad\text{with}\qquad
	\underline{p}_f\comp\twocoker(f)\iso\twocoker(f')\comp p_X\text{,}
\]
and from \(\pi_h\) we obtain \(\ov{p}_h\colon\TwoKer(h)\to\TwoKer(h')\) with \(\twoker(h')\comp\ov{p}_h\iso p_C\comp\twoker(h)\).

\subsection{2-naturality of the comparison}\label{SS:NaturalComparison}

\begin{theorem}\label{T:NaturalComparison}
	Let \(p\colon\Pi\to\Pi'\) be a morphism of pure configurations in a homologically self-dual 2-category. Then there is a unique invertible 2-cell
	\[
		\nu_p\colon j_{\Pi'}\comp\underline{p}_f\iso\ov{p}_h\comp j_\Pi
	\]
	compatible with the characterising 2-cells of \(j_\Pi\) and \(j_{\Pi'}\), in the sense that whiskering \(\nu_p\) by \(\twocoker(f)\) on the right and by \(\twoker(h')\) on the left produces the identity 2-cell of \(p_C\comp b\comp c\) under the two pastings displayed in the proof.
\end{theorem}
\begin{proof}
	Whisker the two 1-cells \(j_{\Pi'}\comp\underline{p}_f\) and \(\ov{p}_h\comp j_\Pi\), both from \(\TwoCoker(f)\) to \(\TwoKer(h')\), by the 2-epimorphism \(\twocoker(f)\) on the right and the 2-monomorphism \(\twoker(h')\) on the left. Each of the two resulting 1-cells \(X\to C'\) is canonically isomorphic to \(p_C\comp b\comp c\). Indeed, on the one hand
	\begin{align*}
		\twoker(h')\comp j_{\Pi'}\comp\underline{p}_f\comp\twocoker(f)
		 & \iso\twoker(h')\comp j_{\Pi'}\comp\twocoker(f')\comp p_X \\
		 & \iso b'\comp c'\comp p_X
		\iso b'\comp p_Y\comp c
		\iso p_C\comp b\comp c\text{,}
	\end{align*}
	using in turn the defining 2-cell of \(\underline{p}_f\), the characterisation of \(j_{\Pi'}\), the 2-cell \(\pi_c\) and the 2-cell \(\pi_b\); and on the other hand
	\[
		\twoker(h')\comp\ov{p}_h\comp j_\Pi\comp\twocoker(f)
		\iso p_C\comp\twoker(h)\comp j_\Pi\comp\twocoker(f)
		\iso p_C\comp b\comp c\text{,}
	\]
	using the defining 2-cell of \(\ov{p}_h\) and the characterisation of \(j_\Pi\).

	Composing the first with the inverse of the second gives an invertible 2-cell between the two whiskered 1-cells. Now whiskering with the 2-monomorphism \(\twoker(h')\) on the left and with the 2-epimorphism \(\twocoker(f)\) on the right is fully faithful in each variable, by \defx\ref{Def Mono}, so the 2-cell just constructed is the image of a unique 2-cell \(\nu_p\colon j_{\Pi'}\comp\underline{p}_f\aR{}\ov{p}_h\comp j_\Pi\); and \(\nu_p\) is invertible because its image is. Uniqueness with the stated compatibility is the injectivity half of the same full faithfulness.
\end{proof}

\begin{remark}\label{Rem Natural Comparison Cost}
	The proof uses homological self-duality only to know that \(j_\Pi\) and \(j_{\Pi'}\) exist; that they are equivalences plays no part. What it does use, essentially, is that the comparison is \emph{characterised} by its triangle, which is the content of the uniqueness clause of \lemx\ref{Pure Snake Lemma}.
\end{remark}

\subsection{2-naturality of the snake sequence}\label{SS:NaturalSnake}

\begin{definition}[morphism of ladders]\label{Def:MorphismLadder}
	A \dfn{morphism of ladders} \(p\) between two diagrams of the shape \eqref{Snake Ladder}, as in Figure~\ref{Figure Morphism of diagrams}, consists of 1-cells \(p_A\), \(p_B\), \(p_C\), \(p_X\), \(p_Y\), \(p_Z\) together with invertible 2-cells filling the four horizontal squares and the three vertical ones, subject to the compatibility described in \lemx\ref{L:InducedVerticals} on each of the two ends (that is not automatic in this case).
\end{definition}

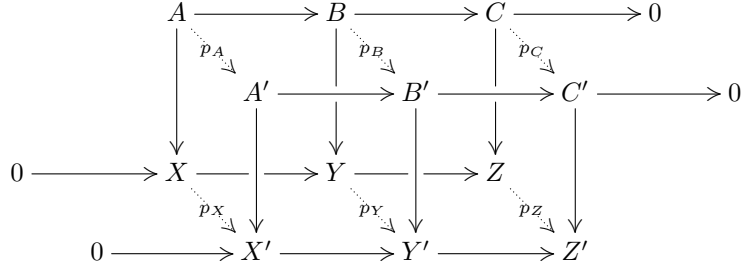
\begin{figure}
	\begin{equation*}
		\xymatrix@!0@=3em{
		&& A \ar@{.>}[rd]|-{{p}_A} \ar[dd] \ar[rr] &&
		B \ar@{.>}[rd]|-{{p}_B} \ar[dd]|(.5){\hole} \ar[rr] &&
		C \ar@{.>}[rd]|-{{p}_C} \ar[dd]|(.5){\hole} \ar[rr] &&
		0 \\
		&&& A' \ar[dd] \ar[rr] &&
		B' \ar[dd] \ar[rr] &&
		C' \ar[dd] \ar[rr] &&
		0 \\
		0 \ar[rr] &&
		X \ar@{.>}[rd]|-{{p}_X} \ar[rr]|(.5){\hole} &&
		Y \ar@{.>}[rd]|-{{p}_Y} \ar[rr]|(.5){\hole} &&
		Z\ar@{.>}[rd]|-{{p}_Z}\\
		&0 \ar[rr] &&
		X' \ar[rr] &&
		Y' \ar[rr] &&
		Z'}
	\end{equation*}
	\caption{A morphism of ladders \({p}\)}\label{Figure Morphism of diagrams}
\end{figure}

\begin{theorem}\label{T:NaturalSnake}
	Let \(p\) be a morphism of ladders between two diagrams satisfying the hypotheses of \thex\ref{Snake General 2D} (The Snake Lemma). Then \(p\) induces 1-cells
	\[
		\ov{p}_A,\ \ov{p}_B,\ \ov{p}_C
		\qquad\text{and}\qquad
		\underline{p}_X,\ \underline{p}_Y,\ \underline{p}_Z
	\]
	on the 2-kernels of \(f\), \(g\), \(h\) and the 2-cokernels of \(f\), \(g\), \(h\) respectively, and the resulting ladder
	\begin{equation*}
		\xymatrix@R=5ex@C=0.9em{
		\TwoKer({f}) \ar@{.>}[d]_-{\ov{p}_A} \ar[r]^-{\ov{a}} &
		\TwoKer({g}) \ar@{.>}[d]_-{\ov{p}_B} \ar[r]^-{\ov{b}} &
		\TwoKer({h}) \ar@{.>}[d]_-{\ov{p}_C} \ar[r]^-{\partial} &
		\TwoCoker({f})\ar@{.>}[d]^-{\underline{p}_X} \ar[r]^-{\underline{c}} &
		\TwoCoker({g}) \ar@{.>}[d]^-{\underline{p}_Y} \ar[r]^-{\underline{d}} &
		\TwoCoker({h}) \ar@{.>}[d]^-{\underline{p}_Z}\\
		\TwoKer({f}') \ar[r]_-{\ov{a}'} &
		\TwoKer({g}') \ar[r]_-{\ov{b}'} &
		\TwoKer({h}') \ar[r]_-{\partial'} &
		\TwoCoker({f}') \ar[r]_-{\underline{c}'} &
		\TwoCoker({g}') \ar[r]_-{\underline{d}'} &
		\TwoCoker({h}')
		}
	\end{equation*}
	commutes up to invertible 2-cells. In particular there is an invertible 2-cell
	\[
		\underline{p}_X\comp\partial\iso\partial'\comp\ov{p}_C\text{.}
	\]
\end{theorem}
\begin{proof}
	The four squares not involving \(\partial\) are immediate. For the first, both \(\ov{p}_B\comp\ov{a}\) and \(\ov{a}'\comp\ov{p}_A\) become the same 1-cell after composition with the 2-monomorphism \(\twoker(g')\): the defining 2-cells of the four induced 1-cells and the 2-cell filling the square \(p_B\comp a\iso a'\comp p_A\) paste to
	\begin{align*}
		\twoker(g')\comp\ov{p}_B\comp\ov{a}
		 & \iso p_B\comp\twoker(g)\comp\ov{a}
		\iso p_B\comp a\comp\twoker(f)        \\
		 & \iso a'\comp p_A\comp\twoker(f)
		\iso a'\comp\twoker(f')\comp\ov{p}_A
		\iso\twoker(g')\comp\ov{a}'\comp\ov{p}_A\text{,}
	\end{align*}
	and \remx\ref{rem2monoismonouptoiso} reflects this to an invertible 2-cell \(\ov{p}_B\comp\ov{a}\iso\ov{a}'\comp\ov{p}_A\). The second square is the same argument with \(b\) in place of \(a\), and the last two are dual.

	For the square involving \(\partial\), recall from \eqref{Def Partial} that \(\partial=\twoker(\underline{c})\comp z\comp\twocoker(\ov{b})\), with \(z=z_3\comp z_2^{-1}\comp z_1\) the composite of the comparisons of the three pure configurations of Section~\ref{SS:Construction}. The morphism \(p\) induces a morphism of pure configurations between each of those three and its primed counterpart. For the first, the components are: on the top row, the 1-cell \(I\to I'\) obtained from the uniqueness of the normal image factorisation of \(b\comp\twoker(g)\), by \prox\ref{Image Factorisation of Normal Map is Unique}, together with the induced 1-cell \(Q\to Q'\) on 2-cokernels; on the bottom row, \(\ov{p}_C\) and the induced 1-cell \(\TwoImg(h)\to\TwoImg(h')\); the middle component is \(p_C\) in both. The other two configurations are handled in the same way, the third using in addition the induced 1-cells on \(\TwoImg(f)\) and \(\TwoImg(g)\).

	\thex\ref{T:NaturalComparison} applied to each of the three yields invertible 2-cells relating \(z_1\), \(z_2\), \(z_3\) to their primed counterparts, and hence, inverting the one for \(z_2\), an invertible 2-cell relating \(z\) to \(z'\). What makes the three paste is that consecutive configurations induce the same 1-cell where they meet: the first and the second both induce a 1-cell on \(\TwoKer(t)\), the second and the third both induce one on \(\TwoCoker(s)\), and in each case the two agree, being induced by the same universal property. Since \(\twocoker(\ov{b})\) and \(\twoker(\underline{c})\) are likewise 2-natural---being the 1-cells induced on a 2-cokernel and on a 2-kernel, so that the argument of the first paragraph applies---pasting the three squares gives the invertible 2-cell \(\underline{p}_X\comp\partial\iso\partial'\comp\ov{p}_C\).
\end{proof}

\begin{remark}\label{Rem Naturality Choices}
	By \remx\ref{Rem Partial Choices} the 1-cell \(\partial\) is determined only up to an invertible 2-cell, so \thex\ref{T:NaturalSnake} is a statement about any choice of \(\partial\) and \(\partial'\): transporting along those invertible 2-cells carries the naturality 2-cell of one choice to that of another.
\end{remark}

\section{Two-dimensional models}\label{S:Model}

By \exax\ref{Ex DiExact}, any abelian category is a locally discrete 2-di-exact 2-category. Thus the theory developed in the previous sections provides the Snake Lemma there; but this is not news. In this section, we investigate models that are genuinely 2-dimensional. The first candidate that suggests itself is the 2-category \(\AbCat\) of abelian categories, exact functors and natural transformations, whose 2-kernels are the Serre subcategories and whose 2-cokernels are the Serre quotients; but this 2-category is \emph{not} 2-di-exact. Condition \textup{(DI2)} reduces there to a question with no two-dimensional content left in it---whether the essential image of a Serre subcategory in a Serre quotient is again a Serre subcategory---and the answer is no, already for the finite-dimensional modules over a Nakayama algebra with three composition factors.

What does work is to turn that question into a hypothesis. Call an abelian category \emph{saturated} when the answer is yes for all of its Serre subcategories. The point of \prox\ref{P:SATclosed} below is that this costs almost nothing: saturation is inherited by Serre subcategories and by Serre quotients, since each of the two constructions comes with a correspondence theorem for Serre subcategories. Taking the obstruction identified above as a hypothesis, that of saturation, will yield a 2-di-exact 2-category (\thex\ref{T:SatModel}). That this 2-category is populated is the content of Sections~\ref{SS:ModelAS} and~\ref{SS:ModelExamples}: saturation follows from a condition on subobjects which is a statement about supports in disguise, and that condition holds for the finitely generated modules over a commutative Noetherian ring and for the coherent sheaves on a Noetherian scheme. Section~\ref{SS:ModelLattices} then exhibits a second model at the opposite end of the spectrum, where the hom-categories are posets: the complete modular lattices. There condition \textup{(DI2)} turns out to be a familiar face, Dedekind's transposition principle \cite[Chapter~I, Theorem~13]{Birkhoff}.

\begin{convention}\label{Conv Size}
	Fix a Grothendieck universe and call a category \dfn{small} when it is equivalent to one whose objects and morphisms form a set in that universe. All abelian categories below are assumed small in this sense, so that they, the exact functors and the natural transformations between them form a 2-category \(\AbCat\). This is a size convention; it is preserved by every construction performed below, and no argument uses it otherwise.
\end{convention}

For the theory of Serre subcategories and Serre quotients, we refer the reader to \cite[Chapter~III]{Gabriel62}. A \dfn{Serre subcategory} of an abelian category \(\onecat{A}\) is a full subcategory closed under subobjects, quotients and extensions; it is again an abelian category, and its inclusion is exact. The \dfn{Serre quotient} \(\onecat{A}/S\) by a Serre subcategory \(S\) has the objects of \(\onecat{A}\) and
\begin{eqD}{SerreHom}
	\Hom_{\onecat{A}/S}(\ov{a},\ov{b})=\varinjlim\Hom_{\onecat{A}}(a',b/b'')\text{,}
\end{eqD}
the colimit being taken over the subobjects \(a'\subseteq a\) with \(a/a'\in S\) and the subobjects \(b''\subseteq b\) with \(b''\in S\). It is an abelian category; the projection \(q\colon\onecat{A}\to\onecat{A}/S\) is exact, is the identity on objects and annihilates exactly \(S\); and composition with \(q\) is an equivalence from the category of exact functors out of \(\onecat{A}/S\) onto the category of exact functors out of \(\onecat{A}\) that annihilate \(S\). We use \dfn{Gabriel's correspondence} in the form: taking preimages along \(q\) is a bijection from the Serre subcategories of \(\onecat{A}/S\) onto the Serre subcategories of \(\onecat{A}\) containing \(S\), with inverse the essential image.

\subsection{In \texorpdfstring{\(\AbCat\)}{AbCat}, 2-kernels are the Serre subcategories and 2-cokernels are the Serre quotients}\label{SS:ModelAbCat}

We recall from \cite[Sections 2 and 7]{CavJanMesRei26} the investigation of the 2-kernels and 2-cokernels in \(\AbCat\).

\begin{proposition}[{\cite{CavJanMesRei26}}]\label{P:AbCatBizero}
	The abelian category with one object is a strong bizero object of \(\AbCat\).
\end{proposition}

\begin{proposition}[{\cite{CavJanMesRei26}}]\label{P:AbCatKernel}
	Let \(f\colon\onecat{A}\to\onecat{B}\) be an exact functor and let \(K\) be the full subcategory of \(\onecat{A}\) on the objects that \(f\) annihilates. Then \(K\) is a Serre subcategory of \(\onecat{A}\) and its inclusion \(k\colon K\to\onecat{A}\) is a 2-kernel of \(f\).
\end{proposition}

\begin{proposition}[{\cite{CavJanMesRei26}}]\label{P:AbCatCokernel}
	Let \(f\colon\onecat{A}\to\onecat{B}\) be an exact functor and let \(S\) be the smallest Serre subcategory of \(\onecat{B}\) containing the objects \(f(a)\). Then the projection \(q\colon\onecat{B}\to\onecat{B}/S\) is a 2-cokernel of~\(f\).
\end{proposition}

\begin{proposition}[{\cite{CavJanMesRei26}}]\label{C:AbCatNormal}
	In \(\AbCat\), every normal 2-monomorphism is isomorphic to the inclusion of a Serre subcategory precomposed with an equivalence, and every normal 2-epimorphism is isomorphic to the projection onto a Serre quotient postcomposed with an equivalence.
\end{proposition}

\subsection{The saturation, and the failure of \textup{(DI2)}}\label{SS:ModelSaturation}

\begin{definition}\label{D:Saturation}
	Let \(K\) and \(S\) be Serre subcategories of an abelian category \(\onecat{A}\). The \dfn{\(S\)-saturation} of \(K\), written \(K^{S}\), is the class of objects of \(\onecat{A}\) that become isomorphic in \(\onecat{A}/S\) to an object of \(K\); equivalently, the class of objects joined to an object of \(K\) by a span of morphisms all of whose kernels and cokernels lie in \(S\).
\end{definition}

\begin{proposition}\label{P:Saturation}
	The class \(K^{S}\) contains \(K\) and \(S\), is closed under subobjects and under quotients, and is contained in every Serre subcategory of \(\onecat{A}\) containing \(K\) and \(S\). It is therefore a Serre subcategory if and only if it is the join \(K\vee S\).
\end{proposition}

\begin{proof}
	It contains \(K\) trivially, and it contains \(S\) because \(q\) sends an object of \(S\) to a zero object, which is \(q\) of the zero object of \(K\). Let \(X\in K^{S}\), say \(q(X)\iso q(a)\) with \(a\in K\), and let \(X'\subseteq X\). Then \(q(X')\) is a subobject of \(q(X)\iso q(a)\), and every subobject of \(q(a)\) is of the form \(q(a')\) for a subobject \(a'\subseteq a\); since \(K\) is closed under subobjects, \(a'\in K\) and \(X'\in K^{S}\). The argument for quotients is the same.

	Let \(T\) be a Serre subcategory containing \(K\) and \(S\), and let \(X\aar{g}D\aar{f}a\) be a span with \(a\in K\) and with the kernels and cokernels of \(f\) and of \(g\) in \(S\). The image of \(f\) is a subobject of \(a\), hence lies in \(T\), and the kernel of \(f\) lies in \(S\subseteq T\); so \(D\), an extension of the one by the other, lies in \(T\). Running the same argument along \(g\) puts \(X\) in \(T\). Finally, \(K\vee S\) is the smallest Serre subcategory containing \(K\) and \(S\), so that \(K^{S}\subseteq K\vee S\) always, with equality precisely when \(K^{S}\) is a Serre subcategory.
\end{proof}

\begin{lemma}\label{L:FullyFaithful}
	Let \(K\) and \(S\) be Serre subcategories of an abelian category \(\onecat{A}\). Then the exact functor \(K/(K\cap S)\to\onecat{A}/S\) induced by the inclusion of \(K\) is fully faithful.
\end{lemma}

\begin{proof}
	For \(a\) and \(b\) in \(K\), the hom-set \(\Hom_{\onecat{A}/S}(\ov{a},\ov{b})\) is the filtered colimit in \eqref{SerreHom}, and \(\Hom_{K/(K\cap S)}(\ov{a},\ov{b})\) is the same colimit taken inside \(K\). Every subobject and every quotient of \(a\) or of \(b\) in \(\onecat{A}\) already lies in \(K\), that subcategory being closed under subobjects and quotients, and such a subobject lies in \(S\) if and only if it lies in \(K\cap S\); so the two index categories coincide. The terms coincide as well, \(K\) being full. Hence the two colimits agree.
\end{proof}

The next proposition will help us prove that \(\AbCat\) does \emph{not} satisfy the condition \textup{(DI2)}; later on, we may reuse its proof in a positive result on a sub-2-category of \(\AbCat\).

\begin{proposition}\label{P:DIabcat}
	Condition \textup{(DI2)} holds in \(\AbCat\) if and only if for every abelian category \(\onecat{A}\) and all Serre subcategories \(K\) and \(S\) of \(\onecat{A}\), the \(S\)-saturation \(K^{S}\) is a Serre subcategory of \(\onecat{A}\).
\end{proposition}

\begin{proof}
	An antinormal 1-cell of \(\AbCat\) is a composite \(e\comp m\) with \(m\) a normal 2-monomorphism into some abelian category \(\onecat{A}\) and \(e\) a normal 2-epimorphism out of it. By \prox\ref{C:AbCatNormal} we may write \(m\iso k\comp u\) and \(e\iso v\comp q\), where \(k\colon K\to\onecat{A}\) is the inclusion of a Serre subcategory, \(q\colon\onecat{A}\to\onecat{A}/S\) is the projection onto a Serre quotient, and \(u\) and \(v\) are equivalences; so \(e\comp m\iso v\comp w\comp u\) with \(w\coloneq q\comp k\), and by \corx\ref{C:NormalTransport} the 1-cell \(e\comp m\) is normal if and only if \(w\) is. By \defx\ref{D:Saturation} the essential image of \(w\) is \(q(K^{S})\), and \(K^{S}\) contains \(S\) by \prox\ref{P:Saturation}; so Gabriel's correspondence makes \(q(K^{S})\) a Serre subcategory of \(\onecat{A}/S\) if and only if \(K^{S}\) is one of \(\onecat{A}\).

	Suppose \(K^{S}\) is a Serre subcategory. Then \(w\) factorises as
	\begin{eqD*}
		K\toe K/(K\cap S)\aar{c}\onecat{A}/S\text{,}
	\end{eqD*}
	whose first factor is the 2-cokernel of the inclusion \(K\cap S\to K\), hence a normal 2-epimorphism. The second factor \(c\) is fully faithful by \lemx\ref{L:FullyFaithful}, hence a 2-monomorphism, and its essential image is the Serre subcategory \(q(K^{S})\); so \(c\) is an equivalence onto that subcategory, whose inclusion is a normal 2-monomorphism by \prox\ref{P:AbCatKernel}, and \corx\ref{C:NormalTransport} makes \(c\) a normal 2-monomorphism. Hence \(w\) is normal.

	Conversely, suppose \(w\iso m'\comp e'\) with \(e'\) a normal 2-epimorphism and \(m'\) a normal 2-monomorphism. By \prox\ref{C:AbCatNormal} the functor \(e'\) is a Serre quotient projection composed with an equivalence, hence essentially surjective, and \(m'\) is the inclusion of a Serre subcategory \(T\) of \(\onecat{A}/S\) composed with an equivalence. The essential image of \(w\) is then that of \(m'\), which is \(T\); so \(q(K^{S})=T\) is a Serre subcategory of \(\onecat{A}/S\), and \(K^{S}\) is one of \(\onecat{A}\).
\end{proof}

Closure of \(K^{S}\) under subobjects and under quotients came for free, so only closure under extensions is at stake---and localisation is exactly what can destroy it.

\begin{proposition}\label{P:AbCatFails}
	The 2-category \(\AbCat\) is not 2-di-exact.
\end{proposition}

\begin{proof}
	Let \(\Lambda\) be the Nakayama algebra over a field \(\Bbbk\) on the cyclic quiver \(1\to2\to1\) with \(\operatorname{rad}^{3}=0\), let \(\onecat{A}\) be its category of finite-dimensional modules and let \(S_{1}\), \(S_{2}\) be the two simple modules. Put
	\begin{align*}
		K & \coloneq\{\text{modules all of whose composition factors are \(\iso S_{1}\)}\}\text{,} \\
		S & \coloneq\{\text{modules all of whose composition factors are \(\iso S_{2}\)}\}\text{,}
	\end{align*}
	both plainly Serre subcategories. Let \(M\) be the uniserial module of length~3 with top \(S_{1}\), middle \(S_{2}\) and socle \(S_{1}\)---that is, the indecomposable projective \(\Lambda e_{1}\)---and let \(U\subset M\) be its submodule of length~2, with socle \(S_{1}\) and top \(S_{2}\).

	Then \(U\in K^{S}\): the inclusion \(S_{1}\subseteq U\) has zero kernel and cokernel \(S_{2}\in S\), so \(\ov{U}\iso\ov{S_{1}}\) with \(S_{1}\in K\). Also \(M/U\iso S_{1}\) lies in \(K\). But \(M\notin K^{S}\), and this is a computation of hom-sets rather than an inspection of filtrations. The only submodule of \(M\) lying in \(S\) is \(0\), the only submodule \(M'\subseteq M\) with \(M/M'\in S\) is \(M\), and the only submodule \(X'\subseteq S_{1}\) with \(S_{1}/X'\in S\) is \(S_{1}\); so the colimits \eqref{SerreHom} collapse to a single term:
	\begin{eqD*}
		\Hom_{\onecat{A}/S}(\ov{S_{1}},\ov{M})=\Hom_{\onecat{A}}(S_{1},M)=\Bbbk\text{,}\qquad\Hom_{\onecat{A}/S}(\ov{S_{1}},\ov{S_{1}}^{n})=\Bbbk^{n}\text{.}
	\end{eqD*}
	Hence \(\ov{M}\iso\ov{S_{1}}^{n}\) would force \(n=1\); but \(\End_{\onecat{A}/S}(\ov{M})=\End_{\Lambda}(M)=e_{1}\Lambda e_{1}\) has dimension~2, being spanned by \(e_{1}\) and the path of length~2, whereas \(\End(\ov{S_{1}})=\Bbbk\). So \(\ov{M}\) is isomorphic to no object of \(K\). The short exact sequence \(0\to U\to M\to M/U\to 0\) therefore has both of its outer terms in \(K^{S}\) and its middle term outside, so \(K^{S}\) is not closed under extensions, and \prox\ref{P:DIabcat} applies.
\end{proof}

\begin{remark}\label{Rem Ext}
	The mechanism is that localisation creates \(\Ext\). Indeed \(\Ext^{1}_{\onecat{A}}(S_{1},S_{1})=0\) in the example above, the quiver having no loop at the first vertex; the self-extension of \(\ov{S_{1}}\) witnessed by \(\ov{M}\) comes into being only once \(S_{2}\) has been killed, as the Yoneda composite of the two extensions between \(S_{1}\) and \(S_{2}\). Conversely, if for all \(a\), \(b\) in \(K\) the comparison
	\begin{eqD*}
		\Ext^{1}_{\onecat{A}}(b,a)\too\Ext^{1}_{\onecat{A}/S}(q(b),q(a))
	\end{eqD*}
	is surjective, then \(K^{S}\) is a Serre subcategory: an extension of \(q(b)\) by \(q(a)\) in \(\onecat{A}/S\) is the image of one in \(\onecat{A}\), whose middle term lies in \(K\). That happens for flat localisations, and whenever \(q\) admits a fully faithful right adjoint, so that the quotient sits inside \(\onecat{A}\) as a full subcategory and can manufacture nothing. That \(\AbCat\) is not 2-di-exact is proved independently in \cite[Section 2]{CavJanMesRei26}, by a Serre quotient of the finite-dimensional representations of the quiver \(1\to2\to3\); the obstruction isolated there \cite[Section 2]{CavJanMesRei26}---that the essential image of the quotient need not be closed under extensions---is the same one.
\end{remark}

\subsection{Saturated abelian categories}\label{SS:ModelSAT}

\begin{definition}\label{D:SAT}
	An abelian category \(\onecat{A}\) is \dfn{saturated} when for all Serre subcategories \(K\) and \(S\) of \(\onecat{A}\) the \(S\)-saturation \(K^{S}\) is again a Serre subcategory of \(\onecat{A}\). We write \(\Sat\) for the full sub-2-category of \(\AbCat\) on the saturated abelian categories.
\end{definition}

\begin{proposition}\label{P:SATclosed}
	Saturation is inherited by Serre subcategories and by Serre quotients.
\end{proposition}

\begin{proof}
	Let \(\onecat{A}\) be saturated and let \(\onecat{B}\) be a Serre subcategory of it. The Serre subcategories of \(\onecat{B}\) are exactly the Serre subcategories of \(\onecat{A}\) contained in \(\onecat{B}\), that subcategory being closed in \(\onecat{A}\) under subobjects, quotients and extensions; let \(K\) and \(S\) be two of them. By \lemx\ref{L:FullyFaithful}, applied to \(\onecat{B}\) and \(S\) inside \(\onecat{A}\), the induced functor \(\onecat{B}/S\to\onecat{A}/S\) is fully faithful, and a fully faithful functor reflects isomorphisms; so an object of \(\onecat{B}\) becomes isomorphic to an object of \(K\) in \(\onecat{B}/S\) exactly when it does in \(\onecat{A}/S\). The \(S\)-saturation of \(K\) computed in \(\onecat{B}\) is therefore the intersection with \(\onecat{B}\) of the one computed in \(\onecat{A}\). The latter is a Serre subcategory of \(\onecat{A}\) by saturation, and it is contained in \(K\vee S\) by \prox\ref{P:Saturation}, hence in \(\onecat{B}\); so the two coincide, and being a Serre subcategory of \(\onecat{A}\) contained in \(\onecat{B}\), it is a Serre subcategory of \(\onecat{B}\).

	Let now \(S\) be a Serre subcategory of \(\onecat{A}\), write \(q\colon\onecat{A}\to\onecat{A}/S\), and let \(K'\) and \(T'\) be Serre subcategories of \(\onecat{A}/S\). Put \(K\coloneq q^{-1}(K')\) and \(T\coloneq q^{-1}(T')\), which are Serre subcategories of \(\onecat{A}\) containing \(S\), the functor \(q\) being exact and annihilating \(S\). By saturation the \(T\)-saturation \(K^{T}\) is a Serre subcategory of \(\onecat{A}\); it contains \(T\) and hence \(S\), so by Gabriel's correspondence \(q(K^{T})\) is a Serre subcategory of \(\onecat{A}/S\). It contains \(K'\) and \(T'\), hence their join. It is also contained in \((K')^{T'}\): applying the exact functor \(q\) to a span witnessing membership in \(K^{T}\) produces a span witnessing membership in \((K')^{T'}\), the kernels and cokernels of the images being the images of the kernels and cokernels. So \((K')^{T'}\) contains \(K'\vee T'\); it is contained in that join by \prox\ref{P:Saturation}, and being equal to it, it is a Serre subcategory.
\end{proof}

\begin{theorem}\label{T:SatModel}
	The 2-category \(\Sat\) of saturated abelian categories, exact functors and natural transformations is 2-di-exact.
\end{theorem}

\begin{proof}
	The abelian category with one object is saturated, and it is a strong bizero object of \(\Sat\) by \prox\ref{P:AbCatBizero}, whose proof takes place inside any full sub-2-category of \(\AbCat\) containing it.

	Condition \textup{(DI1)}. Let \(f\colon\onecat{A}\to\onecat{B}\) be an exact functor between saturated abelian categories. Its 2-kernel in \(\AbCat\) is the inclusion of the Serre subcategory \(K\) of the objects that \(f\) annihilates, by \prox\ref{P:AbCatKernel}, and its 2-cokernel is the projection onto \(\onecat{B}/S\) for a Serre subcategory \(S\) of \(\onecat{B}\), by \prox\ref{P:AbCatCokernel}. Both \(K\) and \(\onecat{B}/S\) are saturated by \prox\ref{P:SATclosed}; and since \(\Sat\) is a full sub-2-category of \(\AbCat\), a 2-kernel or a 2-cokernel computed in \(\AbCat\) whose vertex lies in \(\Sat\) is one computed in \(\Sat\). The same remark makes \prox\ref{C:AbCatNormal} available in \(\Sat\).

	Condition \textup{(DI2)}. By the reduction carried out at the start of the proof of \prox\ref{P:DIabcat} it suffices to show that the composite \(w=q\comp k\) of the inclusion of a Serre subcategory \(K\) into a saturated abelian category \(\onecat{A}\) with the projection onto a Serre quotient \(\onecat{A}/S\) is normal. This is the first half of the proof of \prox\ref{P:DIabcat}, whose hypothesis that \(K^{S}\) be a Serre subcategory is now supplied by saturation; and the intermediate category \(K/(K\cap S)\), being a Serre quotient of a Serre subcategory of \(\onecat{A}\), is itself saturated by \prox\ref{P:SATclosed}, so that the factorisation obtained there lies in \(\Sat\).
\end{proof}

Applying \prox\ref{P:DiExactHSD} we then immediately obtain the following corollary.

\begin{corollary}\label{C:SatHSD}
	The 2-category \(\Sat\) is homologically self-dual, so everything proved in Sections~\ref{S:Exact} to~\ref{S:Naturality} holds in it: the Pure Snake Lemma, the Snake Lemma of \thex\ref{Snake General 2D}, and its 2-naturality.\noproof
\end{corollary}

\subsection{A criterion on subobjects}\label{SS:ModelAS}

We now show that the 2-category of saturated abelian categories is nicely inhabited. For that, we show that the three-step filtration implicit in \defx\ref{D:Saturation} can be shortened.

\begin{proposition}\label{P:TwoStep}
	Let \(K\) and \(S\) be Serre subcategories of an abelian category \(\onecat{A}\). If every object of \(K\vee S\) admits a subobject in \(S\) with quotient in \(K\), or a subobject in \(K\) with quotient in \(S\), then \(K^{S}\) is a Serre subcategory.
\end{proposition}

\begin{proof}
	Let \(M\) be an object of \(K\vee S\) and let \(M_{1}\subseteq M\) be a subobject as in the hypothesis. In the first case the projection \(M\toe M/M_{1}\) has kernel \(M_{1}\in S\) and zero cokernel, so that \(\ov{M}\iso\ov{M/M_{1}}\) with \(M/M_{1}\in K\); in the second the inclusion \(M_{1}\to M\) has zero kernel and cokernel \(M/M_{1}\in S\), so that \(\ov{M}\iso\ov{M_{1}}\) with \(M_{1}\in K\). Either way \(M\in K^{S}\), so \(K\vee S\subseteq K^{S}\); the reverse containment is \prox\ref{P:Saturation}.
\end{proof}

Call an abelian category \dfn{Noetherian} when each of its objects satisfies the ascending chain condition on subobjects.

\begin{definition}\label{D:AS}
	An abelian category \(\onecat{A}\) satisfies \(\AS\) when for every Serre subcategory \(T\) of \(\onecat{A}\) and every object \(X\): if every nonzero subobject of \(X\) has a nonzero subobject lying in \(T\), then \(X\) lies in \(T\).
\end{definition}

The converse implication holds in every abelian category, a Serre subcategory being closed under subobjects, so \(\AS\) asserts that the two properties of \(X\) coincide. It is a genuine restriction. In the category of all \(\mathbb{Z}\)-modules the finitely generated ones form a Serre subcategory \(T\), and every nonzero submodule of \(\mathbb{Q}\) contains a copy of \(\mathbb{Z}\) and so meets \(T\), while \(\mathbb{Q}\) itself does not lie in \(T\). The Noetherian hypothesis that appears in the next proposition solves precisely what fails here.

\begin{proposition}\label{P:ASimplies}
	A Noetherian abelian category satisfying \(\AS\) is saturated.
\end{proposition}

\begin{proof}
	Let \(K\) and \(S\) be Serre subcategories and let \(M\) lie in \(K\vee S\). By \prox\ref{P:TwoStep} it is enough to present \(M\) as an extension of an object of \(K\) by an object of \(S\).

	Among the subobjects of \(M\) that lie in \(S\), choose one that is maximal; one can do this since the zero subobject is available and \(M\) is Noetherian. Call the chosen maximal subobject \(M_{1}\) and put \(N\coloneq M/M_{1}\). Then \(N\) has no nonzero subobject lying in \(S\). Indeed, the preimage in \(M\) of such a subobject is an extension of an object of \(S\) by \(M_{1}\), hence lies in \(S\), and it contains \(M_{1}\) strictly.

	Every nonzero subobject \(W\) of \(N\) has a nonzero subobject lying in \(K\). Indeed \(W\), being a subobject of a quotient of \(M\), lies in \(K\vee S\); and \(K\vee S\) is the class of objects carrying a finite filtration whose successive subquotients lie in \(K\) or in \(S\), that class being closed under subobjects, quotients and extensions. The first nonzero step of such a filtration on \(W\) is a nonzero subobject of \(W\) lying in \(K\) or in \(S\), and the second case is excluded, since it would produce a nonzero subobject of \(N\) lying in \(S\).

	By \(\AS\), applied with \(T=K\) to the object \(N\), we conclude that \(N\) lies in \(K\). The epimorphism \(M\toe N\) has kernel \(M_{1}\) in \(S\), which is the presentation \prox\ref{P:TwoStep} asks for.
\end{proof}

\subsection{Modules and coherent sheaves}\label{SS:ModelExamples}

\begin{proposition}\label{P:ModAS}
	Let \(R\) be a commutative Noetherian ring. Then the category of finitely generated \(R\)-modules is a Noetherian abelian category satisfying \(\AS\).
\end{proposition}

\begin{proof}
	The only part that needs a proof is \(\AS\). Let \(T\) be a Serre subcategory and let \(M\) be a module every nonzero submodule of which has a nonzero submodule in \(T\).

	Let \(p\) be an associated prime of \(M\), so that \(R/p\) embeds in \(M\). By hypothesis \(R/p\) has a nonzero submodule \(J/p\) lying in \(T\); choose a nonzero \(x\) in \(J/p\). Multiplication by \(x\) is injective, \(R/p\) being a domain, so \(R/p\) is isomorphic to the submodule \(x\cdot(R/p)\) of \(J/p\) and therefore lies in \(T\).

	Let now \(p\) be any prime in \(\Supp M\). Then \(M_{p}\neq0\), so \(M_{p}\) has an associated prime, and pulling it back along \(R\to R_{p}\) gives an associated prime \(q\) of \(M\) with \(q\subseteq p\). By the previous paragraph \(R/q\) lies in \(T\); and \(R/p\) is a quotient of \(R/q\), so it lies in \(T\) as well.

	Finally, \(M\) carries a filtration \(0=M_{0}\subset M_{1}\subset\dots\subset M_{n}=M\) whose successive quotients are of the form \(R/p_{i}\) for primes \(p_{i}\) \cite[Theorem~6.4]{Mat89}. Each \(R/p_{i}\) is a subquotient of \(M\), so that \(\Supp(R/p_{i})\subseteq\Supp M\) and in particular \(p_{i}\in\Supp M\). By the previous paragraph every \(R/p_{i}\) lies in \(T\), and \(T\) is closed under extensions, so \(M\) lies in \(T\).
\end{proof}

For coherent sheaves one classical fact has to be added, since \(\AS\) is a hypothesis about subobjects while the associated points of a sheaf, unlike the submodule \(R/p\) above, are visible only locally. The next lemma may also be seen as a consequence of \cite[Proposition~(3.2.6)]{EGAIV2}.

\begin{lemma}\label{L:OneAss}
	Let \(X\) be a Noetherian scheme, let \(\mathcal{F}\) be a coherent sheaf on \(X\) and let \(x\in\Ass\mathcal{F}\). Then \(\mathcal{F}\) has a coherent subsheaf \(\mathcal{G}\) with \(\Ass\mathcal{G}=\{x\}\).
\end{lemma}

\begin{proof}
	Write \(W=\ov{\{x\}}\) and let \(\mathcal{G}_{0}\subseteq\mathcal{F}\) be the subsheaf of sections supported in \(W\), which is coherent because \(X\) is Noetherian; it has \(x\) as an associated point and its support lies in \(W\). Let \(W'\) be the union of the closures of the associated points of \(\mathcal{G}_{0}\) other than \(x\). Each of these is a proper closed subset of \(W\), because \(x\) is the generic point of \(W\), so \(W'\) is a closed subset of \(X\) not containing \(x\). Let \(\mathcal{J}\) be the ideal sheaf of \(W'\) and let \(\mathcal{A}\subseteq\mathcal{G}_{0}\) be the subsheaf of sections supported in \(W'\); being coherent and supported in \(W'\), it is annihilated by a power of \(\mathcal{J}\). By the Artin--Rees lemma, applied on each member of a finite affine cover, \(\mathcal{J}^{n}\mathcal{G}_{0}\cap\mathcal{A}=0\) for \(n\) large; put \(\mathcal{G}=\mathcal{J}^{n}\mathcal{G}_{0}\) for such an \(n\). Then \(\mathcal{G}\) is nonzero, since \(\mathcal{J}_{x}=\mathcal{O}_{X,x}\), and it has no nonzero section supported in \(W'\). Its associated points are associated points of \(\mathcal{G}_{0}\), and one of them other than \(x\) would lie in \(W'\) and would give such a section.
\end{proof}

\begin{proposition}\label{P:CohAS}
	For a Noetherian scheme \(X\), the category \(\Coh(X)\) of coherent sheaves is a Noetherian abelian category satisfying \(\AS\).
\end{proposition}

\begin{proof}
	Only \(\AS\) is at issue, the rest is well known. By Gabriel's classification \cite[Proposition~VI.2.4]{Gabriel62} the Serre subcategories of \(\Coh(X)\) are the \(T_{Z}=\{\mathcal{F}\mid\Supp\mathcal{F}\subseteq Z\}\) for \(Z\subseteq X\) specialisation-closed. Let \(\mathcal{F}\) be coherent and suppose that every nonzero subsheaf of \(\mathcal{F}\) has a nonzero subsheaf lying in \(T_{Z}\).

	Let \(x\in\Ass\mathcal{F}\) and take \(\mathcal{G}\subseteq\mathcal{F}\) as in \lemx\ref{L:OneAss}. By hypothesis \(\mathcal{G}\) has a nonzero subsheaf \(\mathcal{H}\) with \(\Supp\mathcal{H}\subseteq Z\). Now \(\Ass\mathcal{H}\) is nonempty and contained in \(\Ass\mathcal{G}=\{x\}\), so \(x\in\Supp\mathcal{H}\subseteq Z\). Hence \(\Ass\mathcal{F}\subseteq Z\); and since \(\Supp\mathcal{F}\) is the union of the closures of the points of \(\Ass\mathcal{F}\) \cite[Tag~05AL]{Stacks} and \(Z\) is specialisation-closed, \(\Supp\mathcal{F}\subseteq Z\), which is to say that \(\mathcal{F}\) lies in \(T_{Z}\).
\end{proof}

\begin{corollary}\label{C:Populated}
	The 2-category \(\Sat\) contains the category of finitely generated modules over every commutative Noetherian ring, and the category \(\Coh(X)\) of coherent sheaves on every Noetherian scheme \(X\).
\end{corollary}

\begin{proof}
	This is \prox\ref{P:ModAS} and \prox\ref{P:CohAS}, each followed by \prox\ref{P:ASimplies}.
\end{proof}

\begin{remark}\label{Rem WhereNot}
	Condition \(\AS\) is what the category of \prox\ref{P:AbCatFails} violates. Take for \(T\) the Serre subcategory \(K\) of the \(S_{1}\)-modules and for \(X\) the module \(M=\Lambda e_{1}\). Every nonzero submodule of \(M\) contains the socle, which is \(S_{1}\) and lies in \(K\); but \(M\) does not lie in \(K\), having \(S_{2}\) among its composition factors. More generally, in a category of finite length the condition forces every simple subquotient of an object to be a subobject of it.
\end{remark}

\begin{remark}\label{Rem Model Shape}
	We have proved that \(\Sat\) not only solves the issue preventing \(\AbCat\) from being 2-di-exact but contains many interesting abelian categories. The saturation property, which precisely captures the obstruction to \(\AbCat\) being 2-di-exact is satisfied by many of the abelian categories that are studied and used.
\end{remark}

\subsection{A locally ordered model: modular lattices}\label{SS:ModelLattices}

We now exhibit another model, of a different kind. Its hom-categories will be partially ordered sets. Call a 2-category \dfn{locally ordered} when between two parallel 1-cells there is at most one 2-cell, and parallel 1-cells carrying 2-cells in both directions are equal. Note that every locally discrete 2-category is locally ordered, but not conversely.

Write \(\Sup\) for the following locally ordered 2-category. Objects are the complete lattices, small with respect to the universe of Convention~\ref{Conv Size}; 1-cells \(L\to M\) are the maps preserving all joins, the empty one \(\bot\) included; and there is a 2-cell \(f\aR{}g\) exactly when \(f\leq g\) pointwise. A monotone map between complete lattices preserves all joins if and only if it is a left adjoint, so \(\Sup\) is equivalently the 2-category of complete lattices and Galois connections.

\begin{remark}\label{Rem Locally Ordered}
	In a locally ordered 2-category the two-dimensional layer of the theory simplifies. An invertible 2-cell \(f\iso g\) consists of \(f\leq g\) and \(g\leq f\), so it is an identity; hence isomorphic 1-cells are equal, a factorisation up to an invertible 2-cell is a factorisation on the nose, and an equivalence is an isomorphism. The 2-cells themselves do not disappear: the hom-categories of \(\Sup\) are far from discrete, and the two-dimensional universal properties of \defx\ref{Def 2-kernel} and \defx\ref{Def 2-cokernel} quantify over them. Everything below therefore happens on the nose, and what remains two-dimensional is the order.
\end{remark}

\begin{proposition}\label{P:SupBizero}
	The one-element lattice is a strong bizero object of \(\Sup\).
\end{proposition}

\begin{proof}
	Denote the one-element lattice as \(1\). For every complete lattice \(L\) there is exactly one map \(L\to 1\), and a join-preserving map \(1\to L\) must send the empty join to the empty join, so the only one is the constant map at \(\bot\); both hom-posets are singleton categories, and \(1\) is a bizero object. A 1-cell \(L\to M\) is null precisely when it is the constant map at \(\bot\), so two parallel null 1-cells are equal and the unique 2-cell between them is their identity; thus \(1\) is strong.
\end{proof}

\begin{proposition}\label{P:SupKernels}
	The 2-category \(\Sup\) is 2-z-exact. Explicitly, for a join-preserving map \(f\colon {L\to M}\), a 2-kernel of \(f\) is the inclusion \(k\) of the down-segment \(\dseg{a}=\{x\in L\mid x\leq a\}\) for \(a=\bigvee\{x\in L\mid f(x)=\bot\}\), and a 2-cokernel of \(f\) is the map \(q_{s}\colon M\to\useg{s}=\{y\in M\mid s\leq y\}\colon y\mapsto y\vee s\) for \(s=f(\top)\).
\end{proposition}

\begin{proof}
	A down-segment is closed under all joins of \(L\), so it is a complete lattice and its inclusion \(k\) preserves them. An up-segment is a complete lattice whose nonempty joins are those of \(M\) and whose empty join is \(s\), and \(q_{s}\) preserves joins because \((\bigvee y_{i})\vee s=\bigvee(y_{i}\vee s)\).

	Since \(f\) preserves joins, \(f(a)=\bigvee\{f(x)\mid f(x)=\bot\}=\bot\), so \(f(x)=\bot\) for every \(x\leq a\), and \(f\comp k\) is the null 1-cell. For condition~\((1)\) of \defx\ref{Def 2-kernel}, let \(z\colon Z\to L\) carry an invertible 2-cell \(f\comp z\iso 0\), so that \(f\comp z=0\) by \remx\ref{Rem Locally Ordered}. Then \(z(w)\leq a\) for every \(w\in Z\), and \(z\) corestricts to a join-preserving \(u\colon Z\to\dseg{a}\) with \(k\comp u=z\). For condition~\((2)\), a 2-cell \(k\comp u\aR{}k\comp v\) says that \(u(w)\leq v(w)\) in \(L\) for every \(w\), which is the same statement in \(\dseg{a}\); whiskering with \(k\) is therefore a bijection on 2-cells, both sides having at most one element.

	For the 2-cokernel, \(s=f(\top)=\bigvee_{x\in L}f(x)\) is the largest value of \(f\), so \(q_{s}\comp f\) is the constant map at \(s=\bot_{\useg{s}}\), the null 1-cell. Let \(z\colon M\to Z\) satisfy \(z\comp f=0\); then \(z(s)=z(f(\top))=\bot\). The restriction \(u\colon\useg{s}\to Z\) of \(z\) preserves nonempty joins because \(z\) does, and the empty one because \(u(s)=\bot\); and \(u\comp q_{s}=z\) because \(z(y\vee s)=z(y)\vee z(s)=z(y)\). Condition~\((2)\) holds because \(q_{s}\) is surjective: \(q_{s}(y)=y\) for \(y\in\useg{s}\), so a 2-cell \(u\comp q_{s}\aR{}v\comp q_{s}\) immediately gives \(u\leq v\).
\end{proof}

\begin{corollary}\label{C:SupNormal}
	For every element \(a\) of a complete lattice \(L\), the inclusion \(\dseg{a}\to L\) is a 2-kernel of \(q_{a}\colon L\to\useg{a}\), and \(q_{a}\) is a 2-cokernel of that inclusion. Hence, up to equivalence, any short 2-exact sequence takes the shape
	\[
		\dseg{a}\to L\to \useg{a}
	\]
	for some \(a\in L\). Consequently, in \(\Sup\) every normal 2-monomorphism is a down-segment inclusion precomposed with an isomorphism, every normal 2-epimorphism is a map \(q_{a}\) postcomposed with an isomorphism, and both classes are closed under composition.
\end{corollary}

\begin{proof}
	The formulas of \prox\ref{P:SupKernels}, applied to \(q_{a}\), return \(\bigvee\{x\mid x\vee a=a\}=a\); applied to the inclusion, they return \(s=a\). A normal 2-monomorphism is a 2-kernel of some 1-cell, so by \prox\ref{P:SupKernels} and \corx\ref{corollkeruniqueuptoequiv} it differs from a down-segment inclusion by an equivalence, which is an isomorphism by \remx\ref{Rem Locally Ordered}; dually for normal 2-epimorphisms. For the closure, a down-segment of a down-segment is a down-segment, since \(\dseg{b}\) computed in \(\dseg{a}\) is \(\dseg{b}\) computed in \(L\) when \(b\leq a\); and for \(b\leq c\) the composite of \(q_{b}\colon L\to\useg{b}\) with \(q_{c}\colon\useg{b}\to\useg{c}\) is \(q_{c}\colon L\to\useg{c}\), since \((x\vee b)\vee c=x\vee c\). The general case follows by \corx\ref{C:NormalTransport}.
\end{proof}

\begin{proposition}\label{P:SupAntinormal}
	Up to composition with isomorphisms on either side, the antinormal morphisms of \(\Sup\) are the maps
	\[
		c_{a,b}\colon\dseg{a}\too\useg{b}\colon x\mapsto x\vee b
	\]
	for elements \(a\), \(b\) of a complete lattice \(L\). Moreover, \(c_{a,b}\) is normal if and only if the transposition
	\[
		[a\wedge b,a]\too[b,a\vee b]\colon y\mapsto y\vee b
	\]
	is an isomorphism.
\end{proposition}

\begin{proof}
	An antinormal morphism is a normal 2-monomorphism followed by a normal 2-epimorphism, so by \corx\ref{C:SupNormal} it is, up to isomorphisms on either side, of the form \(q_{b}\comp k\colon\dseg{a}\to\useg{b}\) with \(k\) the inclusion, which is \(c_{a,b}\); and \corx\ref{C:NormalTransport} makes normality insensitive to the isomorphisms.

	Suppose that the transposition is an isomorphism \(u\). The interval \([a\wedge b,a]\) is the up-segment \(\useg{(a\wedge b)}\) of \(\dseg{a}\), and \([b,a\vee b]\) is the down-segment \(\dseg{(a\vee b)}\) of \(\useg{b}\). Writing \(q\colon\dseg{a}\to[a\wedge b,a]\) for the 2-cokernel \(q_{a\wedge b}\) and \(n\colon[b,a\vee b]\to\useg{b}\) for the 2-kernel inclusion, both provided by \corx\ref{C:SupNormal}, we find \(n(u(q(x)))=(x\vee(a\wedge b))\vee b=x\vee b=c_{a,b}(x)\), so that \(c_{a,b}=n\comp(u\comp q)\) is a normal 2-epimorphism followed by a normal 2-monomorphism (using \corx\ref{C:NormalTransport} again) and hence normal.

	Suppose conversely that \(c_{a,b}\) is normal, say \(c_{a,b}=m\comp e\) with \(e\) a normal 2-epimorphism and \(m\) a normal 2-monomorphism, the factorisation being on the nose by \remx\ref{Rem Locally Ordered}. The 2-kernel of \(c_{a,b}\) is \(\dseg{(a\wedge b)}\), since for \(x\leq a\) we have \(x\vee b=b\) if and only if \(x\leq b\); by \prox\ref{CoKernel of Composite} it is also the 2-kernel of \(e\). By \corx\ref{C:SupNormal} we may write \(e=v\comp q_{t}\) with \(q_{t}\colon\dseg{a}\to[t,a]\) and \(v\) an isomorphism; the 2-kernel of \(e\) is \(\dseg{t}\), so \(t=a\wedge b\). Replacing \(m\) by \(m\comp v\), a normal 2-monomorphism by \corx\ref{C:NormalTransport}, we may assume that \(c_{a,b}=m\comp q_{a\wedge b}\) with \(m\colon[a\wedge b,a]\to\useg{b}\) a normal 2-monomorphism; and by \corx\ref{C:SupNormal} again, \(m=n\comp w\) with \(w\colon[a\wedge b,a]\to[b,d]\) an isomorphism onto a down-segment \([b,d]\) of \(\useg{b}\). Evaluating at \(y\in[a\wedge b,a]\) gives \(w(y)=c_{a,b}(y)=y\vee b\), so \(d=w(a)=a\vee b\), and \(w\) is the transposition.
\end{proof}

\begin{remark}\label{Rem SupNormal}
	Furthermore, by \corx\ref{C:SupNormal}, a 1-cell \(g\colon L\to M\) is normal precisely when, writing \(n\) for the largest element of \(L\) with \(g(n)=\bot\) and \(m=g(\top)\), the map \(g\) kills \(\dseg{n}\) and restricts to an isomorphism \([n,\top]\to\dseg{m}\).
\end{remark}

\begin{proposition}\label{P:SupModular}
	For a complete lattice \(L\) the following conditions are equivalent:
	\begin{tfae}
		\item \(L\) is modular;
		\item for all \(a\), \(b\in L\), the antinormal morphism \(c_{a,b}\) is normal.
	\end{tfae}
	In particular, the 2-category \(\Sup\) is not 2-di-exact.
\end{proposition}

\begin{proof}
	\((i)\Rightarrow(ii)\) In a modular lattice, the transposition \(y\mapsto y\vee b\colon[a\wedge b,a]\to[b,a\vee b]\) and the map \(z\mapsto z\wedge a\) in the other direction are mutually inverse: for \(y\in[a\wedge b,a]\) the modular law gives \((y\vee b)\wedge a=y\vee(b\wedge a)=y\), and for \(z\in[b,a\vee b]\) it gives \((z\wedge a)\vee b=z\wedge(a\vee b)=z\). Both maps are monotone, so the transposition is an order isomorphism, and an order isomorphism of complete lattices preserves all joins. This is Dedekind's transposition principle \cite[Chapter~I, Theorem~13]{Birkhoff}, and \prox\ref{P:SupAntinormal} concludes.

	\((ii)\Rightarrow(i)\) If \(L\) is not modular then it contains a pentagon \cite[Chapter~I, Theorem~12]{Birkhoff}: elements \(u<w\) and \(v\) with \(u\vee v=w\vee v\) and \(u\wedge v=w\wedge v\). Both \(u\) and \(w\) lie in the interval \([w\wedge v,w]\), and the transposition to \([v,w\vee v]\) carries both to \(w\vee v\), so it is not injective, and \(c_{w,v}\) is not normal by \prox\ref{P:SupAntinormal}.

	The pentagon \(N_{5}\) itself is a finite, hence complete, lattice, so \(\Sup\) contains an object with a non-normal antinormal morphism, and condition \textup{(DI2)} fails.
\end{proof}

Write \(\Supmod\) for the full sub-2-category of \(\Sup\) on the modular lattices.

\begin{theorem}\label{T:LatticeModel}
	The locally ordered 2-category \(\Supmod\) of complete modular lattices and join-preserving maps is 2-di-exact.
\end{theorem}

\begin{proof}
	Down-segments, up-segments and intervals of a lattice are sublattices, and modularity is inherited by sublattices. So the 2-kernels and 2-cokernels of \prox\ref{P:SupKernels}, computed in \(\Sup\) for a 1-cell of \(\Supmod\), have their vertices in \(\Supmod\); and since \(\Supmod\) is a full sub-2-category, a 2-kernel or 2-cokernel computed in \(\Sup\) whose vertex lies in \(\Supmod\) is a 2-kernel or a 2-cokernel in \(\Supmod\) respectively. This is condition \textup{(DI1)}, and the same remark transfers \corx\ref{C:SupNormal} and \prox\ref{P:SupAntinormal}. For condition \textup{(DI2)}, an antinormal morphism of \(\Supmod\) is, up to isomorphisms, of the form \(c_{a,b}\) for elements \(a\), \(b\) of a modular lattice \(L\), hence normal by \prox\ref{P:SupModular}; and the middle vertex \([a\wedge b,a]\) of the factorisation constructed in \prox\ref{P:SupAntinormal} is an interval of \(L\), so the factorisation lies in \(\Supmod\).
\end{proof}

\begin{proposition}\label{P:SupNotDPN}
	The 2-category \(\Sup\) is homologically self-dual, and its classes of normal 2-monomorphisms and normal 2-epimorphisms are closed under composition; but it does not satisfy condition \(\DPN\) of \defx\ref{Def:DPN}.
\end{proposition}

\begin{proof}
	Closure under composition is part of \corx\ref{C:SupNormal}. To prove that \(\Sup\) is homologically self-dual, let \((m,e)\) be an antinormal decomposition of the zero map in \(\Sup\). Up to isomorphisms, which change nothing by \corx\ref{C:NormalTransport} and \corx\ref{corollkeruniqueuptoequiv}, \(m\) is the inclusion \(\dseg{a}\to L\) and \(e=q_{b}\); and that \(e\comp m=c_{a,b}\) is null says that \(x\vee b=b\) for all \(x\leq a\), that is, \(a\leq b\). The dinverse pair is \((\twoker(e),\twocoker(m))=(\dseg{b}\to L,q_{a})\), so the dinversion is \(c_{b,a}\colon\dseg{b}\to\useg{a}\), whose transposition is \([b\wedge a,b]=[a,b]\to[a,b\vee a]=[a,b]\colon y\mapsto y\vee a\), the identity of \([a,b]\); by \prox\ref{P:SupAntinormal} the dinversion is normal, which is \defx\ref{Def:HSD}.

	For \(\DPN\), take the pentagon \(N_{5}=\{\bot,x,z,y,\top\}\) with \(\bot<x<z<\top\), the element \(y\) incomparable to \(x\) and \(z\), and \(x\vee y=z\vee y=\top\), \(x\wedge y=z\wedge y=\bot\). The antinormal pair \((\dseg{z}\to N_{5},q_{y})\) has composite \(c_{z,y}\), whose transposition \([\bot,z]\to[y,\top]\) goes from a three-element chain to a two-element one, so it is no isomorphism and \(c_{z,y}\) is not normal. The dinversion is \(c_{y,z}\), whose transposition \([\bot,y]\to[z,\top]\) is an isomorphism of two-element chains, so \(c_{y,z}\) is normal, and the biconditional of \defx\ref{Def:DPN} fails.
\end{proof}

\begin{remark}\label{Rem Lattice Profile}
	\prox\ref{P:SupNotDPN}, together with what we will establish in Section~\ref{SS:NSDAbCat} shows that \(\Sup\) has the exact same profile of \(\AbCat\): homologically self-dual, both composition closure properties, not \(\DPN\). Moreover, it yields a five-element witness for the separation recorded in \corx\ref{C:HSDstrict}.

	The resonance with dimension 1 is also interesting. In \cite[Example~4.2.1]{Afsa-24} the failure of \(\DPN\) in the category of commutative monoids is exhibited on a pentagonal semilattice, and \cite[Theorem~4.3.4]{Afsa-24} shows that in a di-exact category the lattice of normal subobjects of every object is modular. \prox\ref{P:SupModular} closes the circle: read on the 2-category of complete lattices themselves, condition \textup{(DI2)} is not merely implied by modularity but equivalent to it. The level in between is inhabited as well: Section~\ref{SS:NSDHilbert} exhibits, in the Hilbert lattices, complete lattices that are not modular and yet support condition \(\DPN\).
\end{remark}

\begin{remark}\label{Rem Lattice Snake}
	\thex\ref{Snake General 2D}, read in \(\Supmod\), is a Snake Lemma for modular lattices, and it is worth spelling out what it says when the rows are short exact sequences. The rows of the ladder are, up to isomorphism, the sequences \(\dseg{u}\to L\to\useg{u}\) and \(\dseg{v}\to M\to\useg{v}\) of \corx\ref{C:SupNormal}, for elements \(u\in L\) and \(v\in M\), and the squares commute on the nose by \remx\ref{Rem Locally Ordered}, so the outer verticals are the restriction \(f\colon\dseg{u}\to\dseg{v}\) of \(g\), which forces \(g(u)\leq v\), and the induced map \(h\colon\useg{u}\to\useg{v}\colon y\mapsto g(y)\vee v\). Their normality comes for free: each factors as a normal 2-epimorphism, followed by a composite of one transposition with a restriction of the isomorphism above, followed by a normal 2-monomorphism, and the transpositions are invertible because \(L\) and \(M\) are modular. The input of the theorem thus reduces to a normal 1-cell \(g\) together with a pair of elements \(u\), \(v\) such that \(g(u)\leq v\). As in \remx\ref{Rem SupNormal} characterising normality of morphisms, we write \(n\) for the largest element of \(L\) with \(g(n)=\bot\) and \(m=g(\top)\). Writing \(s\) for the largest element of \(L\) with \(g(s)\leq v\), the snake sequence becomes
	\[
		0\to \dseg{(u\wedge n)}\too\dseg{n}\too[u,s]\too[g(u),v]\too\useg{m}\too\useg{(m\vee v)}\to 0\text{,}
	\]
	whose 1-cells are, in order, the inclusion, then \(x\mapsto x\vee u\), then the connecting 1-cell \(\partial\)---which may be taken to be \(g\) itself, restricted to \([u,s]\)---then \(y\mapsto y\vee m\), then \(z\mapsto z\vee v\). Every 1-cell of the sequence is normal, and 2-exactness at each of the four middle positions amounts to a single identity between an image element and a kernel element. At \(\TwoKer(g)\) and \(\TwoCoker(g)\) the identity holds by construction; the two with content, at \(\TwoKer(h)\) and \(\TwoCoker(f)\), read
	\[
		\bigvee\{x\in[u,s]\mid g(x)=g(u)\}=u\vee n\qquad\text{and}\qquad g(s)=v\wedge m\text{,}
	\]
	and both are consequences of the isomorphism \([n,\top]\to\dseg{m}\) which holds by \remx\ref{Rem SupNormal} since \(g\) is normal.
\end{remark}

\begin{remark}\label{Rem Lattice Grandis}
	The normal morphisms of \(\Sup\) are, by \corx\ref{C:SupNormal}, exactly the maps that project onto an up-segment and then embed it as a down-segment of the codomain. These are the \emph{exact} Galois connections of Grandis's lattice-theoretic homological algebra \cite{Grandis-HA1,Grandis-HA2}---characterised in \cite[1.2.5]{Grandis-HA2} by precisely this factorisation---and between modular lattices they are his \emph{modular connections} \cite[1.2.8]{Grandis-HA2}. The settings differ in two respects. Grandis's lattices need not be complete, whereas on complete lattices a monotone map preserves all joins if and only if it is the covariant part of a Galois connection, so that his connections between complete lattices are exactly the 1-cells of \(\Sup\); and his category \(\Mlc\) of modular lattices and modular connections keeps only the normal morphisms, where \(\Supmod\) keeps every join-preserving map and normality singles the connections out. His kernels, cokernels and short exact sequences agree with those of \prox\ref{P:SupKernels} and \corx\ref{C:SupNormal} \cite[1.2.3]{Grandis-HA2}, and \(\Mlc\) is a p-exact category \cite[1.2.8]{Grandis-HA2}, so the one-dimensional content of \remx\ref{Rem Lattice Snake}---the six-term exact sequence of modular connections---is available in his framework: the connecting morphism exists in any of his homological categories \cite[Lemma~3.3.3]{Grandis-HA2}, and the six-term sequence is exact under modularity conditions on the diagram \cite[Proposition~3.3.4]{Grandis-HA2} which hold automatically for modular connections. Part~(d) of the latter is the one-dimensional counterpart of the reduction in \remx\ref{Rem Lattice Snake}: exactness of the middle vertical alone already gives exactness at the two central objects. What the two-dimensional reading adds is the derivation. The connections are not chosen but forced, as the normal 1-cells among all join-preserving maps, by universal properties whose 2-cells are the order; and modularity, which for Grandis delimits \(\Mlc\) by hypothesis, enters here as an instance of condition \textup{(DI2)}, by \prox\ref{P:SupModular}. The model is populated by the lattices on which classical homological algebra acts. Indeed, submodule lattices of modules and, more generally, normal-subobject lattices in any semi-abelian category are complete modular lattices.
\end{remark}

\begin{remark}[Noetherian spectra]\label{Rem Lattice Spectra}
	The two models of Section~\ref{S:Model} meet over the same commutative algebra. Call a subset of \(\Spec R\) \dfn{specialisation-closed} when it contains, with every prime, all primes containing it. Such subsets are closed in the power set under arbitrary unions and arbitrary intersections, so they form a complete lattice \(\Scl(R)\) whose joins are unions; both operations being computed in the power set, that lattice is distributive, hence modular, and \(\Scl(R)\) is an object of \(\Supmod\). Its 1-cells come from geometry. A ring homomorphism \(\phi\colon R\to S\) makes \(\Spec\phi\colon\Spec S\to\Spec R\) continuous, and a continuous map preserves specialisation, so that \((\Spec\phi)^{-1}\) carries specialisation-closed subsets to specialisation-closed subsets and preserves arbitrary unions: it is a 1-cell \(\Scl(R)\to\Scl(S)\) of \(\Supmod\), and a 2-cell between two such is an inclusion of subsets.

	These are the lattices of Section~\ref{SS:ModelExamples}, read one dimension down. By the classification invoked in the proof of \prox\ref{P:CohAS}, the Serre subcategories of the category of finitely generated \(R\)-modules are the \(T_{Z}=\{M\mid\Supp M\subseteq Z\}\) for \(Z\subseteq\Spec R\) specialisation-closed \cite[Proposition~VI.2.4]{Gabriel62}, so \(\Scl(R)\) is the lattice of Serre subcategories of that object of \(\Sat\), which by \prox\ref{P:AbCatKernel} is its lattice of 2-kernels. The same Noetherian spectrum is therefore read by \corx\ref{C:Populated} as an object of \(\Sat\) and by the present subsection as an object of \(\Supmod\).
\end{remark}

\begin{remark}\label{Rem Lattice Sharp}
	In the locally discrete abelian context, the hypothesis of \thex\ref{Snake General 2D} that the three verticals be normal is invisible: in an abelian category every morphism is normal. In \(\Supmod\) this has concrete meaning and cannot be dropped. Let \(L=\{\bot<x<\top\}\) and \(M=\{\bot<\top\}\) be chains---distributive, so certainly modular---and let \(g\colon L\to M\) send \(\bot\) to \(\bot\) and the other two elements to \(\top\): a join-preserving map with trivial 2-kernel which is not a 2-monomorphism, the freedom that abelian categories lack, and which normality of \(g\) would remove. Take the rows determined by \(u=x\) and \(v=\top\). Both squares commute on the nose, and the outer verticals are normal: \(f\colon\dseg{x}\to\dseg{\top}=M\) is an isomorphism of two-element chains, and \(h\colon\useg{x}\to\useg{\top}\) is null, hence normal by \prox\ref{P:NullNormal}. Yet no 1-cell \(\partial\) makes the snake sequence 2-exact at \(\TwoKer(h)\): the lattice \(\TwoCoker(f)\) is trivial, so \(\partial\) is null and its 2-kernel is all of \(\TwoKer(h)=[x,\top]\), while \(\ov{b}\colon\TwoKer(g)=\dseg{\bot}\to[x,\top]\) is null, with trivial 2-image. Indeed the first exactness identity displayed in \remx\ref{Rem Lattice Snake} fails: here \(s=\top\), so its left-hand side is \(\top\) while \(u\vee n=x\). In the form \(\bigvee\{x'\in L\mid g(x')\leq g(x)\}=x\vee n\), valid for normal \(g\) and violated here at \(x\), this is exactly the identity which Grandis isolates as the lattice-theoretic substitute for the subtraction step of the classical diagram chase \cite[2.3.5]{Grandis-HA2}.
\end{remark}

\begin{remark}[\(\onecat{Ord}\)-enriched categories]\label{Rem Lattice Enriched}
	Order is brought into categorical algebra from the other end by Clementino, Montoli and Rodelo~\cite{CleMonRod23}, who work in categories enriched in the category \(\onecat{Ord}\) of preordered sets---a locally ordered 2-category is the antisymmetric case---and replace the pullback there by the \emph{comma object}: for 1-cells \(f\colon X\to Y\) and \(g\colon Z\to Y\), the comma object \(f/g\) is a pair of 1-cells \(p_{1}\colon f/g\to X\), \(p_{2}\colon f/g\to Z\) with \(f\comp p_{1}\leq g\comp p_{2}\), universal in the two-dimensional sense that every pair \((a,b)\) with \(f\comp a\leq g\comp b\) factors uniquely as \((p_{1}\comp w,p_{2}\comp w)\), with \(w\leq w'\) whenever \(a\leq a'\) and \(b\leq b'\) \cite[Section~1]{CleMonRod23}. Protomodularity becomes \emph{lax protomodularity}, the conservativity of the comma change-of-base functors \(V_{a}\) \cite[Definition~2.13]{CleMonRod23}---in the pointed case, a Split Short Five Lemma whose kernels are the comma objects \(0/f\) \cite[Theorem~2.14]{CleMonRod23}---and \emph{colax protomodularity} is lax protomodularity for the opposite order on the 2-cells. The theory is made for preordered algebra: its motivating example is the category of preordered abelian groups, and since the implication from protomodularity to the Mal'tsev property fails in the enriched setting \cite[Proposition~3.5]{CleRod26}, examples of (co)lax protomodular categories whose enrichment is not degenerate are provided by varieties of ordered algebras \cite[Example~2.6]{CleRod26}.

	The 2-category \(\Sup\) is such a category. It is finitely complete as a 2-category, and it has all comma objects, \(f/g\) being the set \(\{(x,z)\in L\times N\mid f(x)\leq g(z)\}\), closed under joins because \(f\) and \(g\) preserve them. What decides its behaviour is the position of the null 1-cell: the null map \(L\to M\) is the constant map at \(\bot\), the \emph{least} element of the hom-poset \(\Sup(L,M)\). Three things follow. The comma object \(0/f\) is \((L,1_{L})\) for every \(f\colon L\to M\), since \(0\leq f\comp p_{2}\) is no condition at all: the lax kernel of every 1-cell is the identity. The comma object in the other direction, \(f/0\), is the 2-kernel of \defx\ref{Def 2-kernel}: the inequality \(f\comp p_{1}\leq 0\) is the equation \(f\comp p_{1}=0\), the factorisation is condition~(1), and the reflection of the order by \(p_{1}\) is condition~(2). So the 2-kernels of \prox\ref{P:SupKernels} are the \emph{colax} kernels. And every \(V_{a}\) is conservative: if \(b\colon(X,f,s)\to(X',f',s')\) is a morphism of points over \(Y\) for which \(V_{a}(b)=\langle p_{1},b\comp p_{2}\rangle\colon a/f\to a/f'\) is invertible, then \(j=\langle 0,1_{X}\rangle\colon X\to a/f\) exists because \(a\comp 0=0\leq f\), satisfies \(V_{a}(b)\comp j=j'\comp b\) by the uniqueness in the universal property, and \(p_{2}\comp V_{a}(b)^{-1}\comp j'\) inverts \(b\). Hence \(\Sup\) is lax protomodular, and it is neither a variety nor degenerate. The reason, however, has nothing to do with lattices: any pointed \(\onecat{Ord}\)-enriched category with comma objects in which the null morphisms are least is lax protomodular, and dually colax protomodular when they are greatest. Clementino, Montoli and Rodelo observe this for the chaotic enrichment, where all parallel morphisms are comparable \cite[Example~2.16(3)]{CleMonRod23}; the examples of \cite[Example~2.6]{CleRod26} are ordered sets with \(0=\bot\), respectively \(0=\top\); and the enrichment of the preordered abelian groups of \cite[Section~3]{CleMonRod23}, \(f\leq g\) when \(f(x)\leq g(x)\) for all \(x\geq 0\), makes \(0\leq g\) for every \(g\), so that it is of the same kind. Colax protomodularity, on the other hand, fails in \(\Sup\): the 1-cell \(g\colon L\to M\) of \remx\ref{Rem Lattice Sharp} is split by the inclusion of \(M\) in \(L\) as \(\{\bot,\top\}\), it has the same 2-kernel \(\dseg{\bot}\) as \(1_{M}\), and the comparison from the split sequence of \(g\) to that of \(1_{M}\) which is the identity on the kernels and on \(M\) is \(g\) itself.

	The order which the enriched programme puts into the limits, this paper keeps out of them---2-kernels are bipullbacks, \prox\ref{Kernel vs pullback}---and puts into the classes of 1-cells; and where the two meet, in \(\Sup\), they do so on the side on which the enriched kernel carries no information.
\end{remark}

\section{The Snake Lemma without self-duality}\label{S:NonSelfDual}

Condition \textup{(DI2)}, which asks every antinormal morphism to be normal, is self-dual. The classical proof of the Snake Lemma given in \cite{Borceux-Bourn}---see also \cite{PVdL3}---does not ask for less than this; it asks for something else. It runs instead on two hypotheses that are not self-dual, and neither of them is 2-di-exactness weakened: the first constrains only part of what \textup{(DI2)} constrains, and is strictly weaker than it (\remx\ref{Rem NSD Strict}), whereas the second asks for something \textup{(DI2)} does not ask for at all, namely that normal 2-epimorphisms be closed under composition. The two sets of hypotheses are not comparable, and \remx\ref{Rem NSD Discrete} separates them in both directions already in dimension one.

In this section, we generalise to dimension 2 the non-self-dual version of the Snake Lemma, under the 2-dimensional counterpart of the second set of hypotheses. We will use the same useful results of Sections~\ref{S:Preliminaries} to~\ref{S:Bipullbacks}, but change the input to Section~\ref{S:Snake} and with it the shape of the construction.

\subsection{The two hypotheses}\label{SS:NSDHypotheses}

\begin{definition}[dinversion preserves normality]\label{Def:DPN}
	A 2-z-exact 2-category satisfies condition \(\DPN\)---\dfn{dinversion preserves normality}---when for every antinormal pair \((m,e)\) the composite \(e\comp m\) is normal if and only if the dinversion \(w=\twocoker(m)\comp\twoker(e)\) is normal.
\end{definition}

Any two 2-kernels of \(e\) differ by an equivalence, and likewise any two 2-cokernels of \(m\), so \corx\ref{C:NormalTransport} makes the condition independent of those choices.

\begin{proposition}\label{P:DPNPlace}
	A 2-di-exact 2-category satisfies \(\DPN\); and a 2-z-exact 2-category satisfying \(\DPN\) is homologically self-dual.
\end{proposition}
\begin{proof}
	Under \textup{(DI2)} both sides of the biconditional hold outright: the composite \(e\comp m\) is antinormal by definition, and so is the dinversion \(w\), being the normal 2-monomorphism \(\twoker(e)\) followed by the normal 2-epimorphism \(\twocoker(m)\); so both are normal, as observed in \prox\ref{P:DiExactHSD}.

	For the second assertion, let \((m,e)\) be an antinormal decomposition of the zero map. A null morphism is normal, by \prox\ref{P:NullNormal}, so \(e\comp m\) is normal, and \(\DPN\) makes \(w\) normal. That is homological self-duality.
\end{proof}

Condition \(\DPN\) therefore lies between 2-di-exactness and homological self-duality, and everything proved in Section~\ref{S:Exact} is available under it: the Pure Snake Lemma, \prox\ref{Criteria HSD}, \prox\ref{Third Iso} and \prox\ref{Exactness via Homology}.

The second hypothesis is a closure condition, and it is the one that is not self-dual.

\begin{definition}\label{Def:NEC}
	A 2-category satisfies condition \(\NEC\) when the composite of two normal 2-epimorphisms is a normal 2-epimorphism.
\end{definition}

\begin{remark}\label{Rem NSD Strict}
	Both implications of \prox\ref{P:DPNPlace} are strict. As Section~\ref{SS:NSDAbCat} shows, \(\AbCat\) is homologically self-dual and satisfies \(\NEC\) and its dual, but it does not satisfy \(\DPN\) (\corx\ref{C:HSDstrict}) (similarly to what the category \(\CMon\) of commutative monoids does in dimension one); the 2-category \(\Sup\) of complete lattices does the same with a five-element witness, the pentagon (\prox\ref{P:SupNotDPN}). The fact that the first implication is strict as well is already shown in dimension one: for a field \(F\), the category of short exact sequences of \(F\)-vector spaces satisfies \(\DPN\) and is not di-exact \cite[Example~4.3.8]{Afsa-24}, so under the dictionary of Section~\ref{SS:Discretisation} it is a 2-category that satisfies \(\DPN\) and not \textup{(DI2)}. Section~\ref{SS:NSDHilbert} exhibits a genuinely 2-dimensional witness: the 2-category of Hilbert lattices satisfies \(\DPN\) and not \textup{(DI2)} (\thex\ref{T:HilbertModel}).
\end{remark}

\begin{remark}\label{Rem NSD Discrete}
	The two \emph{sets} of hypotheses---\textup{(DI2)} on the one side, \(\DPN\) together with \(\NEC\) on the other---are incomparable, and dimension one already shows it. The proof of \cite{Borceux-Bourn} that Section~\ref{SS:NSDTheorem} lifts is carried out in a homological category in the sense of that book: pointed, regular and protomodular. There the normal epimorphisms are the regular epimorphisms, so \(\NEC\) holds; the normal monomorphisms are not closed under composition, since normality of subgroups is not transitive---in the dihedral group \(D_{4}=\langle r,s\mid r^{4}=s^{2}=1,\;srs=r^{-1}\rangle\), the subgroup \(\langle s\rangle\) is normal in \(\langle s,r^{2}\rangle\), which is normal in \(D_{4}\), whereas \(\langle s\rangle\) is not. Di-exactness is exactly the further demand that turns that setting into the semi-abelian one: a homological category with binary coproducts is semi-abelian if and only if it is di-exact \cite[Theorem~5.7.3 with Definition~5.1.1]{PVdL3}; compare with the ``old'' axiom (SA*6) in the foundational article~\cite{Janelidze-Marki-Tholen}.

	The separation runs in both directions. On the one hand \(\Grp\) is semi-abelian, hence di-exact, and di-exactness is self-dual, so \(\Grp\op\) is di-exact as well; but \(\NEC\) in \(\Grp\op\) asks that normal monomorphisms of \(\Grp\) be closed under composition, which the subgroups above refute. On the other hand the category \(\SES(\VectF)\) of short exact sequences of vector spaces over a field \(F\) satisfies \(\DPN\) and is not di-exact \cite[Example~4.3.8]{Afsa-24}, and it satisfies \(\NEC\): identify it with the category of pairs \((V,W)\) of a vector space and a subspace, in which kernels are computed componentwise while cokernels are computed componentwise in the category of morphisms of vector spaces and then reflected along the image factorisation. The normal epimorphisms out of \((V,W)\) are then exactly the morphisms
	\[
		(V,W)\longrightarrow\bigl(V/U,\;(W+U)/U\bigr)\text{,}
	\]
	one for each subspace \(U\subseteq V\), being the cokernels of the normal monos \((U,U\cap W)\rightarrowtail(V,W)\); and the composite of the one for \(U\) with the one for \(U'/U\) is the one for \(U'\). Read as locally discrete 2-categories, \(\Grp\op\) and \(\SES(\VectF)\) therefore separate the hypotheses of this section from the hypothesis of Section~\ref{S:Snake}, in both directions at once.
\end{remark}

\subsection{Why \texorpdfstring{\(\DPN\)}{(DPN)} alone is not enough}\label{SS:NSDCircular}

The construction of Section~\ref{S:Snake} uses \textup{(DI2)} exactly twice: in \lemx\ref{L:ImageOfBBar}, to make \(b\comp\twoker(g)\) normal, and in \lemx\ref{L:ImgMapNormal}, to make \(\twocoim(g)\comp a\) normal. Observe the following.

\begin{proposition}\label{P:TwoSites}
	Let \(A\aar{a}B\aar{b}C\) be a short 2-exact sequence in a 2-z-exact 2-category satisfying \(\DPN\), and let \(g\colon B\to Y\) be a morphism. Then \(b\comp\twoker(g)\) is normal if and only if \(\twocoim(g)\comp a\) is.
\end{proposition}
\begin{proof}
	The pair \((\twoker(g),b)\) is antinormal, \(\twoker(g)\) being a normal 2-monomorphism and \(b\) a normal 2-epimorphism, and its composite is \(b\comp\twoker(g)\). Its dinversion is \(\twocoker(\twoker(g))\comp\twoker(b)\), that is \(\twocoim(g)\comp a\), since \(a=\twoker(b)\) and \(\twocoim(g)\) is by definition a 2-cokernel of \(\twoker(g)\). Now apply \defx\ref{Def:DPN}.
\end{proof}

So \(\DPN\) turns each of the two sites into the other, but settles neither. What breaks the circle is \(\NEC\), and it does so at a single point, \lemx\ref{L:NSDEpi} below; after that the two hypotheses are used only through the normality of two further morphisms, \prox\ref{P:NSDcbar} and \prox\ref{P:NSDbbar}.

\subsection{Two normal morphisms}\label{SS:NSDNormal}

\emph{Throughout this subsection and the next, \(\L\) is a 2-z-exact 2-category satisfying \(\DPN\) and \(\NEC\), and we are given a snake ladder diagram \eqref{Snake Ladder} whose rows are short 2-exact---so that \(a=\twoker(b)\) and \(d=\twocoker(c)\)---and whose verticals \(f\), \(g\) and \(h\) are normal.}

\begin{lemma}\label{L:NSDEpi}
	The composite \(e\coloneq\twocoim(h)\comp b\colon B\to\TwoImg(h)\) is a normal 2-epimorphism; the composite \(e\comp\twoker(g)\) is null; and there is a normal 2-epimorphism \(t\colon\TwoImg(g)\to\TwoImg(h)\) with
	\[
		t\comp\twocoim(g)\iso e\qquad\text{and}\qquad\twoimg(h)\comp t\iso d\comp\twoimg(g)\text{.}
	\]
\end{lemma}
\begin{proof}
	The morphism \(b\) is a normal 2-epimorphism, being a 2-cokernel of \(a\), and \(\twocoim(h)\) is one because \(h\) is normal; so \(e\) is a normal 2-epimorphism by \(\NEC\). This is the only use of that hypothesis.

	Next, \(h\comp b\comp\twoker(g)\iso d\comp g\comp\twoker(g)\iso 0\), by \(\psi\) and the structure 2-cell of \(\twoker(g)\). Since \(h\iso\twoimg(h)\comp\twocoim(h)\) and \(\twoimg(h)\) reflects null morphisms, by \prox\ref{prop2monoreflectsnull}, we get \(e\comp\twoker(g)\iso 0\). As \(\twocoim(g)\) is a 2-cokernel of \(\twoker(g)\), this induces \(t\) with \(t\comp\twocoim(g)\iso e\); and \(t\) is a normal 2-epimorphism by the dual of part~\textup{(ii)} of \prox\ref{Composites of Normal Monos}, the 2-epimorphism \(\twocoim(g)\) cancelling off the normal 2-epimorphism \(e\).

	Finally
	\[
		\twoimg(h)\comp t\comp\twocoim(g)\iso\twoimg(h)\comp\twocoim(h)\comp b\iso h\comp b\iso d\comp g\iso d\comp\twoimg(g)\comp\twocoim(g)\text{,}
	\]
	and \(\twocoim(g)\) is a 2-epimorphism, so \(\twoimg(h)\comp t\iso d\comp\twoimg(g)\).
\end{proof}

\begin{proposition}\label{P:NSDKappa}
	There is a morphism \(k\colon\TwoKer(t)\to X\), essentially unique with the property \(c\comp k\iso\twoimg(g)\comp\twoker(t)\), and it is a 2-kernel of \(\twocoker(g)\comp c\). In particular \(k\) is a normal 2-monomorphism.
\end{proposition}
\begin{proof}
	By \lemx\ref{L:NSDEpi}, \(d\comp\twoimg(g)\comp\twoker(t)\iso\twoimg(h)\comp t\comp\twoker(t)\iso 0\); as \(c=\twoker(d)\), this produces \(k\) with \(c\comp k\iso\twoimg(g)\comp\twoker(t)\), essentially unique because \(c\) is a 2-monomorphism. The diagram
	\begin{equation*}
		\xymatrix@R=5ex@C=4em{
		\TwoKer({t}) \ar@{{ |>}->}[r]^-{\twoker(t)} \ar[d]_{k} &
		\TwoImg({g}) \ar[d]_{\twoimg(g)} \ar@{-{ >>}}[r]^-{t} &
		\TwoImg({h}) \ar[d]^{\twoimg(h)} \\
		X \ar@{{ |>}->}[r]_-{c} & Y \ar@{-{ >>}}[r]_-{d} & Z}
	\end{equation*}
	is then a morphism of short 2-exact sequences: its upper row is short 2-exact because \(t\) is a normal 2-epimorphism, hence a 2-cokernel of its 2-kernel by \prox\ref{kernel is kernel of its cokernel}; its lower row is the lower row of \eqref{Snake Ladder}; its left square commutes by the choice of \(k\) and its right square by \lemx\ref{L:NSDEpi}; and no coherence condition arises, by \prox\ref{P:CoherenceFree}.

	Its right vertical \(\twoimg(h)\) is a 2-monomorphism, so \prox\ref{Mono Implies Left Pullback} makes the left square a bipullback. The morphism \(k\) is a 2-monomorphism, since \(c\comp k\iso\twoimg(g)\comp\twoker(t)\) is a composite of two 2-monomorphisms and \(c\) is a 2-monomorphism, so that part~\textup{(i)} of \prox\ref{Composites of Normal Monos} applies. Finally \(\twoimg(g)\) is a normal 2-monomorphism, hence the 2-kernel of its own 2-cokernel, and that 2-cokernel is \(\twocoker(g)\) by \prox\ref{CoKernel of Composite}, the morphism \(\twocoim(g)\) being a 2-epimorphism. So \prox\ref{P:NormalMonoBipullback}, applied to the left square with \(w=\twocoker(g)\), exhibits \(k\) as a 2-kernel of \(\twocoker(g)\comp c\).
\end{proof}

\begin{proposition}\label{P:NSDcbar}
	The composite \(\twocoker(g)\comp c\) is normal; its normal image factorisation may be written \(n\comp\twocoker(k)\) with \(n\colon\TwoCoker(k)\to\TwoCoker(g)\) a normal 2-monomorphism. The morphism \(\underline{c}\) is normal.
\end{proposition}
\begin{proof}
	The pair \((c,\twocoker(g))\) is antinormal: \(c\) is a normal 2-monomorphism, being a 2-kernel of~\(d\), and \(\twocoker(g)\) is a normal 2-epimorphism. Its dinversion is \(\twocoker(c)\comp\twoker(\twocoker(g))\). Here \(\twocoker(c)\simeq d\), the lower row of \eqref{Snake Ladder} being short 2-exact, and \(\twoker(\twocoker(g))\simeq\twoimg(g)\), as in the proof of \prox\ref{P:NSDKappa}. So the dinversion is \(d\comp\twoimg(g)\iso\twoimg(h)\comp t\), by \lemx\ref{L:NSDEpi}: a normal 2-epimorphism followed by a normal 2-monomorphism, hence normal. Condition \(\DPN\) now makes \(\twocoker(g)\comp c\) normal.

	By \prox\ref{P:NSDKappa} its 2-kernel is \(k\). In any normal image factorisation the 2-epimorphism part is a 2-cokernel of that 2-kernel---this is the first step of the proof of \prox\ref{Image Factorisation of Normal Map is Unique}---so we may take it to be \(\twocoker(k)\), and write \(n\) for the accompanying normal 2-monomorphism.

	Finally, \(\underline{c}\) is characterised by \(\underline{c}\comp\twocoker(f)\iso\twocoker(g)\comp c\), and \(\twocoker(f)\) is a 2-epimorphism; so \(\underline{c}\) is normal by the dual of \prox\ref{P:NormalCancel}.
\end{proof}

\begin{proposition}\label{P:NSDbbar}
	The morphism \(\ov{b}\) is normal.
\end{proposition}
\begin{proof}
	We produce an antinormal decomposition of \(\ov{b}\) whose dinversion is visibly normal, and appeal to \(\DPN\).

	Since \(e\comp a\iso\twocoim(h)\comp b\comp a\iso 0\), the morphism \(a\) factors as \(a\iso\twoker(e)\comp a_1\) for an essentially unique \(a_1\colon A\to\TwoKer(e)\), and \(a_1\) is a normal 2-monomorphism by part~\textup{(ii)} of \prox\ref{Composites of Normal Monos}. Since \(e\comp\twoker(g)\iso 0\), by \lemx\ref{L:NSDEpi}, the morphism \(\twoker(g)\) factors likewise as \(\twoker(g)\iso\twoker(e)\comp m_1\) with \(m_1\colon\TwoKer(g)\to\TwoKer(e)\) a normal 2-monomorphism.

	The rows \(A\aar{a}B\aar{b}C\) and \(\TwoKer(e)\aar{\twoker(e)}B\aar{e}\TwoImg(h)\) are short 2-exact and share the middle object \(B\); with verticals \(a_1\) and \(\twocoim(h)\) they form a morphism of short 2-exact sequences with identity middle component. By \lemx\ref{Pure Snake Lemma} its comparison \(\TwoCoker(a_1)\to\TwoKer(\twocoim(h))\) is an equivalence, and \(\TwoKer(\twocoim(h))\simeq\TwoKer(h)\) by \prox\ref{CoKernel of Composite}. Write
	\[
		e_1\colon\TwoKer(e)\longrightarrow\TwoKer(h)
	\]
	for the composite of \(\twocoker(a_1)\) with that equivalence. It is a 2-cokernel of \(a_1\) followed by an equivalence, hence a normal 2-epimorphism with 2-kernel \(a_1\); and the characterisation of the comparison in \lemx\ref{Pure Snake Lemma} reads \(\twoker(h)\comp e_1\iso b\comp\twoker(e)\).

	The rows \(\TwoKer(g)\aar{\twoker(g)}B\aar{\twocoim(g)}\TwoImg(g)\) and \(\TwoKer(e)\aar{\twoker(e)}B\aar{e}\TwoImg(h)\) are short 2-exact and share the middle object \(B\) as well; their verticals are \(m_1\) and \(t\), by \lemx\ref{L:NSDEpi}. The same lemma supplies an equivalence \(\TwoCoker(m_1)\simeq\TwoKer(t)\); write
	\[
		q_1\colon\TwoKer(e)\longrightarrow\TwoKer(t)
	\]
	for the composite of \(\twocoker(m_1)\) with it, a 2-cokernel of \(m_1\), with \(\twoker(t)\comp q_1\iso\twocoim(g)\comp\twoker(e)\).

	Now \(\ov{b}\iso e_1\comp m_1\). Indeed
	\[
		\twoker(h)\comp e_1\comp m_1\iso b\comp\twoker(e)\comp m_1\iso b\comp\twoker(g)\iso\twoker(h)\comp\ov{b}\text{,}
	\]
	and \(\twoker(h)\) is a 2-monomorphism. So \((m_1,e_1)\) is an antinormal pair with composite \(\ov{b}\), and its dinversion is \(q_1\comp a_1\), up to the equivalences just used and hence up to \corx\ref{C:NormalTransport}.

	It remains to see that \(q_1\comp a_1\) is normal, and for that we compose with \(k\):
	\[
		c\comp k\comp q_1\comp a_1\iso\twoimg(g)\comp\twoker(t)\comp q_1\comp a_1\iso\twoimg(g)\comp\twocoim(g)\comp\twoker(e)\comp a_1\iso g\comp a\iso c\comp f\text{,}
	\]
	so that \(k\comp q_1\comp a_1\iso f\), the morphism \(c\) being a 2-monomorphism. Now \(f\) is normal and \(k\) is a 2-monomorphism, so \(q_1\comp a_1\) is normal by \prox\ref{P:NormalCancel}. By \(\DPN\), the composite \(\ov{b}\iso e_1\comp m_1\) is normal.
\end{proof}

\begin{remark}\label{Rem NSD Sites}
	The two normality statements \prox\ref{P:NSDcbar} and \prox\ref{P:NSDbbar} are what Section~\ref{S:Snake} extracted from \textup{(DI2)} at its two sites, and they are obtained here at very different prices. The first is immediate once \(t\) exists, because the dinversion of \((c,\twocoker(g))\) \emph{is} \(\twoimg(h)\comp t\); the second needs the morphism \(k\) of \prox\ref{P:NSDKappa}, and hence the bipullback argument of Section~\ref{S:Bipullbacks}.
\end{remark}

\subsection{The connecting 1-cell}\label{SS:NSDConnecting}

\begin{proposition}\label{P:NSDLambda}
	There exists a normal 2-monomorphism \(s\colon\TwoImg(f)\to\TwoKer(t)\) such that \(s\comp\twocoim(f)\iso q_1\comp a_1\) and \(k\comp s\iso\twoimg(f)\).
\end{proposition}
\begin{proof}
	From \(k\comp q_1\comp a_1\comp\twoker(f)\iso f\comp\twoker(f)\iso 0\) and the fact that the 2-monomorphism \(k\) reflects null morphisms (\prox\ref{prop2monoreflectsnull}) we get \(q_1\comp a_1\comp\twoker(f)\iso 0\). As \(\twocoim(f)\) is a 2-cokernel of \(\twoker(f)\), this induces \(s\) with \(s\comp\twocoim(f)\iso q_1\comp a_1\). Then
	\[
		k\comp s\comp\twocoim(f)\iso k\comp q_1\comp a_1\iso f\iso\twoimg(f)\comp\twocoim(f)\text{,}
	\]
	and \(\twocoim(f)\) is a 2-epimorphism, so \(k\comp s\iso\twoimg(f)\). Since \(\twoimg(f)\) is a normal 2-monomorphism and \(k\) is a 2-monomorphism, part~\textup{(ii)} of \prox\ref{Composites of Normal Monos} makes \(s\) a normal 2-monomorphism.
\end{proof}

We can now assemble the connecting 1-cell. Consider the rows
\[
	\xymatrix@C=4em{0 \ar[r] & \TwoImg({f}) \ar[d]_-{s} \ar@{{ |>}->}[r]^-{\twoimg(f)} & X \ar@{=}[d] \ar@{-{ >>}}[r]^-{\twocoker(f)} & \TwoCoker({f}) \ar[d]^{v} \ar[r]& 0\\
	0 \ar[r] & \TwoKer({t}) \ar@{{ |>}->}[r]_-{k} & X \ar@{-{ >>}}[r]_-{\twocoker(k)} & \TwoCoker({k}) \ar[r] &  0}
\]
through the shared object \(X\). The upper row is short 2-exact because \(\twoimg(f)\) is a normal 2-monomorphism, its 2-cokernel being \(\twocoker(f)\) by \prox\ref{CoKernel of Composite}; the lower row is short 2-exact by \prox\ref{P:NSDKappa}; the left square commutes by \prox\ref{P:NSDLambda}; and the right vertical \(v\) is induced by \(\twocoker(k)\comp\twoimg(f)\iso\twocoker(k)\comp k\comp s\iso 0\), so that \(v\comp\twocoker(f)\iso\twocoker(k)\). By \lemx\ref{Pure Snake Lemma} the comparison
\[
	z_2\colon\TwoCoker(s)\longrightarrow\TwoKer(v)
\]
is an equivalence. Three further identifications turn this into the equivalence we want.

\begin{lemma}\label{L:NSDChain}
	There are equivalences
	\[
		\TwoCoker(\ov{b})\simeq\TwoCoker(q_1\comp a_1)\simeq\TwoCoker(s)\text{,}\qquad\TwoKer(v)\simeq\TwoKer(\underline{c})\text{.}
	\]
\end{lemma}
\begin{proof}
	The first is \lemx\ref{L:DinversionCoker}, applied to the antinormal pair \((m_1,e_1)\) of \prox\ref{P:NSDbbar}, whose composite is \(\ov{b}\) and whose dinversion is \(q_1\comp a_1\). The second is \prox\ref{CoKernel of Composite}: \(q_1\comp a_1\iso s\comp\twocoim(f)\) with \(\twocoim(f)\) a 2-epimorphism. For the third, \(n\comp v\comp\twocoker(f)\iso n\comp\twocoker(k)\iso\twocoker(g)\comp c\iso\underline{c}\comp\twocoker(f)\) by \prox\ref{P:NSDcbar}, and \(\twocoker(f)\) is a 2-epimorphism, so \(\underline{c}\iso n\comp v\); as \(n\) is a 2-monomorphism, \prox\ref{CoKernel of Composite} gives \(\TwoKer(\underline{c})\simeq\TwoKer(v)\).
\end{proof}

Writing \(z\colon\TwoCoker(\ov{b})\to\TwoKer(\underline{c})\) for the composite of the four equivalences of \lemx\ref{L:NSDChain} and \(z_2\), we define the \dfn{connecting 1-cell} to be
\begin{equation}\label{Def Partial NSD}
	\partial\coloneq\twoker(\underline{c})\comp z\comp\twocoker(\ov{b})\colon\TwoKer(h)\longrightarrow\TwoCoker(f)\text{,}
\end{equation}
exactly as in \eqref{Def Partial}. As in \remx\ref{Rem Partial Choices}, it depends on the choices made only up to an invertible 2-cell.

\subsection{The theorem}\label{SS:NSDTheorem}

\begin{theorem}[The Snake Lemma, non-self-dual version]\label{T:SnakeNonSelfDual}
	Let \(\L\) be a 2-z-exact 2-category which satisfies \(\DPN\) and \(\NEC\). Then the conclusion of \thex\ref{Snake General 2D} holds: for every diagram \eqref{Snake Ladder} with 2-exact rows whose two squares commute up to invertible 2-cells and whose verticals \(f\), \(g\) and \(h\) are normal, there is a 1-cell \(\partial\) making the sequence \eqref{Snake Sequence} 2-exact; and if \(a=\twoker(b)\) then \(\ov{a}=\twoker(\ov{b})\), dually for \(\underline{d}\).

	The same holds if \(\NEC\) is replaced by its dual (closure of normal 2-monomorphisms under composition).
\end{theorem}
\begin{proof}
	Consider first the special case in which \(a=\twoker(b)\) and \(d=\twocoker(c)\), the standing hypothesis of Sections~\ref{SS:NSDNormal} and~\ref{SS:NSDConnecting}. By \prox\ref{P:DPNPlace} the 2-category is homologically self-dual, so the whole of Section~\ref{S:Exact} is available, as it was in Section~\ref{S:Snake}.

	The morphism \(\ov{b}\) is normal by \prox\ref{P:NSDbbar} and \(\underline{c}\) is normal by \prox\ref{P:NSDcbar}. Applying \prox\ref{P:Shape} to \eqref{Def Partial NSD}---with \(q'=\twocoker(\ov{b})\), with the equivalence \(z\) of \lemx\ref{L:NSDChain} and with \(m=\twoker(\underline{c})\)---gives 2-exactness at \(\TwoCoker(f)\) and at \(\TwoKer(h)\), both \(\underline{c}\) and \(\ov{b}\) being normal. By \prox\ref{P:KerBarA} we have \(\ov{a}=\twoker(\ov{b})\), so the pair \((\ov{a},\ov{b})\) is 2-exact at \(\TwoKer(g)\); dually \(\underline{d}=\twocoker(\underline{c})\), and since \(\underline{c}\) is normal the pair \((\underline{c},\underline{d})\) is 2-exact at \(\TwoCoker(g)\), by \prox\ref{Exactness Self-Dual}. The six 1-cells of \eqref{Snake Sequence} are normal: \(\ov{a}\) and \(\underline{d}\) are a 2-kernel and a 2-cokernel, \(\ov{b}\) and \(\underline{c}\) are normal as just recalled, and \(\partial\) is a normal 2-epimorphism followed by a normal 2-monomorphism by construction.

	The general case reduces to this one exactly as in Section~\ref{SS:GeneralCase}. That reduction uses \lemx\ref{L:Restriction}, \prox\ref{CoKernel of Composite} and \prox\ref{P:KerComparison}, and by \remx\ref{Rem General Cost} none of them needs more than the Pure Snake Lemma, which \prox\ref{P:DPNPlace} makes available here.

	For the last assertion, note that \(\DPN\) is self-dual: passing to the opposite 2-category interchanges normal 2-monomorphisms with normal 2-epimorphisms and 2-kernels with 2-cokernels, so it carries an antinormal pair to an antinormal pair, its composite to the composite and its dinversion to the dinversion. A 2-category satisfying \(\DPN\) in which normal 2-monomorphisms are closed under composition therefore has an opposite satisfying \(\DPN\) and \(\NEC\), to which the theorem just proved applies; and the snake sequence of the opposite is the snake sequence of the original read backwards.
\end{proof}

\begin{remark}\label{Rem NSD NotSwap}
	The two proofs are not related by exchanging a hypothesis. Section~\ref{S:Snake} builds \(\partial\) from a pure configuration whose middle object is \(C\), and for that it needs the 2-image \(i\) of \(b\comp\twoker(g)\) as a \emph{normal} 2-monomorphism. Knowing that \(\ov{b}\) is normal does not supply it: \(b\comp\twoker(g)\iso\twoker(h)\comp\ov{b}\), so passing from one to the other composes two normal 2-monomorphisms, which is the hypothesis dual to \(\NEC\). This is why the construction above uses pure configurations with middle objects \(B\) and \(X\) instead, and why \prox\ref{P:NSDKappa} has to be proved at all.
\end{remark}

\subsection{\texorpdfstring{\(\AbCat\)}{AbCat} is homologically self-dual but does not satisfy \texorpdfstring{\(\DPN\)}{(DPN)}}\label{SS:NSDAbCat}

Section~\ref{S:Model} shows that \(\AbCat\) is not 2-di-exact, and the natural question is whether the weaker hypotheses of this section rescue it. We show that they do not: \(\AbCat\) satisfies \(\NEC\) and its dual but not \(\DPN\). We use the reduction of \prox\ref{P:DIabcat} in the pointwise form its proof establishes: for a single pair of Serre subcategories \(K\) and \(S\) of an abelian category \(\onecat{A}\), the antinormal 1-cell \(K\rightarrowtail\onecat{A}\twoheadrightarrow\onecat{A}/S\) is normal if and only if \(K^{S}\) is a Serre subcategory of \(\onecat{A}\).

\begin{proposition}[{\cite[Sections 2 and 7]{CavJanMesRei26}}]\label{P:AbCatComposition}
	In \(\AbCat\), normal 2-monomorphisms are closed under composition, and so are normal 2-epimorphisms.
\end{proposition}

The following proposition can also be derived from {\cite[Sections 2 and 7]{CavJanMesRei26}}. We give another proof here in terms of dinversions.

\begin{proposition}\label{P:AbCatHSD}
	The 2-category \(\AbCat\) is homologically self-dual.
\end{proposition}
\begin{proof}
	Let \((m,e)\) be an antinormal decomposition of the zero map; by \prox\ref{C:AbCatNormal} and \corx\ref{C:NormalTransport} we may take \(m\) to be the inclusion of a Serre subcategory \(K\) of an abelian category \(\onecat{A}\) and~\(e\) the projection \(q\colon\onecat{A}\to\onecat{A}/S\). That \(e\comp m\) is essentially null says exactly that every object of \(K\) becomes zero in \(\onecat{A}/S\), that is, \(K\subseteq S\).

	The dinversion is \(S\rightarrowtail\onecat{A}\twoheadrightarrow\onecat{A}/K\), which is normal if and only if the \(K\)-saturation \(S^{K}\) is a Serre subcategory. By \prox\ref{P:Saturation} we have \(S\subseteq S^{K}\subseteq S\vee K\), and \(S\vee K=S\) because \(K\subseteq S\); so \(S^{K}=S\), which is a Serre subcategory.
\end{proof}

\begin{remark}\label{Rem NSD Symmetric}
	The counterexample of \prox\ref{P:AbCatFails} does not manage to refute \(\DPN\); the reason why explains where an example may be found instead. The cyclic quiver \(1\to2\to1\) admits the rotation exchanging its two vertices, and \(\operatorname{rad}^{3}=0\) is preserved by it, so \(\Lambda\) has an automorphism carrying \(S_{1}\) to \(S_{2}\) and hence \(K\) to \(S\). The two saturations \(K^{S}\) and \(S^{K}\) are therefore carried into one another, and both fail to be Serre subcategories; the equivalence in \defx\ref{Def:DPN} is trivially true for that pair. What is wanted is an algebra on the same quiver whose relations break the rotation.
\end{remark}

\begin{proposition}\label{P:AbCatNotDPN}
	The 2-category \(\AbCat\) does not satisfy \(\DPN\).
\end{proposition}
\begin{proof}
	Let \(\Bbbk\) be a field and let \(\onecat{A}\) be the category of representations
	\[
		\xymatrix@1@C=4em{V_{1} \ar@<.5ex>[r]^-{a} & V_{2} \ar@<.5ex>[l]^-{b}}
	\]
	of finite-dimensional \(\Bbbk\)-vector spaces subject to the single relation \(a\comp b=0\); no condition is imposed on \(b\comp a\). This is the category of finite-dimensional modules over \(\Lambda\coloneq\Bbbk Q/(\alpha\beta)\), where \(Q\) is the cyclic quiver \(1\xrightarrow{\alpha}2\xrightarrow{\beta}1\) and \(\alpha\beta\) is the path of length~2 from~2 to~2; so \(\onecat{A}\) is abelian, with kernels and cokernels formed componentwise. The algebra \(\Lambda\) is five-dimensional, a quotient of the six-dimensional Nakayama algebra of \prox\ref{P:AbCatFails}: the relation forces every path of length~3 to vanish, and it is exactly the rotation symmetry of \remx\ref{Rem NSD Symmetric} that it destroys. Put
	\[
		K\coloneq\{V\mid V_{2}=0\}\text{,}\qquad S\coloneq\{V\mid V_{1}=0\}\text{,}
	\]
	the objects all of whose composition factors are \(\iso S_{1}\), respectively \(\iso S_{2}\); both are Serre subcategories, and \(K\vee S=\onecat{A}\) since every object has finite length.

	\emph{The \(K\)-saturation of \(S\) is a Serre subcategory.} Let \(V\) be any object. Then \(N\coloneq(\ker a,0)\) is a subobject of \(V\), because \(a(\ker a)=0\) and \(b(0)\subseteq\ker a\), and it lies in \(K\). In the quotient \(V/N\) the map induced by \(a\) is injective and the relation still holds, so \(a\comp b=0\) forces \(b=0\) on \(V/N\); hence \((0,V_{2})\) is a subobject of \(V/N\) lying in \(S\), and the remaining quotient has zero first component, so lies in \(K\). This is a filtration of \(V\) with successive subquotients in \(K\), \(S\) and \(K\), so \(S^{K}=\onecat{A}=K\vee S\), and \prox\ref{P:Saturation} applies. Note that no finiteness was used.

	\emph{The \(S\)-saturation of \(K\) is not.} Let \(M\) be the representation with \(M_{1}=\langle x,z\rangle\), \(M_{2}=\langle y\rangle\), \(a(x)=y\), \(a(z)=0\) and \(b(y)=z\); the relation holds, since \(a(b(y))=a(z)=0\), whereas \(b(a(x))=z\neq0\). A subobject \((W_{1},W_{2})\) with \(y\in W_{2}\) contains \(z=b(y)\), and if it contains \(x\) as well it is all of \(M\); a subobject with \(W_{2}=0\) has \(W_{1}\subseteq\ker a=\langle z\rangle\). So \(M\) is uniserial of length~3,
	\[
		0\subset(\langle z\rangle,0)\subset(\langle z\rangle,M_{2})\subset M\text{,}
	\]
	with successive subquotients \(S_{1}\), \(S_{2}\), \(S_{1}\); it is the indecomposable projective \(\Lambda e_{1}\). Let \(X_{1}\subseteq X_{2}\subseteq M\) be a filtration with \(X_{1}\in S\), \(X_{2}/X_{1}\in K\) and \(M/X_{2}\in S\). The three proper subobjects listed have first components \(0\), \(\langle z\rangle\) and \(\langle z\rangle\), of which only the first is zero, so \(X_{1}=0\); then \(X_{2}\in K\) leaves \(X_{2}=0\) or \(X_{2}=(\langle z\rangle,0)\), and in either case \(M/X_{2}\) has nonzero first component and so does not lie in \(S\). By \defx\ref{D:Saturation} and \prox\ref{P:Saturation}, \(M\notin K^{S}\) and \(K^{S}\) is not a Serre subcategory.

	Taking \(m\) to be the inclusion of \(K\) and \(e\) the projection onto \(\onecat{A}/S\), the pair \((m,e)\) is antinormal, its dinversion \(S\rightarrowtail\onecat{A}\twoheadrightarrow\onecat{A}/K\) is normal and the composite \(e\comp m\) is not.
\end{proof}

\begin{corollary}\label{C:HSDstrict}
	Homological self-duality does not imply \(\DPN\). Consequently the Pure Snake Lemma holds in \(\AbCat\) while neither \thex\ref{Snake General 2D} nor \thex\ref{T:SnakeNonSelfDual} applies to it.
\end{corollary}

\subsection{A two-dimensional model: Hilbert lattices}\label{SS:NSDHilbert}

This subsection exhibits a genuinely 2-dimensional witness that separates the hypothesis of this section from 2-di-exactness. The witness arises within the lattice framework of Section~\ref{SS:ModelLattices}: it will be given by the lattices of closed subspaces of complex Hilbert spaces, the lattices with which quantum logic began \cite{BirVNeu36}. Where 2-di-exactness selects the modular lattices (\thex\ref{T:LatticeModel}), the hypotheses of \thex\ref{T:SnakeNonSelfDual} select a wider class, and what places the Hilbert lattices inside it is a theorem of Mackey from functional analysis.

The construction of Section~\ref{SS:ModelLattices} is not tied to modularity. Call a complete lattice \(L\) \dfn{transposition-symmetric} when for all \(a\), \(b\in L\) the transposition \([a\wedge b,a]\to[b,a\vee b]\colon y\mapsto y\vee b\) of \prox\ref{P:SupAntinormal} is invertible if and only if the transposition \([a\wedge b,b]\to[a,a\vee b]\colon y\mapsto y\vee a\) is. Every modular lattice is transposition-symmetric, both transpositions being invertible by Dedekind's principle, and the pentagon is not, as the proof of \prox\ref{P:SupNotDPN} shows.

\begin{proposition}\label{P:SupClass}
	Let \(\mathcal{C}\) be a class of complete lattices, closed under isomorphism, containing a one-element lattice, and such that every interval of a member of \(\mathcal{C}\) belongs to \(\mathcal{C}\). The full sub-2-category \(\Sup_{\mathcal{C}}\) of \(\Sup\) on \(\mathcal{C}\) is then 2-z-exact; \corx\ref{C:SupNormal} and \prox\ref{P:SupAntinormal} describe its normal 2-monomorphisms, its normal 2-epimorphisms and its antinormal morphisms, and both closure properties of \corx\ref{C:SupNormal} hold in it. Moreover, \(\Sup_{\mathcal{C}}\) satisfies condition \textup{(DI2)} if and only if every member of \(\mathcal{C}\) is modular, and condition \(\DPN\) if and only if every member of \(\mathcal{C}\) is transposition-symmetric.
\end{proposition}
\begin{proof}
	The first assertions are proved exactly as \thex\ref{T:LatticeModel}. Down-segments, up-segments and intervals of a member of \(\mathcal{C}\) are intervals, so lie in \(\mathcal{C}\) up to isomorphism; hence the 2-kernels and 2-cokernels of \prox\ref{P:SupKernels}, computed in \(\Sup\), have their vertices in \(\Sup_{\mathcal{C}}\), and since the sub-2-category is full, they are 2-kernels and 2-cokernels there. A one-element lattice is a strong bizero object as in \prox\ref{P:SupBizero}, and \corx\ref{C:SupNormal} and \prox\ref{P:SupAntinormal} transfer, as they did for \(\Supmod\).

	For \textup{(DI2)}: if every member of \(\mathcal{C}\) is modular, the argument of \thex\ref{T:LatticeModel} applies verbatim; if some member \(L\) is not modular, the proof of \prox\ref{P:SupModular} produces elements \(w\), \(v\) of \(L\) for which \(c_{w,v}\) is antinormal in \(\Sup_{\mathcal{C}}\)---its two constituents have their vertices \(\dseg{w}\) and \(\useg{v}\) in \(\mathcal{C}\)---and not normal.

	For \(\DPN\): up to isomorphisms, which change nothing by \corx\ref{C:NormalTransport}, an antinormal pair of \(\Sup_{\mathcal{C}}\) is of the form \((\dseg{a}\to L,q_{b})\) for elements \(a\), \(b\) of a member \(L\), with composite \(c_{a,b}\), and as in the proof of \prox\ref{P:SupNotDPN} its dinversion is \(c_{b,a}\). By \prox\ref{P:SupAntinormal}, whose factorisation passes through the interval \([a\wedge b,a]\) of \(L\), which lies in \(\mathcal{C}\), the composite is normal precisely when the transposition at \((a,b)\) is invertible, and the dinversion is normal precisely when the transposition at \((b,a)\) is. The biconditional of \defx\ref{Def:DPN} is therefore exactly transposition-symmetry of the members of \(\mathcal{C}\).
\end{proof}

Transposition-symmetry has a classical name. Following \cite{MaeMae70}, in the notation of \cite[\S2]{Schreiner66}, a pair \((a,b)\) of elements of a lattice is a \dfn{modular pair}, written \(\mathrm{M}(a,b)\), when \((c\vee a)\wedge b=c\vee(a\wedge b)\) for every \(c\leq b\), and a \dfn{dual modular pair}, written \(\mathrm{M}^{*}(a,b)\), when \((c\wedge a)\vee b=c\wedge(a\vee b)\) for every \(c\geq b\); a lattice is modular precisely when all of its pairs are modular, and \dfn{\(\mathrm{M}\)-symmetric} when \(\mathrm{M}(a,b)\) implies \(\mathrm{M}(b,a)\).

\begin{lemma}\label{L:ModularPairs}
	For elements \(a\), \(b\) of a complete lattice, the transposition \([a\wedge b,a]\to[b,a\vee b]\) is invertible if and only if \(\mathrm{M}(b,a)\) and \(\mathrm{M}^{*}(a,b)\) both hold. In particular, a complete lattice in which both the relation \(\mathrm{M}\) and the relation \(\mathrm{M}^{*}\) are symmetric is transposition-symmetric.
\end{lemma}
\begin{proof}
	For \(x\in[a\wedge b,a]\) and \(z\in[b,a\vee b]\) we have \(x\vee b\leq z\) if and only if \(x\leq z\wedge a\), so the transposition and the monotone map \(z\mapsto z\wedge a\) in the opposite direction form an adjunction. The transposition is invertible if and only if the two composites are identities: in one direction because an invertible monotone map is left adjoint to its inverse and adjoints are unique, in the other because mutually inverse monotone bijections between complete lattices are isomorphisms of \(\Sup\).

	The first composite is the identity when \((x\vee b)\wedge a=x\) for all \(x\in[a\wedge b,a]\), which is \(\mathrm{M}(b,a)\): for \(c\leq a\) the element \(x=c\vee(a\wedge b)\) lies in \([a\wedge b,a]\) and satisfies \(x\vee b=c\vee b\), so the former condition yields \((c\vee b)\wedge a=c\vee(a\wedge b)\), and conversely. The second composite is the identity when \((z\wedge a)\vee b=z\) for all \(z\in[b,a\vee b]\), which is \(\mathrm{M}^{*}(a,b)\) by the dual translation. The final assertion holds since \(\mathrm{M}(b,a)\wedge\mathrm{M}^{*}(a,b)\) and \(\mathrm{M}(a,b)\wedge\mathrm{M}^{*}(b,a)\) are then each equivalent to \(\mathrm{M}(a,b)\wedge\mathrm{M}^{*}(a,b)\).
\end{proof}

Now let \(H\) be a complex Hilbert space and \(\Lat(H)\) its lattice of closed subspaces, ordered by inclusion, with \(A\wedge B=A\cap B\), with \(A\vee B=\overline{A+B}\) and with the orthocomplementation \(A\mapsto A^{\perp}\): a complete lattice. Every interval of \(\Lat(H)\) is again a Hilbert lattice: for closed subspaces \(P\subseteq Q\), the assignments \(X\mapsto X\cap P^{\perp}\) and \(Y\mapsto Y\vee P\) are mutually inverse isomorphisms between \([P,Q]\) and \(\Lat(Q\cap P^{\perp})\), because every closed \(X\) with \(P\subseteq X\) decomposes orthogonally as \(X=P\oplus(X\cap P^{\perp})\), and the sum of two mutually orthogonal closed subspaces is closed. The class of complete lattices isomorphic to some \(\Lat(H)\) therefore satisfies the hypotheses of \prox\ref{P:SupClass}; we write \(\Suphil\) for the resulting full sub-2-category of \(\Sup\) and call it the \dfn{2-category of Hilbert lattices}.

\begin{proposition}\label{P:HilbertSymmetric}
	For closed subspaces \(A\), \(B\) of a Hilbert space \(H\), the transposition \([A\wedge B,A]\to[B,A\vee B]\) of \(\Lat(H)\) is invertible if and only if both \(A+B\) and \(A^{\perp}+B^{\perp}\) are closed. In particular, every Hilbert lattice is transposition-symmetric.
\end{proposition}
\begin{proof}
	By \lemx\ref{L:ModularPairs}, invertibility is the conjunction of \(\mathrm{M}(B,A)\) and \(\mathrm{M}^{*}(A,B)\). A theorem of Mackey identifies the dual modular pairs of \(\Lat(H)\): the condition \(\mathrm{M}^{*}(A,B)\) holds if and only if the vector sum \(A+B\) is closed \cite[Theorem~III-6]{Mackey45}; see also \cite{Schreiner66}. One direction is simple: if \(A+B\) is closed then \(A\vee B=A+B\), and for a closed \(C\supseteq B\), any \(x\in C\cap(A+B)\) is \(p+q\) with \(p\in A\) and \(q\in B\subseteq C\), so that \(p=x-q\in C\cap A\) and \(x\in(C\cap A)+B\subseteq(C\wedge A)\vee B\); the reverse inclusion holds in every lattice, and \(\mathrm{M}^{*}(A,B)\) follows. The converse is Mackey's own argument: if \(A+B\) is not closed, choose \(x\in(A\vee B)\setminus(A+B)\) and test \(\mathrm{M}^{*}(A,B)\) at \(C=B+\langle x\rangle\), which is closed because a closed subspace and a line span a closed subspace; a nonzero multiple of \(x\) in \(C\cap A\) would return \(x\) to \(A+B\), so \(C\cap A\subseteq B\) and \((C\wedge A)\vee B=B\neq C\).

	Orthocomplementation is an anti-automorphism of \(\Lat(H)\), turning joins into meets and reversing the order, so it carries the condition \(\mathrm{M}(B,A)\) to \(\mathrm{M}^{*}(B^{\perp},A^{\perp})\) \cite[Theorem~5(i)]{Schreiner66}, which by Mackey's theorem holds if and only if \(B^{\perp}+A^{\perp}\) is closed. The stated criterion follows, and it is symmetric in \(A\) and \(B\) because vector sums are. By a classical duality the two conditions are in fact equivalent to one another \cite[Theorem~IV.4.8]{Kato}, so that invertibility amounts to closedness of \(A+B\) alone; but the symmetry of the two-condition form is all we use.
\end{proof}

\begin{theorem}\label{T:HilbertModel}
	The 2-category \(\Suphil\) of Hilbert lattices is 2-z-exact, its normal 2-mono\-morphisms and its normal 2-epimorphisms are closed under composition, and it satisfies condition \(\DPN\); but it is not 2-di-exact.
\end{theorem}
\begin{proof}
	Everything except the failure of \textup{(DI2)} is \prox\ref{P:SupClass} combined with \prox\ref{P:HilbertSymmetric}. For the failure, let \(H\) be infinite-dimensional, take an orthonormal sequence \((e_{n})_{n\geq1}\), and consider the closed subspaces
	\[
		A=\overline{\operatorname{span}}\,\{e_{2n}\mid n\geq1\}\qquad\text{and}\qquad B=\overline{\operatorname{span}}\,\{b_{n}\mid n\geq1\}\text{,}\qquad b_{n}=e_{2n}+\tfrac{1}{n}\,e_{2n-1}\text{.}
	\]
	The vectors \(b_{n}\) are pairwise orthogonal, so \(B\) consists of the sums \(\sum_{n}c_{n}b_{n}\) with square-summable coefficients, and \(A\wedge B=0\), since a vector of \(B\) lying in \(A\) has all coefficients \(\tfrac{c_{n}}{n}\) of the \(e_{2n-1}\) zero. The subspace \(A+B\) contains every \(e_{2n}\) and every \(e_{2n-1}=n(b_{n}-e_{2n})\), so \(v\coloneq\sum_{n}\tfrac{1}{n}\,e_{2n-1}\) lies in \(A\vee B\); but \(v\notin A+B\), since \(v=a+\sum_{n}c_{n}b_{n}\) forces \(c_{n}=1\) for every \(n\), which is not square-summable. Now \(Z\coloneq B\vee\langle v\rangle=B+\langle v\rangle\) is closed, being the sum of a closed subspace and a line, and \(Z\wedge A=0\): if \(\sum_{n}c_{n}b_{n}+tv\in A\), then comparing coefficients of \(e_{2n-1}\) gives \(c_{n}=-t\) for every~\(n\), so \(t=0\), and the vector lies in \(B\wedge A=0\). Hence \((Z\wedge A)\vee B=B\), whereas \(Z\wedge(A\vee B)=Z\), which contains \(v\notin B\): the subspace \(Z\) witnesses the failure of \(\mathrm{M}^{*}(A,B)\). So the transposition at \((A,B)\) is not invertible by \lemx\ref{L:ModularPairs}, the antinormal morphism \(c_{A,B}\) of \(\Suphil\) is not normal by \prox\ref{P:SupAntinormal}, and \textup{(DI2)} fails; by \prox\ref{P:SupModular}, \(\Lat(H)\) is not modular either.
\end{proof}

\begin{corollary}\label{SnakeHilbertModel}
	The Snake Lemma holds in \(\Suphil\) in the form of \thex\ref{T:SnakeNonSelfDual} (while \thex\ref{Snake General 2D} does not apply to it).\noproof
\end{corollary}

\begin{remark}\label{Rem NSD Hilbert}
	Condition \textup{(DI2)} would ask every sum of closed subspaces to be closed, which is false, while \(\DPN\) asks only that the two sums attached to an antinormal pair and to its dinversion be closed together, which is Mackey's symmetry.

	In \(\Suphil\), \remx\ref{Rem Lattice Snake} reads almost word for word: the input of the Snake Lemma reduces to a normal 1-cell \(g\) together with closed subspaces \(U\), \(V\) such that \(g(U)\leq V\), the six terms of the snake sequence are again Hilbert lattices, and \(\partial\) may be taken to be \(g\) itself, restricted. The one difference is that the normality of the two outer verticals, which modularity provided for free, is now a genuine hypothesis---by \prox\ref{P:HilbertSymmetric} a pair of closed-sum conditions---and \remx\ref{Rem Lattice Sharp} shows that it cannot be dropped.
\end{remark}

\begin{remark}\label{Rem NSD Grandis}
	The introduction of \cite{Grandis-HA2} observes that the homological categories of Banach and of Hilbert spaces are not modular---their lattices of normal subobjects are lattices of closed subspaces, ``just orthomodular'' \cite[2.3.2]{Grandis-HA2}---so that their connected sequences of homology functors fail to be exact, and remarks that a deeper study of such settings could be of interest \cite[0.4]{Grandis-HA2}. \thex\ref{T:HilbertModel} and \corx\ref{SnakeHilbertModel} may be read as a step in that study, taken from the two-dimensional side: what a Snake Lemma asks of \(\Lat(H)\) is not the modular law but Mackey's symmetry, at the price of the non-self-dual pair of hypotheses. The projection lattices of von Neumann algebras stratify along the same line: for a finite factor the projection lattice is a continuous geometry \cite{vNeu60}, hence modular and within \(\Supmod\), while \(\Lat(H)\) is the projection lattice of the type \(\mathrm{I}\) factor \(\mathcal{B}(H)\); whether the projection lattice of an arbitrary von Neumann algebra is transposition-symmetric we do not know.
\end{remark}

The following proposition shows that the separation of \(\DPN\) from \textup{(DI2)} within \(\Sup\) is an intrinsically infinite phenomenon: no lattice of finite length and indeed no class of them can witness it.

\begin{proposition}\label{P:FiniteLength}
	A lattice of finite length which is transposition-symmetric is modular. Hence if every member of a class \(\mathcal{C}\) as in \prox\ref{P:SupClass} is of finite length, then \(\Sup_{\mathcal{C}}\) satisfies \(\DPN\) if and only if it satisfies \textup{(DI2)}.
\end{proposition}
\begin{proof}
	Transposition-symmetry is inherited by intervals, the transpositions of a pair computed in an interval agreeing with those computed in the ambient lattice, and it is invariant under dualisation: the transposition of \(L\op\) at \((a,b)\) is the right adjoint \(z\mapsto z\wedge b\colon[a,a\vee b]\to[a\wedge b,b]\) of the transposition of \(L\) at \((b,a)\), and an adjoint is invertible precisely when its partner is. Call \(L\) \dfn{semimodular} when \(a\succ a\wedge b\) implies \(a\vee b\succ b\) for all \(a\), \(b\), where \(x\succ y\) means that \(x\) covers \(y\); a lattice of finite length which is semimodular and dually semimodular is modular \cite[Chapter~II, Theorem~16]{Birkhoff}. By the invariance under dualisation it therefore suffices to prove that \(L\) is semimodular.

	Let \(a\succ a\wedge b\) and suppose, for a contradiction, that \(a\vee b\not\succ b\). We work in the interval \([a\wedge b,a\vee b]\) and write \(\bot\) and \(\top\) for its bounds; thus \(a\) is an atom there, \(a\wedge b=\bot\) and \(a\vee b=\top\), while \(\top\not\succ b\). The transposition at \((a,b)\) goes from the two-element chain \([\bot,a]\) to the interval \([b,\top]\), which has at least three elements, so it is not invertible; its unit is nevertheless trivial, since \((\bot\vee b)\wedge a=\bot\) and \((a\vee b)\wedge a=a\), so by the proof of \lemx\ref{L:ModularPairs} its counit fails: there is \(w\in[b,\top]\) with \((w\wedge a)\vee b<w\). Were \(w\wedge a=a\), then \(w\geq a\vee b=\top\) and the counit would hold at \(w\); so \(w\wedge a=\bot\), and \(b<w<\top\). Every \(t\) with \(w\leq t<\top\) satisfies \(t\wedge a=\bot\)---otherwise \(a\leq t\), whence \(t\geq a\vee w\geq a\vee b=\top\)---as well as \(t\vee a=\top\); since the length is finite, we may ascend from \(w\) to a coatom \(B\), with \(b<B\), \(B\wedge a=\bot\) and \(B\vee a=\top\). The transposition at \((B,a)\), from \([\bot,B]\) to \([a,\top]\), is not invertible, its unit failing at \(b\): \((b\vee a)\wedge B=\top\wedge B=B\neq b\). By transposition-symmetry, the transposition at \((a,B)\) is not invertible either. But it is the map \(\{\bot,a\}\to\{B,\top\}\) sending \(\bot\) to \(B\) and \(a\) to \(\top\), an isomorphism of two-element chains---a contradiction.
\end{proof}

\section*{Acknowledgements}
Many thanks to George Peschke, without whose relentless work towards the self-dual proof of the 1-dimensional Snake Lemma in~\cite{PVdL3} this article would not have existed.

\bibliographystyle{amsplain-nodash}
\bibliography{tim}

\providecommand{\noopsort}[1]{}
\providecommand{\bysame}{\leavevmode\hbox to3em{\hrulefill}\thinspace}
\providecommand{\MR}{\relax\ifhmode\unskip\space\fi MR }
\providecommand{\MRhref}[2]{%
  \href{http://www.ams.org/mathscinet-getitem?mr=#1}{#2}
}
\providecommand{\href}[2]{#2}
\begin{thebibliography}{10}

\bibitem{Afsa-24}
F.~Afsa, \emph{Di-exact categories and lattices of normal subobjects}, preprint
  {\texttt{arXiv:2411.18333}}, 2024.

\bibitem{Birkhoff}
G.~Birkhoff, \emph{Lattice theory}, third ed., Amer. Math. Soc. Colloq. Publ.,
  vol.~25, Amer. Math. Soc., 1967.

\bibitem{BirVNeu36}
G.~Birkhoff and J.~von Neumann, \emph{The logic of quantum mechanics}, Ann. of
  Math. (2) \textbf{37} (1936), 823--843.

\bibitem{Borceux-Bourn}
F.~Borceux and D.~Bourn, \emph{Mal'cev, protomodular, homological and
  semi-abelian categories}, Math. Appl., vol. 566, Kluwer Acad. Publ., 2004.

\bibitem{CavJanMes25}
E.~Caviglia, Z.~Janelidze, and L.~Mesiti, \emph{Fibrational approach to
  {G}randis exactness for 2-categories}, preprint {\texttt{arXiv:2504.01011}},
  2025.

\bibitem{CavJanMesRei26}
E.~Caviglia, Z.~Janelidze, L.~Mesiti, and {\"U}.~Reimaa, \emph{Exactness of the
  2-categories of abelian and triangulated categories}, preprint
  {\texttt{arXiv:2609.06278}}, 2026.

\bibitem{CleMonRod23}
M.~M. Clementino, A.~Montoli, and D.~Rodelo, \emph{On lax protomodularity for
  {O}rd-enriched categories}, J.~Pure Appl. Algebra \textbf{227} (2023), Paper
  No. 107348.

\bibitem{CleRod24}
M.~M. Clementino and D.~Rodelo, \emph{Enriched aspects of calculus of relations
  and 2-permutability}, J. Algebra Appl. (2026), Paper No. 2650233.

\bibitem{CleRod26}
M.~M. Clementino and D.~Rodelo, \emph{A note on varieties of ordered algebras},
  J.~Pure Appl. Algebra \textbf{230} (2026), Paper No. 108363.

\bibitem{Dup08}
M.~Dupont, \emph{Abelian categories in dimension 2}, Ph.D. thesis,
  Universit\'{e} catholique de Louvain, 2008.

\bibitem{DupVit03}
M.~Dupont and E.~M. Vitale, \emph{Proper factorization systems in
  2-categories}, J.~Pure Appl. Algebra \textbf{179} (2003), no.~1--2, 65--86.

\bibitem{Gabriel62}
P.~Gabriel, \emph{Des cat\'egories ab\'eliennes}, Bull. Soc. Math. France
  \textbf{90} (1962), 323--448.

\bibitem{Grandis-HA1}
M.~Grandis, \emph{Homological algebra: The interplay of homology with
  distributive lattices and orthodox semigroups}, World Scientific, Singapore,
  2012.

\bibitem{Grandis-HA2}
M.~Grandis, \emph{Homological algebra: In strongly non-abelian settings}, World
  Scientific, Singapore, 2013.

\bibitem{EGAIV2}
A.~Grothendieck and J.~Dieudonn{\'e}, \emph{{\'E}l{\'e}ments de
  g{\'e}om{\'e}trie alg{\'e}brique: {IV}. {{\'E}}tude locale des sch{\'e}mas et
  des morphismes de sch{\'e}mas, {S}econde partie}, Publ. Math. Inst. Hautes
  {\'E}tudes Sci. \textbf{24} (1965), 5--231.

\bibitem{Janelidze-Marki-Tholen}
G.~Janelidze, L.~M{\'a}rki, and W.~Tholen, \emph{Semi-abelian categories},
  J.~Pure Appl. Algebra \textbf{168} (2002), no.~2--3, 367--386.

\bibitem{Kato}
T.~Kato, \emph{Perturbation theory for linear operators}, Classics in
  Mathematics, Springer, 1995, Reprint of the 1980 edition.

\bibitem{Lac10}
S.~Lack, \emph{A 2-categories companion}, Towards Higher Categories (J.~C. Baez
  and J.~P. May, eds.), IMA Vol. Math. Appl., vol. 152, Springer, 2010,
  pp.~105--191.

\bibitem{Mackey45}
G.~W. Mackey, \emph{On infinite-dimensional linear spaces}, Trans. Amer. Math.
  Soc. \textbf{57} (1945), 155--207.

\bibitem{MaeMae70}
F.~Maeda and S.~Maeda, \emph{Theory of symmetric lattices}, Grundlehren math.
  Wiss., vol. 173, Springer, 1970.

\bibitem{Mat89}
H.~Matsumura, \emph{Commutative ring theory}, second ed., Cambridge Studies in
  Advanced Mathematics, vol.~8, Cambridge University Press, 1989.

\bibitem{PVdL3}
G.~Peschke and T.~Van~der Linden, \emph{A homological view of categorical
  algebra}, preprint {\texttt{arXiv:2404.15896v2}}, 2024.

\bibitem{Schreiner66}
E.~A. Schreiner, \emph{Modular pairs in orthomodular lattices}, Pacific J.
  Math. \textbf{19} (1966), no.~3, 519--528.

\bibitem{Str82b}
R.~Street, \emph{Characterization of bicategories of stacks}, Category Theory
  (Gummersbach, 1981) (K.~H. Kamps, D.~Pumpl\"{u}n, and W.~Tholen, eds.),
  Lecture Notes in Math., vol. 962, Springer, 1982, pp.~282--291.

\bibitem{Stacks}
{The {S}tacks {P}roject Authors}, \emph{The {S}tacks {P}roject}, 2025,
  \url{https://stacks.math.columbia.edu}.

\bibitem{vNeu60}
J.~von Neumann, \emph{Continuous geometry}, Princeton Math. Series, vol.~25,
  Princeton University Press, 1960.

\end{thebibliography}

\end{document}